\documentclass[times]{amsart}

\usepackage{amsmath,amsopn,amssymb,amsthm,mathtools,mathrsfs,yhmath,tensor,graphicx,geometry,stmaryrd,tabulary,booktabs,leftindex}
\usepackage[euler]{textgreek}

\usepackage{xcolor}

	\definecolor{UCIB}{HTML}{0064A4}
	\definecolor{UCSDB}{HTML}{00629B}

\usepackage{accents}
	\newcommand{\ubar}[1]{\underaccent{\bar}{#1}}
	
\usepackage{upgreek,textgreek}
	
\usepackage{hyperref}
\hypersetup{backref, pdfpagemode=FullScreen, urlcolor = UCIB, citecolor=blue, colorlinks=true}

\newcommand{\Simplehref}[3][blue]{\href{#2}{\color{#1}{#3}}}%

\usepackage{ifthen}

\usepackage{lipsum} 

\usepackage{tikz}
\usetikzlibrary{calc}
\usetikzlibrary{intersections}
\usetikzlibrary{graphs,graphs.standard,quotes}

\usepackage{enumitem}

\usepackage{tikz}

\usepackage{pgfplots}
\pgfplotsset{compat=1.17}
\usetikzlibrary{quotes,angles}

\usepackage{booktabs}

\usepackage{xcolor}
\theoremstyle{plain}
\newtheorem{fthm}{Theorem}[section]
\newtheorem*{fthm*}{Theorem}
\newtheorem{flemma}{Lemma}[section]
\newtheorem*{flemma*}{Lemma}
\newtheorem{fprop}{Proposition}[section]
\newtheorem*{fprop*}{Proposition}
\newtheorem{fcor}{Corollary}[section]
\newtheorem*{fcor*}{Corollary}

\theoremstyle{definition}
\newtheorem{fdefi}{Definition}[section]
\newtheorem*{fdefi*}{Definition}

\newtheorem*{fexmp*}{Example}

\theoremstyle{remark}
\newtheorem{frmk}{Remark}[section]
\newtheorem*{frmk*}{Remark}
\newtheorem{fconj}{Conjecture}[section]
\newtheorem*{fconj*}{Conjecture}

\newtheorem*{fclaim*}{Claim}

\newtheorem*{fquest*}{Question}

\newcommand{\Address}{{
  \bigskip
  \footnotesize

  Chao-Ming Lin, \textsc{Department of Mathematics, University of California-Irvine, CA}\par\nopagebreak
  \textit{E-mail address:} \href{chaominl@uci.edu}{chaominl@uci.edu}\par\nopagebreak
  \textit{Personal Website:} \href{https://chaominl.github.io}{https://chaominl.github.io}

}}

\newcommand{\Acknow}{{
  \bigskip
\textbf{Acknowledgements:} The author is grateful to Zhiqin~Lu and Xiangwen~Zhang for giving me enlightening helps. The author would like to thank Tristan~Collins for his interest in this work. The author is grateful to \Simplehref[UCSDB]{http://www.symbol.codes}{Hsin-Po~Wang} for several helpful conversations and some programming supports.

}}

\DeclareMathOperator\osc{osc}

\DeclareMathOperator\dist{dist}

\DeclareMathOperator\cs{$\csc^2\bigl(\hat{\theta}\bigr)$}

\def\cs{\csc^2 ( \hat{\theta}  )}
\def\s{\sec^2 ( \hat{\theta}  )}

\def\ct{\cot ( \hat{\theta}  )}
\def\ta{\tan ( \hat{\theta}  )}
\def\taa{\tan^2 ( \hat{\theta}  )}
\def\ctt{\cot^2 ( \hat{\theta}  )}

\def\cc{\bigl ( 3 \csc^2  ( \hat{\theta}  ) -4  \bigr) }

\def\sii{\sin^2 ( \hat{\theta}  )}
\def\css{\cos^2 ( \hat{\theta}  )}
\def\({\bigl (}
\def\){\bigr )}

\def\coo{ \cos ( \hat{\theta}  ) }
\def\cso{\csc ( \hat{\theta}  )}
\def\sco{\sec ( \hat{\theta}  )}

\begin{document}

\title{The Deformed Hermitian--Yang--Mills  Equation, the Positivstellensatz, and the Solvability}
\author{Chao-Ming Lin}

\begin{abstract}
Let $(M, \omega)$ be a compact connected Kähler manifold of complex dimension four and let $[\chi] \in H^{1,1}(M; \mathbb{R})$. We confirmed the conjecture by Collins--Jacob--Yau \cite{collins20151} of the solvability of the deformed Hermitian--Yang--Mills equation, which is given by the following nonlinear elliptic equation $\sum_{i} \arctan (\lambda_i) = \hat{\theta}$, where $\lambda_i$ are the eigenvalues of $\chi$ with respect to $\omega$ and $\hat{\theta}$ is a topological constant. 
This conjecture was stated in \cite{collins20151}, wherein they proved that the existence of a supercritical $C$-subsolution or the existence of a $C$-suboslution when $\hat{\theta} \in \bigl[   ( (n-2) + {2}/{n}     )  {\pi}/{2},  n\pi/2  \bigr)$ will give the solvability of the deformed Hermitian--Yang--Mills equation. Collins--Jacob--Yau conjectured that their existence theorem can be improved when $\hat{\theta} \in \bigl(  (n-2 ) {\pi}/{2},  ( (n-2) + {2}/{n}     )  {\pi}/{2}  \bigr)$, where $n$ is the complex dimension of the manifold. In this paper, we confirmed their conjecture that when the complex dimension equals four and $\hat{\theta}$ is close to the supercritical phase $\pi$ from the right, then the existence of a $C$-subsolution implies the solvability of the deformed Hermitian--Yang--Mills equation.  
\end{abstract}
\maketitle
\vspace{-0.6cm}



\section{Introduction}
\label{sec:1}

Motivated by mirror symmetry in string theory, the \emph{deformed Hermitian--Yang--Mills equation}, which will be abbreviated as dHYM equation later on, was discovered around the same time by Mariño--Minasian--Moore--Strominger \cite{marino2000nonlinear} and Leung--Yau--Zaslow \cite{leung2000special} using different points of view. Mariño--Minasian--Moore--Strominger \cite{marino2000nonlinear} found out that the dHYM equation is the requirement for a $D$-brane on the $B$-model of mirror symmetry to be supersymmetric. It was shown by Leung--Yau--Zaslow \cite{leung2000special} that, in the semi-flat model of mirror symmetry, solutions of the dHYM equation are related via the Fourier--Mukai transform to special Lagrangian submanifolds of the mirror. Thus it is crucial to understand under which condition does the dHYM equation has a solution.\bigskip 

Let $(M, \omega)$ be a compact connected Kähler manifold of complex dimension $n$ and $[\chi_0] \in H^{1,1}(M; \mathbb{R})$, where $H^{1,1}(M; \mathbb{R})$ is the $(1, 1)$-Dolbeault cohomology group. The study of the dHYM equation for a holomorphic line bundle over a compact Kähler manifold was initiated by Jacob--Yau \cite{jacob2017special}; they introduced the following problem: if $\omega$ is a Kähler form, does there exists a real smooth, closed $(1,1)$-form $\chi \in [\chi_0]$ such that, 
\begin{align}
\label{eq:1.1}
\Im \bigl ( \omega + \sqrt{-1} \chi \bigr )^n &= \tan \bigl (  \hat{\theta}   \bigr ) \cdot \Re \bigl ( \omega + \sqrt{-1} \chi \bigr )^n? \tag{1.1}
\end{align}Here $\Im$ and $\Re$ are the imaginary and real parts, respectively, and $\hat{\theta}$ is a topological constant determined by the cohomology classes $[\omega]$ and $[\chi_0]$. The above equation is called the deformed Hermitian--Yang--Mills equation. In the supercritical phase case, which means that the phase $\hat{\theta}$ satisfies $\hat{\theta} >  {(n-2)} \pi/2$, Collins--Jacob--Yau \cite{collins20151} showed that if there exists a supercritical $C$-subsolution (which will be introduced and defined later in Section~\ref{sec:2.3}), then the dHYM equation is solvable. Collins--Jacob--Yau conjectured that the existence of the solution to the dHYM equation (\ref{eq:1.1}) is equivalent to a certain stability condition for any analytic subvarieties and confirmed this numerical conjecture for complex surfaces. Recently, Chen \cite{chen2021j} proved a Nakai--Moishezon type criterion for the supercritical dHYM equation (and also for the $J$-equation) under a slightly stronger condition that these holomorphic intersection numbers have a uniform lower bound independent of analytic subvarieties. Datar--Pingali \cite{datar2021numerical} extended the techniques in Chen \cite{chen2021j} to the case of the generalized Monge--Ampère equation. As a result, Datar--Pingali confirmed this numerical conjecture of the dHYM equation when the complex dimension equals three and $\hat{\theta} \in [\pi, 3\pi/2)$ (also projective and positive line bundles for the $J$-equation). Jacob--Sheu \cite{jacob2020deformed} showed that the numerical conjecture holds on the blowup of complex projective space. Chu--Lee--Takahashi \cite{chu2021nakai} improved the result without assuming a uniform lower bound for these intersection numbers. The method by Chu--Lee--Takahashi was inspired by a remarkable work of Song \cite{song2020nakai} in the study of the $J$-equation. Song extended the method of Chen \cite{chen2021j} and showed a Nakai--Moishezon type criterion without assuming a uniform lower bound, which confirms the conjecture by Lejmi--Székelyhidi \cite{lejmi2015j}. We should emphasize that there are many significant works which have been done recently. The interested reader is referred to \cite{collins2020stability, collins2018moment, collins2018deformed, jacob2019weak, schlitzer2021deformed} and the references therein.\bigskip


If we write the dHYM equation (\ref{eq:1.1}) in terms of the eigenvalues of the Hermitian endomorphism $\Lambda = \omega^{-1}\chi$, then we can rewrite equation (\ref{eq:1.1}) as
\begin{align*}
\label{eq:1.2}
\Theta_{\omega}(\chi) \coloneqq \sum_{i = 1}^n \arctan \lambda_i = \hat{\theta}, \tag{1.2}
\end{align*}where $\lambda_i$ are the eigenvalues of $\Lambda$ and we specify the branch so that $\lambda_i \in (-\pi/2, \pi/2)$ for all $i \in \{1, \cdots, n\}$. In Collins--Jacob--Yau \cite{collins20151}, they proved the following. 
\hypertarget{T:1.1}{\begin{fthm}[Collins--Jacob--Yau \cite{collins20151}]}
Suppose that the topological constant $\hat{\theta}$ satisfies the supercritical phase condition
$\hat{\theta} > (n-2)  {\pi}/{2}$ and there exists a $C$-subsolution $\ubar{\chi} \coloneqq \chi_0 + \sqrt{-1} \partial \bar{\partial} \ubar{u}$ (in the sense of Definition~\hyperlink{D:2.2}{2.2} introduced by Székelyhidi \cite{szekelyhidi2018fully}). Assume either
\begin{itemize}
\item $\Theta_\omega (\ubar{\chi}) > (n-2)\pi/2$ (which will be abbreviated as supercritical $C$-subsolution), or
\item $\hat{\theta} \geq \bigl( (n-2) + 2/n \bigr) \pi/2$.
\end{itemize}
Then there exists a unique smooth $(1,1)$-form $\chi \in [\chi_0]$ solving the dHYM equation
\begin{align*}
\label{eq:1.3}
\Theta_\omega \left ( \chi \right ) = \hat{\theta}. \tag{1.3}
\end{align*}
\end{fthm}

\hypertarget{R:1.1}{\begin{frmk}[Collins--Jacob--Yau \cite{collins20151}]}
If $\hat{\theta} \geq \bigl( n-2 +  {2}/{n} \bigr) {\pi}/{2}$, then any $C$-subsolution becomes supercritical, that is, $\Theta_\omega ( \bullet ) > (n-2)\pi/2$. So if there exists a $C$-subsolution, then there exists a unique smooth $(1,1)$-form $\chi \in [\chi_0]$ solving the dHYM equation (\ref{eq:1.3}).
\end{frmk}From the above remark, Collins--Jacob--Yau conjectured that their existence result can be improved when the topological constant $\hat{\theta}$ is close to the supercritical phase.
\hypertarget{C:1.1}{\begin{fconj}[Collins--Jacob--Yau \cite{collins20151}]}
Theorem~\hyperlink{T:1.1}{1.1} can be improved when  
\begin{align*}
\label{eq:1.4}
\hat{\theta} \in \Bigl( \bigl(n-2\bigr) {\pi}/{2}, \bigl( (n-2) + {2}/{n}    \bigr)  {\pi}/{2}  \Bigr). \tag{1.4}
\end{align*}
\end{fconj}

In this work, we will focus on this range (\ref{eq:1.4}). We always assume that $\hat{\theta} \in \bigl(  (n-2 ) {\pi}/{2},  ( (n-2) + {2}/{n}     )  {\pi}/{2}  \bigr)$ for the rest of this paper. We confirm this Conjecture~\hyperlink{C:1.1}{1.1} when the complex dimension equals three or four. The main purpose of this work is to prove that the existence of a $C$-subsolution (not necessarily a supercritical $C$-subsolution) will lead us to the solvability of the dHYM equation (\ref{eq:1.1}), when the topological constant $\hat{\theta}$ is close to the supercritical phase. When complex dimension equals three, Pingali \cite{pingali2019deformed} proved that if there exists a $C$-subsolution, then the dHYM equation is solvable. He had a nice observation finding a continuity path connecting the dHYM equation to the complex Hessian equation and he showed that the $C$-subsolution condition will be maintained. Thus if one has a priori estimates along the continuity path, then by the result of Fang--Lai--Ma \cite{fang2011class} on complex Hessian equation, the dHYM equation is solvable. We use some real algebraic geometry techniques to approach this Conjecture~\hyperlink{C:1.1}{1.1}. Moreover, we will show the solvability of the dHYM equation (\ref{eq:1.1}) when the complex dimension equals four. \bigskip

\begin{figure}
\centering
\begin{tikzpicture}[scale=0.7]

\path[draw,name path=border1] (0,0) to[out=-10,in=150] (6,-2);
\path[draw,name path=border2] (12,1) to[out=150,in=-10] (5.5,2.6);
\draw[draw,name path=line1] (6,-2) -- (12,1);
\path[draw,name path=line2] (5.5,2.6) -- (0,0);
\shade[left color=gray!10,right color=gray!70] 
  (0,0) to[out=-10,in=150] (6,-2) -- 
  (12,1) to[out=150,in=-10] (5.5,2.6) -- cycle;

\draw[->,black] (1.7,3.4) -- (2.9,3.8);
\draw[->,black] (2,3.2) -- (2,4.9);
\draw[color=red,variable=\t,domain=-pi/2+1.04:pi/2-0.24,samples=100]    plot ({2 + 9*tan(\t r)/40  },{ 3.5 + 7*tan( (pi/4 - \t) r)/20 + 3*tan(\t r)/40 });
\filldraw [red!50,opacity=0.3] (2.25,3.23) -- (71/40 +1.2,123/40 +0.4) --  (71/40 +1.2,123/40+0.4*2+1.4) -- (71/40 ,123/40+0.4+1.4) -- (71/40,3.74) --cycle ;
\draw[dashed,thin] (2.25,0.8) -- (2.25,3.2);
\filldraw (2.25,0.8) circle (1pt) node[anchor=north,font = \tiny] {$p_1$};

\pgfsetplottension{0.8}
\pgfsetlinewidth{0.8pt}
\pgfplothandlercurveto
\pgfplotstreamstart
\pgfplotstreampoint{\pgfpoint{2cm}{3.3cm}} 
\pgfplotstreampoint{\pgfpoint{6.5cm}{4.7cm}} 
\pgfplotstreampoint{\pgfpoint{9.3cm}{2.8cm}} 
\pgfplotstreamend
\color{green!60}
\pgfusepath{stroke}

\pgfsetplottension{0.75}
\pgfsetlinewidth{0.8pt}
\pgfplothandlercurveto
\pgfplotstreamstart
\pgfplotstreampoint{\pgfpoint{1.8cm}{4.2cm}} 
\pgfplotstreampoint{\pgfpoint{6.1cm}{4.9cm}} 
\pgfplotstreampoint{\pgfpoint{9cm}{2.5cm}} 
\pgfplotstreamend
\color{red!60}
\pgfusepath{stroke}

\color{black}
\pgfsetlinewidth{0.4pt}

\draw[black!60] (2.25,0.8) .. 	controls (4.4,1.4) .. (6.25,1) 
				..	controls (7.8,0.4) ..  (8.75,0.6);

\draw[->,black] (5.7,3.9) -- (6.9,4.3);
\draw[->,black] (6,3.7) -- (6,5.4);
\draw[color=red,variable=\t,domain=0.3:pi/2-0.24,samples=100]    plot ({6 + 9*tan(\t r)/40  },{ 4 + 7*tan( (pi/2 - \t) r)/20 + 3*tan(\t r)/40 });
\filldraw [red!50,opacity=0.3] (6,4) -- (6 +1.2*0.8,4 +0.4*0.8) --  (6 +1.2*0.8,4+0.4*2*0.8+1.4*0.7) -- (6 ,4+0.4*0.8+1.4*0.7) -- cycle ;
\draw[dashed,thin] (6.25,1) -- (6.25,4);
\filldraw (6.25,1) circle (1pt) node[anchor=north,font = \tiny] {$p_2$};

\draw[->,black] (8.2,2.4) -- (9.4,2.8);
\draw[->,black] (8.5,2.2) -- (8.5,3.9);
\draw[color=red,variable=\t,domain=-pi/2+1.2:pi/2-0.24,samples=100]    plot ({8.5 + 9*tan(\t r)/40  },{ 2.5 + 7*tan( (3*pi/10 - \t) r)/20 + 3*tan(\t r)/40 });
\filldraw [red!50,opacity=0.3] {(8.65,2.31) -- (8.5-9*0.73/40 + 1.2,2.5-3*0.73/40-7*0.73/20 + 0.4)-- (8.5-9*0.73/40 +1.2,2.5-3*0.73/40-7*0.73/20 + 0.4*2 +1.4*0.85) -- (8.5-9*0.73/40,2.5-3*0.73/40-7*0.73/20+0.4+1.4*0.85) -- (8.5-9*0.73/40,2.7) } --cycle ;
\draw[dashed,thin] (8.75,0.6) -- (8.75,2.2);
\filldraw (8.75,0.6) circle (1pt) node[anchor=north,font = \tiny] {$p_3$};

\filldraw[red!60!black] (1.92,4.25) circle (1pt) ; 
\filldraw[red!60!black] (6.1,4.9) circle (1pt) ; 
\filldraw[red!60!black] (8.94,2.58) circle (1pt) ; 

\filldraw[green!60!black] (2.3,3.44) circle (1pt) ; 
\filldraw[green!60!black] (6.4,4.71) circle (1pt) ; 
\filldraw[green!60!black] (9.16,2.98) circle (1pt) ; 

\end{tikzpicture}
\caption{The solution and a supercritical $C$-subsolution}
\label{fig:1.1}
\end{figure}
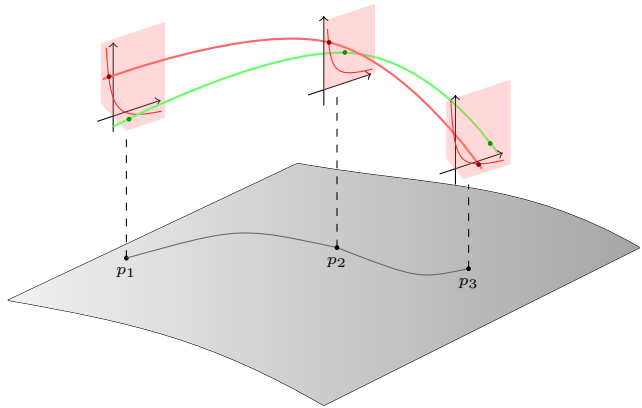


We will sketch the ideas and proofs here. First, we start with the classical result. In 1978, by studying the complex Monge--Ampère equation, Shing-Tung~Yau \cite{yau1978ricci} resolved the Calabi conjecture \cite{calabi1954kahler, calabi1957kahler}, which had been posed by Eugenio~Calabi in 1954. This celebrated method by Yau is well-known nowadays, which is called the continuity method. The idea is to find a path connecting the unsolved equation to a well-understood solvable equation. And on this path, if we have a priori estimates, then we can apply the Arzelà--Ascoli theorem to get a smooth solution to the equation. The continuity method is widely used to solve many famous equations now, for example the complex Hessian equation, the $J$-equation, the general inverse $\sigma_k$ equation, etc. Since we mainly focus on the dHYM equation in this work and because the space limitations, the interested reader is referred to \cite{chen2021j, collins2017convergence, fang2013convergence, fang2011class, lejmi2015j} and the references therein. To get a priori estimates for these equations, Székelyhidi \cite{szekelyhidi2018fully} introduced a notion of $C$-subsolution (see Definition~\hyperlink{D:2.2}{2.2}), the existence of a $C$-subsolution will lead us to the solvability of many equations, for example, the complex Monge--Ampère equation, the complex Hessian equation, the $J$-equation, the general inverse $\sigma_k$ equation, etc. But for the dHYM equation, till now, if a $C$-subsolution is not strong enough, then we cannot derive nice a priori estimates from this weak $C$-subsolution (see Collins--Jacob--Yau \cite{collins20151} for the Kähler case and Lin \cite{lin2020} for the Hermitian case). This is the reason why Collins--Jacob--Yau assumed the $C$-subsolution to be supercritical in Theorem~\hyperlink{T:1.1}{1.1}; see above Figure~\ref{fig:1.1} for a graphical illustration. In Figure~\ref{fig:1.1}, first, the red curve on each plane attached to a point is the solution set of the dHYM equation at this point, the section passing through the red curves is the solution to the dHYM equation. The pink shaded region attached to each point, is the strong enough $C$-subsolutions region for the complex dimension two dHYM equation, that we previously knew can provide us a priori estimates. Last, the green section is a supercritical $C$-subsolution on the manifold. \bigskip

In this work, for complex dimension three and four, even when a $C$-subsolution is not supercritical, we can still find a special continuity path and a priori estimates such that the dHYM equation (\ref{eq:1.1}) is solvable. Let us state some of our settings and results now.\bigskip


If the complex dimension $n = 4$, then we are interested in the case $\hat{\theta} \in   (   \pi, 5\pi/4     )$. By specifying the branch, the dHYM equation (\ref{eq:1.1}) becomes 
\begin{align*}
\label{eq:1.5}
\ct \cdot \Im \bigl (  \omega + \sqrt{-1} \chi \bigr )^4 &=   \Re \bigl ( \omega + \sqrt{-1} \chi \bigr )^4, \tag{1.5}
\end{align*}that is, 
\begin{align*}
\label{eq:1.6}
\chi^4 + 4 \ct \omega \wedge \chi^3 - 6 \omega^2 \wedge \chi^2 -4 \ct \omega^3 \wedge \chi + \omega^4 = 0. \tag{1.6}
\end{align*}

By doing the following substitution: $X \coloneqq \chi + \ct \omega$, we get
\begin{align*}
\label{eq:1.7}
\sii X^4 - 6   \omega^2 \wedge X^2 + 8 \ct   \omega^3 \wedge X - \cc   \omega^4 = 0. \tag{1.7}
\end{align*} By our assumption that $\hat{\theta} \in   (   \pi, 5\pi/4     )$, thus $\cot   (   \hat{\theta}    ) < 0$. Otherwise if $\cot   (   \hat{\theta}    ) \geq 0$ and $3 \cs -4 \geq 0$, then equation (\ref{eq:1.7}) becomes a general inverse $\sigma_k$ equation with non-negative coefficients. This is solvable by the result of Collins--Székelyhidi \cite{collins2017convergence}. Now, we consider the following continuity path:
\begin{align*}
\label{eq:1.8}
\sii X^4 - 6 c_2(t)   \omega^2 \wedge X^2 + 8 c_1(t) \ct   \omega^3 \wedge X - c_0(t) \cc   \omega^4 = 0. \tag{1.8}
\end{align*}

Suppose the triple $\bigl(c_2(t), c_1(t), c_0(t)\bigr)$ satisfies the following \hypertarget{cons}{constraints} for all $t \in [0, 1]$:

\begin{enumerate}[leftmargin=4.5cm]
	\setlength\itemsep{+0.4em}
\item[Topological constraint:] $\sii \Omega_0 - 6c_2(t) \Omega_2 + 8c_1(t) \ct \Omega_3 - c_0(t) \cc   \Omega_4 =0$.
\item[Boundary constraints:] $c_2(1) = c_1(1) = c_0(1) = 1;\quad c_2(0)>0; \quad c_1(0) = 0$.
\item[Positivstellensatz constraint:] $\cc c_0(t) >  -24   {c_2^2(t) \cs \cos^2(\theta_{c_1, c_2})\cos(2\theta_{c_1, c_2})} $.
\item[$\Upsilon$-cone constraints:] $\frac{d}{dt} \bigl( c_2^{3/2}(t) \bigr) \geq -\cos(\hat{\theta}) c_1'(t); \quad c_2'(t) > 0; \quad c_1'(t) > 0$.
\end{enumerate}Here we denote $\Omega_i \coloneqq \int_M \omega^i \wedge X^{4-i}$ and we define $\theta_{c_1, c_2} \coloneqq  \arccos \bigl(   { c_1(t) \coo}/{c_2^{3/2}(t)}  \bigr)\big/ 3 -  {2\pi}/{3}$, where we specify the branch so that $\arccos \bigl(   {c_1(t) \cos(\hat{\theta})}/{c_2^{3/2}(t)}  \bigr) \in   \bigl( \pi, 3\pi/2 \bigr]$. Then we have the following a priori estimates.

\hypertarget{T:1.2}{\begin{fthm}[A Priori Estimates]}
Suppose $X$ is a $C$-subsolution to equation (\ref{eq:1.7}) and $u \colon M  \rightarrow \mathbb{R}$ is a solution to equation (\ref{eq:1.7}), then for every $\alpha \in (0, 1)$, we have
\begin{align*}
\| \partial \bar{\partial} u \|_{C^{2, \alpha(M)}} \leq C \bigl(M, X, \omega, \alpha, \hat{\theta}, c_0, c_1, c_2  \bigr ).
\end{align*}
\end{fthm}

\hypertarget{R:1.2}{\begin{frmk}}
Here, let us emphasize that this $C$-subsolution need not to be supercritical. 
\end{frmk}

With these a priori estimates, we can indeed find a triple $\bigl(c_2(t), c_1(t), c_0(t)\bigr)$ satisfies these \hyperlink{cons}{4-dimensional four constraints}, that is the following. 

\hypertarget{T:1.3}{\begin{fthm}} If $ ( \Omega_2, \Omega_3, \Omega_4   ) \in \Omega^{4, \hat{\theta}}_{\ell}$, then the following triple will satisfy all \hyperlink{cons}{4-dimensional four constraints}
\begin{align*}
c_{2, \ell} (t) \coloneqq \bigl[ 1 + (1-t) \ell   \cos (\hat{\theta} )   \bigr ]^{2/3};\  c_1(t) \coloneqq t; \  c_{0, \ell}(t) \coloneqq \frac{ \sii  \Omega_0 -  6c_{2, \ell}(t)      \Omega_2 + 8c_1(t) \ct     \Omega_3}{\cc    \Omega_4}.
\end{align*}Here $\ell \in \bigl [ 1, -\sco \bigr)$ is a constant, $\Omega_i \coloneqq \int_M \omega^i \wedge X^{4-i}$, and  
\begin{align*}
\label{eq:1.9}
\Omega^{4, \hat{\theta}}_{\ell} \coloneqq \Bigl \{     \Omega_3 <   \inf_{t \in [0, 1) } \frac{3(1- c_{2, \ell} (t))\ta}{4(1-t)} \Omega_2 + \frac{  \tilde{c}_{0, \ell} (t) \ta}{8(1-t)} \Omega_4  \Bigr \}, \tag{1.9} 
\end{align*}where we define $\tilde{c}_{0, \ell} (t) \coloneqq 3 \cs -4 +  24   {c_{2, \ell}^2 \cs \cos^2(\theta_{c_1, c_{2, \ell}})\cos(2\theta_{c_1, c_{2, \ell}})}$. Moreover, $\tilde{c}_{0, \ell}(t)$ is a positive function.
\end{fthm}



\begin{figure}
	\centering
	\includegraphics[width=0.8\textwidth]{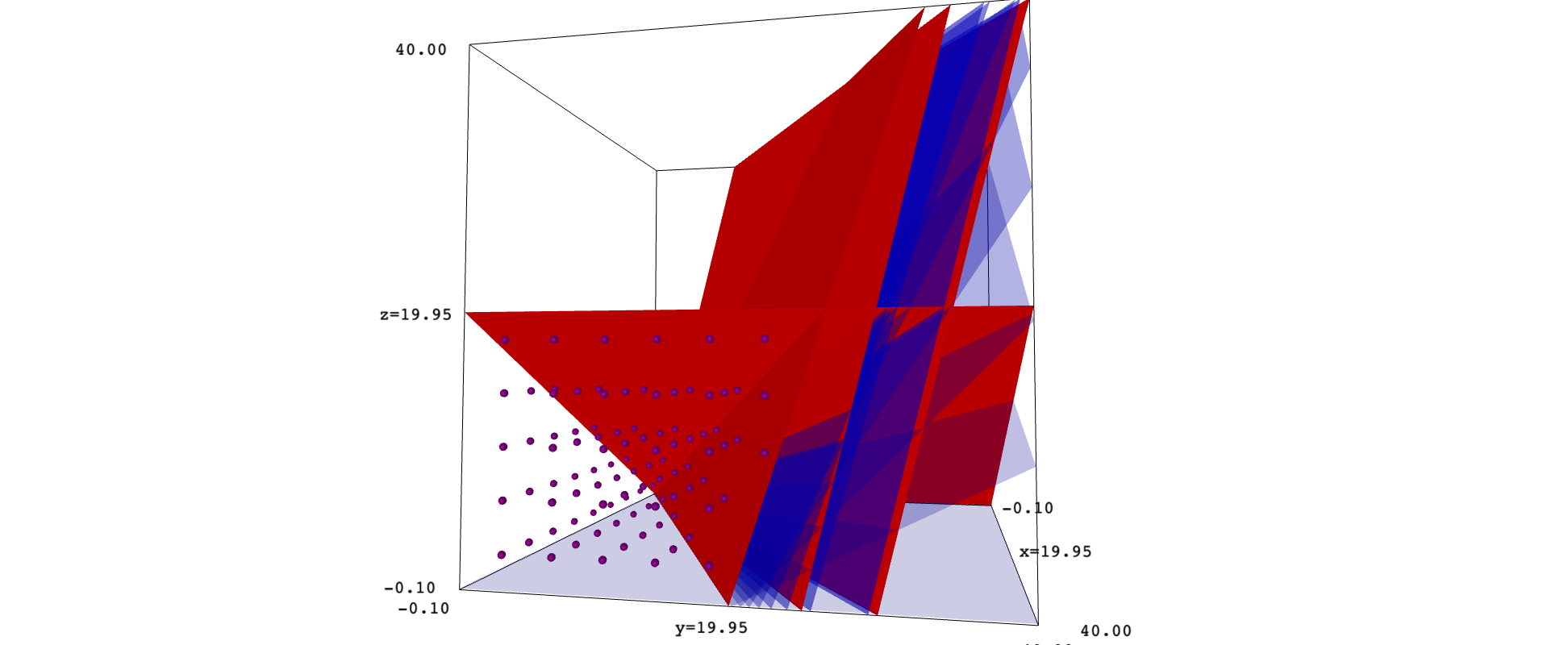}
	\caption{$C$-Subsolution constraints and the numerical constraint}
	\label{fig:1.2}
\end{figure}

In conclusion, as long as $ ( \Omega_2, \Omega_3, \Omega_4   ) \in \Omega^{4, \hat{\theta}}_{\ell}$, then we can apply Theorem~\hyperlink{T:1.2}{1.2}. So the complex dimension four dHYM equation (\ref{eq:1.7}) is solvable if the following equation is solvable.
\begin{align*}
\label{eq:1.10}
\sii X^4 - 6 c_{2, \ell}(0)   \omega^2 \wedge X^2  - c_{0, \ell}(0) \cc   \omega^4 = 0. \tag{1.10}
\end{align*}Equation (\ref{eq:1.10}) is solvable if $c_{0, \ell}(0) \geq 0$ due to Collins--Székelyhidi \cite{collins2017convergence}. A natural way to make $c_{0, \ell}(0) \geq 0$ is to let $c_{2, \ell}(0)$ positive but sufficiently small, so $\sii \Omega_0 - 6 c_{2, \ell}(0) \Omega_2$ becomes non-negative. This can be done by letting $\ell$ sufficiently close to $-\sco$, but the issue here is that the set $\Omega^{4, \hat{\theta}}_{\ell}$ will also be deformed. Fortunately, we have the following.

\hypertarget{T:1.4}{\begin{fthm}}
There exists an $\ell$ sufficiently close to $-\sco$, such that if there exists a $C$-subsolution $X$ to equation (\ref{eq:1.7}), then
\begin{align*}
(\Omega_2, \Omega_3, \Omega_4 ) \in \Omega^{4, \hat{\theta}}_{{\ell}}
\end{align*}
and the four dimensional dHYM equation (\ref{eq:1.7}) is solvable. Here $\Omega_i \coloneqq \int_M \omega^i \wedge X^{4-i}$.
\end{fthm}
This result confirms the \hyperlink{C:1.1}{Conjecture} by Collins--Jacob--Yau when the complex dimension equals four.

\begin{frmk}
Above Figure~\ref{fig:1.2} is a numerical graph of $\Omega^{4, \hat{\theta}}_{\ell}$ for $\ell$ sufficiently close to $-\sco$, the blue hyperplanes are the hyperplanes 
\begin{align*}
{6(1- c_{2, \ell}(t))} \Omega_2 - 8(1-c_1(t)) \ct \Omega_3 +  {  \tilde{c}_{0, \ell}(t)  }  \Omega_4 = 0
\end{align*}for $t \in [0, 1)$. So if the triple $( \Omega_2, \Omega_3, \Omega_4   )$ lies above these hyperplanes, then the complex four dimensional dHYM equation (\ref{eq:1.7}) is solvable if it is solvable when $t = 0$. On the other hand, the existence of $C$-subsolution will also give us some constraints. The red hyperplanes are the constraints generated by the $C$-subsolution condition we know so far; there might be more constraints. From Figure~\ref{fig:1.2}, if a $C$-subsolution exists, then it must lie in the dotted region. So it must also satisfy the numerical inequality (\ref{eq:1.9}), which implies that the complex four dimensional dHYM equation is solvable (we will show that it is solvable when $t = 0$). The author wants to point out that this work shows that when the complex dimension equals three or four, the existence of $C$-subsolution is equivalent to the solvability of the dHYM equation. In addition, it is also equivalent to the existence of supercritical $C$-subsolution by Collins--Jacob--Yau \cite{collins20151} and the numerical criterion by Chen \cite{chen2021j} (see also Chu--Lee--Takahashi \cite{chu2021nakai}).
\end{frmk}

The layout of this paper is as follows: in Section~\ref{sec:2}, we discuss some background materials. In Section~\ref{sec:2.2}, we consider some special semialgebraic sets in real algebraic geometry, which are defined by systems of inequalities of polynomials with real coefficients. We calculate some optimal bounds such that we can obtain some Positivstellensatz type results over these semialgebraic sets. In Section~\ref{sec:2.3}, we introduce the notion of $\Upsilon$-cones, which is an extension of the $C$-subsolution cone introduced by Székelyhidi \cite{szekelyhidi2018fully}. Roughly speaking, we consider the $C$-subsolution cone of the $C$-subsolution cone, etc. Under our structural assumption, we get Noetherian ascending type cones. In Section~\ref{sec:3}, we explain the reasons of the constraints of our continuity path connecting the dHYM equation to the general inverse $\sigma_k$ type equation. For example, why are these constraints related to the Positivstellensatz type theorem and the $\Upsilon$-cones. In Section~\ref{sec:4}, we prove a priori estimates if all the four constraints are satisfied. In Section~\ref{sec:5}, we discuss the existence results when the phase $\hat{\theta}$ is close to the supercritical phase. To be more precise, for complex dimension three or four, we prove the conjecture by Collins--Jacob--Yau \cite{collins20151} that their existence theorem can be improved when $\hat{\theta} \in \bigl(  (n-2 ) {\pi}/{2},  ( (n-2) + {2}/{n}     )  {\pi}/{2}  \bigr)$, where $n$ is the complex dimension of the manifold. In particular, in Section~\ref{sec:5.1}, for complex dimension three, we find explicit continuity path satisfying the constraints in Section~\ref{sec:3} and show that the existence of $C$-subsolution will be sufficient to obtain the solvability of the dHYM equation. In Section~\ref{sec:5.2}, for complex dimension four, we prove that the existence of $C$-subsolution will also be sufficient to provide the solvability of the dHYM equation.



\Acknow

\section{Preliminaries}
\label{sec:2}
\subsection{Basic Formulas of Eigenvalues and Symmetric Functions}
\label{sec:2.1}

In this subsection, we state some lemmas for symmetric functions first. One can also check the following references \cite{szekelyhidi2018fully, szekelyhidi2017gauduchon} for more details.


\hypertarget{L:2.1}{\begin{flemma}}
If $F(\Lambda) = f(\lambda_1, \dots, \lambda_n)$ is a smooth function in the eigenvalues $\lambda = \{ \lambda_1, \cdots, \lambda_n\}$ of a Hermitian matrix $\Lambda$, then at a diagonal matrix $\Lambda$ with distinct eigenvalues $\lambda_i$, we get
\begin{align*}
\frac{\partial F}{\partial \Lambda_{i }^j}(\Lambda) = \delta_{ij} f_i(\lambda);\quad  \frac{\partial^2 F}{\partial \Lambda_{i }^j \Lambda_{r  }^s} (\Lambda) = f_{ir}(\lambda) \delta_{ij} \delta_{rs} + \frac{f_i - f_j}{ \lambda_i - \lambda_j}(\lambda) (1 - \delta_{ij})\delta_{is} \delta_{jr}, 
\end{align*}where $f_i(\lambda) =  \frac{\partial f}{\partial \lambda_i}(\lambda)$ and $f_{ir} = \frac{\partial^2 f}{\partial \lambda_i \partial \lambda_r} (\lambda)$.
\end{flemma}

Let $\lambda = \bigl \{  \lambda_1,  \dots, \lambda_n \bigr \}$ be the eigenvalues of the Hermitian endomorphism $\omega^{i \bar{k}} \bigl (  X + \sqrt{-1} \partial \bar{\partial}  {u}   \bigr )_{j \bar{k}}$. Since we are on a Kähler manifold, we can pick the following coordinates to simplify our computation.

\hypertarget{L:2.2}{\begin{flemma}}At any point $p \in M$, there exists a holomorphic coordinates near $p$ such that 
\begin{align*}
\omega_{i \bar{j}} (p) = \delta_{ij}; \quad     ( X_u     )_{i \bar{j}} (p) = \lambda_i \delta_{ij};  \quad \omega_{i \bar{j}, k} (p) = 0,   
\end{align*}for all $i, j, k \in \{ 1, \dots, n\}$.
\end{flemma}

From now on, without further notice, we always use the above coordinates. We denote $\Lambda$ as the Hermitian endomorphism $\omega^{i \bar{k}} \bigl (  X + \sqrt{-1} \partial \bar{\partial}  {u}   \bigr )_{j \bar{k}}$. Then the first and second derivatives of $\Lambda$ will be the following. 


\hypertarget{L:2.3}{\begin{flemma}}The first and second derivatives of $\Lambda$ are
\begin{align*}
\frac{\partial \Lambda_i^j }{\partial \bar{z}_k} &=  \tensor[]{\omega}{^{j \bar{p}}_{, \bar{k}}}     (  X_u      )_{i \bar{p}} +  \tensor[]{\omega}{^{j \bar{p}}}     (  X_u      )_{i \bar{p},\bar{k}} =  - \omega^{j \bar{b}} \omega_{a \bar{b}, \bar{k}} \omega^{a \bar{p}}     (  X_u      )_{i \bar{p}} +  \tensor[]{\omega}{^{j \bar{p}}}    (  X_u      )_{i \bar{p},\bar{k}},  \\ 
\frac{\partial^2 \Lambda_i^j }{\partial z_l \partial \bar{z}_k} &= \tensor[]{\omega}{^{j \bar{p}}_{, \bar{k}l}}    (  X_u      )_{i \bar{p}} + \tensor[]{\omega}{^{j \bar{p}}_{, \bar{k}}}     (  X_u      )_{i \bar{p},l} +  \tensor[]{\omega}{^{j \bar{p}}_{, l}}     (  X_u      )_{i \bar{p},\bar{k}} + \tensor[]{\omega}{^{j \bar{p}}}    (  X_u      )_{i \bar{p},\bar{k}l}   \\
&=  \omega^{j \bar{d}} \omega_{c \bar{d}, l} \omega^{c \bar{b}} \omega_{a \bar{b}, \bar{k}} \omega^{a \bar{p}}    (  X_u      )_{i \bar{p}} - \omega^{j \bar{b}} \omega_{a \bar{b}, \bar{k} l} \omega^{a \bar{p}}     (  X_u      )_{i \bar{p}} + \omega^{j \bar{b}} \omega_{a \bar{b}, \bar{k}} \omega^{a \bar{d}} \omega_{c \bar{d}, l} \omega^{c \bar{p}}    (  X_u      )_{i \bar{p}} \\
&\kern2em - \omega^{j \bar{b}} \omega_{a \bar{b}, \bar{k}} \omega^{a \bar{p}}    (  X_u      )_{i \bar{p},l} - \omega^{j \bar{b}} \omega_{a \bar{b}, l} \omega^{a \bar{p}}     (  X_u      )_{i \bar{p},\bar{k}} + \omega^{j \bar{p}}    (  X_u      )_{i \bar{p},\bar{k}l} , 
\end{align*}where we denote $   ( X_u      )_{i \bar{j}} = X_{i \bar{j}} + u_{i \bar{j}}$ and $\Lambda$ is the Hermitian endomorphism $\omega^{-1}   ( X_u     )$. 
\end{flemma}

If we evaluate at any fixed point $p \in M$ and we use the coordinates in Lemma~\hyperlink{L:2.2}{2.2}, we can simplify the first and second derivatives of $\Lambda$.

\hypertarget{L:2.4}{\begin{flemma}}At any fixed point $p$, by picking the coordinates in Lemma~\hyperlink{L:2.2}{2.2}, we get
\begin{align*}
\frac{\partial \Lambda_i^j}{\partial \bar{z}_k} (p) =      (  X_u      )_{i \bar{j},\bar{k}}; \quad \frac{\partial^2 \Lambda_i^j}{\partial z_l \partial \bar{z}_k} (p) =   - \lambda_i \omega_{i \bar{j}, \bar{k} l}    +     (  X_u       )_{i \bar{j},\bar{k}l}.
\end{align*}
\end{flemma}

\subsection{\texorpdfstring{$\Upsilon$}{}-cones and the Positivstellensatz}
\label{sec:2.2}
In this subsection, we introduce the following $\Upsilon$-cones and these $\Upsilon$-cones will be used throughout this paper. $\Upsilon$-cones indeed keep track of a lot of informations when we are deforming the $C$-subsolution cone introduced by Székelyhidi \cite{szekelyhidi2018fully} (see also Guan \cite{guan2014second}). We will see this in the later sections. 

\hypertarget{D:2.1}{\begin{fdefi}[$\Upsilon$-cone]}
Let $\Upsilon^n_{i; c_i, c_{i-1}, \dots, c_{1}, c_0 } \subset \mathbb{R}^n$ be defined by
\begin{align*}
\label{eq:2.1}
\Upsilon^n_{i; c_i, c_{i-1}, \dots, c_{1}, c_0 } \coloneqq \Bigl \{ \bigl ( \lambda_1, \lambda_2, \cdots, \lambda_n \bigr)  \colon    \sum_{k=0}^{i} c_{k} \sigma_{k} (\lambda_s^{(i)})   > 0,\ \forall s \in S_n  \Bigr\}. \tag{2.1}
\end{align*}Here $S_n$ is the symmetric group on $n$ objects, $c_k$ are constants, $\lambda_s^{(i)} \coloneqq \bigl \{ \lambda_{s(1)}, \lambda_{s(2)}, \cdots, \lambda_{s(i)} \bigr\}$, and $\sigma_k(\lambda_s^{(i)})$ are the $k$-th symmetric polynomials of $\lambda_s^{(i)}$, that is,
\begin{align*}
\sigma_k(\lambda_s^{(i)}) \coloneqq \sum_{\substack{\lambda_{j_l} \in \lambda_s^{(i)} \\  \lambda_{j_1} <  \cdots <  \lambda_{j_k}}} \lambda_{j_1} \lambda_{j_2} \cdot \cdots \cdots \lambda_{j_k}.
\end{align*}For convenience, we define $\sigma_0(\lambda_s^{(i)}) \coloneqq 1$.
\end{fdefi}


\hypertarget{R:2.1}{\begin{frmk}}
These $\Upsilon$-cones are all semialgebraic sets in real algebraic geometry. Here, since they are defined by finitely many inequalities, they are open subsets in $\mathbb{R}^n$. So any open connected component of $\Upsilon$-cone will be path-connected. 
\end{frmk}

\hypertarget{P:2.1}{\begin{fprop}}We have the following
\begin{enumerate}[itemsep=1mm]
\item[(a)] $\Gamma_k = \Upsilon^n_{n; 1, 0, \cdots, 0, 0} \cap \Upsilon^n_{n; 0, 1, 0, \cdots, 0, 0} \cap  \cdots \cap \Upsilon^n_{n; 0,   \cdots, 0, \underset{k\text{-th}}{1}, 0, \cdots, 0, 0}$.
\item[(b)] $\Upsilon^n_{1; 1, 0}$ is the positive orthant $\Gamma_n$.  
\end{enumerate}Here $\Gamma_k =   \{  \lambda \in \mathbb{R}^n \colon \  \sigma_1(\lambda) > 0, \cdots, \sigma_k(\lambda) > 0     \}$  is the $k$-positive cone.
\end{fprop}

\begin{proof}
By definition, we have
\begin{align*}
\Gamma_k = \Upsilon^n_{n; 1, 0, \cdots, 0, 0} \cap \Upsilon^n_{n; 0, 1, 0, \cdots, 0, 0} \cap  \cdots \cap \Upsilon^n_{n; 0,   \cdots, 0, \underset{k \text{th}}{1}, 0, \cdots, 0, 0}.
\end{align*}Also, we get
\begin{align*}
\Upsilon^n_{1; 1, 0} &= \bigl \{   ( \lambda_1, \lambda_2, \cdots, \lambda_n  )  \colon   \sigma_1 (\lambda_s^{(1)})  = \lambda_{s(1)}   > 0,\ \forall s \in S_n  \bigr\} =  \cap_i   \{ \lambda_i > 0   \}.
\end{align*}
\end{proof}

Here, since we only focus on complex dimension three and four manifolds, we assume $n \geq 3$.

\hypertarget{T:2.1}{\begin{fthm}[Positivstellensatz]}  \ 
\begin{enumerate}[itemsep=1mm]
\item[(a)] For $c > 0$, $\Upsilon^n_{2; 1, 0, -c} \cap \Upsilon^n_{1; 1, 0}$ is one of the connected component of $\Upsilon^n_{2; 1, 0, -c}$.
\item[(b)] For $c > 0$, $d \geq  0$ and $2 c^{3/2} > d$, $\Upsilon^n_{3; 1, 0, -c, d} \cap \Upsilon^n_{2; 1, 0, -c} \cap \Upsilon^n_{1; 1, 0}$ is one of the connected component of $\Upsilon^n_{3; 1, 0, -c, d}$.
\item[(c)] For $c > 0$ and $d \geq 0$, $\Upsilon^n_{3; 0, c, -d, e} \cap \Upsilon^n_{2; 1, 0, -c} \cap \Upsilon^n_{1; 1, 0}$ is a connected set. 
\item[(d)] For $c > 0$, $d \geq 0$ and $2 c^{3/2} > d$, then $\Upsilon^n_{2; 1, 0, -c} \cap \Upsilon^n_{1; 1, 0}$ is contained in $\Upsilon^n_{2; 0, c, -d} \cap \Upsilon^n_{1; 1, 0}$. 
\item[(e)] For $c > 0$, $d  \geq 0$, $2 c^{3/2} > d$, and $e > -24 c^2 \cos^2(\theta_{c, d})\cos(2\theta_{c, d})$, then $\Upsilon^n_{3; 1, 0, -c, d} \cap \Upsilon^n_{2; 1, 0, -c} \cap \Upsilon^n_{1; 1, 0}$ is contained in $\Upsilon^n_{3; 0, c, -d, e} \cap \Upsilon^n_{2; 1, 0, -c} \cap \Upsilon^n_{1; 1, 0}$, where
\begin{align*}
\theta_{c,d} \coloneqq   \arccos \bigl(   {-d}/{2c^{3/2}}  \bigr)/3  - {2\pi}/{3}
\end{align*}and we specify the branch so that $\arccos \bigl(   {-d}/{2c^{3/2}}  \bigr) \in \bigl( \pi, 3\pi/2 \bigr]$.
\item[(f)] For $c > 0$, $d \geq  0$, $2 c^{3/2} > d$, and $e > -24 c^2 \cos^2(\theta_{c, d})\cos(2\theta_{c, d})$, then $\Upsilon^n_{4; 1, 0, -c, d, -e} \cap \Upsilon^n_{3; 1, 0, -c, d} \cap \Upsilon^n_{2; 1, 0, -c} \cap \Upsilon^n_{1; 1, 0}$ is one of the connected component of $\Upsilon^n_{4; 1, 0, -c, d, -e}$.
\end{enumerate}
\end{fthm}

\begin{proof}
To prove (a), we first prove that $\Upsilon^n_{2; 1, 0, -c} \cap \Upsilon^n_{1; 1, 0}$ only has one connected component. Let $\lambda, \tilde{\lambda} \in \Upsilon^n_{2; 1, 0, -c} \cap \Upsilon^n_{1; 1, 0}$, we show that $t \lambda + (1-t) \tilde{\lambda} \in \Upsilon^n_{2; 1, 0, -c} \cap \Upsilon^n_{1; 1, 0}$ for $t \in [0, 1]$. We have
\begin{align*}
\sigma_1 \bigl (   (t \lambda + (1-t) \tilde{\lambda}  )^{(1)}_s   \bigr) = t \lambda_{s(1)} + (1-t) \tilde{\lambda}_{s(1)} > 0
\end{align*}for all $s \in S_n$, hence $t \lambda + (1-t) \tilde{\lambda} \in \Upsilon^n_{1; 1, 0}$. Moreover, we get
\begin{align*}
\sigma_2 \bigl (   (t \lambda + (1-t) \tilde{\lambda}  )^{(2)}_s   \bigr) &=   ( t \lambda_{s(1)} + (1-t) \tilde{\lambda}_{s(1)}    )   ( t \lambda_{s(2)} + (1-t) \tilde{\lambda}_{s(2)}    ) \\
&= t^2 \sigma_2    (   \lambda^{(2)}_s    ) + (1-t)^2 \sigma_2    (   \tilde{\lambda}^{(2)}_s    ) + t (1-t)   (  \lambda_{s(1)}   \tilde{\lambda}_{s(2)} +  \tilde{\lambda}_{s(1)}  \lambda_{s(2)}    ) \\
&\geq t^2 \sigma_2    (   \lambda^{(2)}_s    ) + (1-t)^2 \sigma_2    (   \tilde{\lambda}^{(2)}_s    ) + 2 t (1-t) \sqrt{\lambda_{s(1)} \lambda_{s(2)} \tilde{\lambda}_{s(1)} \tilde{\lambda}_{s(2)} } \\
&\geq t^2 \sigma_2    (   \lambda^{(2)}_s    ) + (1-t)^2 \sigma_2    (   \tilde{\lambda}^{(2)}_s    ) + 2 t (1-t)  \min \bigl\{ \sigma_2    (   \lambda^{(2)}_s    ), \sigma_2    (   \tilde{\lambda}^{(2)}_s    )   \bigr\} > c
\end{align*}for all $s \in S_n$. Thus $t \lambda + (1-t) \tilde{\lambda} \in \Upsilon^n_{2; 1, 0, -c}$, which implies that $\Upsilon^n_{2; 1, 0, -c} \cap \Upsilon^n_{1; 1, 0}$ is path-connected. Moreover, it is convex. \bigskip

Then, we try to show that $\Upsilon^n_{2; 1, 0, -c} \cap \Upsilon^n_{1; 1, 0}$ is indeed one of the connected component of $\Upsilon^n_{2; 1, 0, -c}$. If not, since any open connected subset of $\mathbb{R}^n$ is path-connected, there exists a $\lambda$ in this connected component such that $\sigma_1   (   \lambda^{(1)}_s    ) = 0$ for some $s \in S_n$. By fixing this $\lambda$ and $s \in S_n$, we have $\sigma_2   (   \lambda^{(2)}_s    ) = \lambda_{s(1)} \lambda_{s(2)} = 0$, which is a contradiction.\bigskip

To prove (b), we first prove that $\Upsilon^n_{3; 1, 0, -c, d} \cap \Upsilon^n_{2; 1, 0, -c} \cap \Upsilon^n_{1; 1, 0}$ only has one connected component. Let $\lambda = (\lambda_1, \lambda_2, \cdots, \lambda_n) \in \Upsilon^n_{3; 1, 0, -c, d} \cap \Upsilon^n_{2; 1, 0, -c} \cap \Upsilon^n_{1; 1, 0}$ and $\tilde{\lambda} = (\tilde{\lambda}_1, \lambda_2, \cdots, \lambda_n) \in \Upsilon^n_{3; 1, 0, -c, d} \cap \Upsilon^n_{2; 1, 0, -c} \cap \Upsilon^n_{1; 1, 0}$ with $\tilde{\lambda}_1 \geq \lambda_1$. Then we may show that $t \lambda + (1-t) \tilde{\lambda} \in \Upsilon^n_{3; 1, 0, -c, d} \cap \Upsilon^n_{2; 1, 0, -c} \cap \Upsilon^n_{1; 1, 0}$ for $t \in [0, 1]$. According to (a), we only need to show that $t \lambda + (1-t) \tilde{\lambda} \in \Upsilon^n_{3; 1, 0, -c, d}$ for $t \in [0, 1]$. We see

\begin{align*}
\kern1em
&\kern-1em \sigma_3 \bigl (   (t \lambda + (1-t) \tilde{\lambda}  )^{(3)}_s   \bigr)\\ 
&=   ( t \lambda_{s(1)} + (1-t) \tilde{\lambda}_{s(1)}    )   ( t \lambda_{s(2)} + (1-t) \tilde{\lambda}_{s(2)}    )   ( t \lambda_{s(3)} + (1-t) \tilde{\lambda}_{s(3)}    ) \\
&= t^3 \lambda_{s(1)}  \lambda_{s(2)} \lambda_{s(3)} + t^2(1-t)   (  \tilde{\lambda}_{s(1)}  {\lambda}_{s(2)}  \lambda_{s(3)} +  \lambda_{s(1)} \tilde{\lambda}_{s(2)}  \lambda_{s(3)}  +  \lambda_{s(1)}  {\lambda}_{s(2)}  \tilde{\lambda}_{s(3)}     ) \\
&\kern2em + t(1-t)^2   (   {\lambda}_{s(1)}  \tilde{\lambda}_{s(2)}  \tilde{\lambda}_{s(3)} +  \tilde{\lambda}_{s(1)}  {\lambda}_{s(2)}  \tilde{\lambda}_{s(3)}  +  \tilde{\lambda}_{s(1)}  \tilde{\lambda}_{s(2)}   {\lambda}_{s(3)}     )  + (1-t)^3 \tilde{\lambda}_{s(1)}  \tilde{\lambda}_{s(2)}   \tilde{\lambda}_{s(3)}.
\end{align*}There are two cases, if $s(i) = 1$ for some $i \in \{1, 2, 3\}$, say $s(1) = 1$ for convenience, then we get
\begin{align*}
\sigma_3 \bigl (   (t \lambda + (1-t) \tilde{\lambda}  )^{(3)}_s   \bigr) &= t^3 \lambda_{s(1)}  \lambda_{s(2)} \lambda_{s(3)} + t^2(1-t)   (  \tilde{\lambda}_{s(1)}  {\lambda}_{s(2)}  \lambda_{s(3)} + 2 \lambda_{s(1)}  {\lambda}_{s(2)}  \lambda_{s(3)}      ) \\
&\kern2em + t(1-t)^2   (   {\lambda}_{s(1)}   {\lambda}_{s(2)}   {\lambda}_{s(3)} + 2 \tilde{\lambda}_{s(1)}  {\lambda}_{s(2)}   {\lambda}_{s(3)}       )  + (1-t)^3 \tilde{\lambda}_{s(1)}   {\lambda}_{s(2)}    {\lambda}_{s(3)} \\
&= t \lambda_{s(1)}  \lambda_{s(2)} \lambda_{s(3)} + (1-t) \tilde{\lambda}_{s(1)}  \lambda_{s(2)} \lambda_{s(3)} \\
&>  ct \sigma_1  \bigl ( \lambda^{(3)}_s \bigr ) -dt  + c(1-t) \sigma_1 \bigl ( \tilde{\lambda}^{(3)}_s \bigr ) -d(1-t) \\
&= c \sigma_1 \bigl(    (t \lambda + (1-t) \tilde{\lambda}  )^{(3)}_s  \bigr) -d.
\end{align*}If $s(i) \neq 1$ for all $i \in \{1, 2, 3\}$, then we have
\begin{align*}
\sigma_3  \bigl (   (t \lambda + (1-t) \tilde{\lambda}  )^{(3)}_s  \bigr  ) &= \sigma_3 \bigl  (   \lambda^{(3)}_s  \bigr  ) >   c \sigma_1  (    \lambda^{(3)}_s   ) -d = c \sigma_1  (    (t \lambda + (1-t) \tilde{\lambda}  )^{(3)}_s   ) -d.
\end{align*}

We can do this process on each entry to connect any two points in $\Upsilon^n_{3; 1, 0, -c, d} \cap \Upsilon^n_{2; 1, 0, -c} \cap \Upsilon^n_{1; 1, 0}$, so this set has only one connected component. \bigskip

Last, we try to show that $\Upsilon^n_{3; 1, 0, -c, d} \cap \Upsilon^n_{2; 1, 0, -c} \cap \Upsilon^n_{1; 1, 0}$ is indeed one of the connected component of $\Upsilon^n_{3; 1, 0, -c, d}$. If not, consider the connected component of $\Upsilon^n_{3; 1, 0, -c, d}$ containing $\Upsilon^n_{3; 1, 0, -c, d} \cap \Upsilon^n_{2; 1, 0, -c} \cap \Upsilon^n_{1; 1, 0}$, then there exists a $\lambda$ in this connected component such that $\sigma_2 \bigl (   \lambda^{(2)}_s   \bigr) - c = 0$ for some $s \in S_n$. By fixing this $\lambda$ and $s \in S_n$, we have
\begin{align*}
0 < \sigma_3 \bigl (   \lambda^{(3)}_s   \bigr) -c \sigma_1 \bigl (   \lambda^{(3)}_s   \bigr) + d \sigma_0 \bigl (   \lambda^{(3)}_s   \bigr) &= \lambda_{s(1)} \lambda_{s(2)} \lambda_{s(3)} - c \bigl( \lambda_{s(1)} + \lambda_{s(2)} + \lambda_{s(3)} \bigr) +d \\
&= -c \bigl( \lambda_{s(1)} + \lambda_{s(2)}  \bigr) + d. 
\end{align*}So by the Cauchy--Schwarz inequality, we get $d > c \bigl( \lambda_{s(1)} + \lambda_{s(2)} \bigr) \geq 2 c\sqrt{\lambda_{s(1)}  \lambda_{s(2)}} = 2 c^{3/2}$, which leads to a contradiction.\bigskip

To show (c), we first prove that $\Upsilon^n_{3; 0, c, -d, e} \cap \Upsilon^n_{2; 1, 0, -c} \cap \Upsilon^n_{1; 1, 0}$ only has one connected component. Let $\lambda = (\lambda_1, \lambda_2, \cdots, \lambda_n) \in \Upsilon^n_{3; 0, c, -d, e} \cap \Upsilon^n_{2; 1, 0, -c} \cap \Upsilon^n_{1; 1, 0}$. We can show that if we have $\tilde{\lambda} = (\tilde{\lambda}_1, \lambda_2, \cdots, \lambda_n) \in \Upsilon^n_{3; 0, c, -d, e} \cap \Upsilon^n_{2; 1, 0, -c} \cap \Upsilon^n_{1; 1, 0}$ with $\tilde{\lambda}_1 \geq \lambda_1$, then $t \lambda + (1-t) \tilde{\lambda} \in \Upsilon^n_{3; 0, c, -d, e} \cap \Upsilon^n_{2; 1, 0, -c} \cap \Upsilon^n_{1; 1, 0}$ for $t \in [0, 1]$. According to (a), we only need to show that $t \lambda + (1-t) \tilde{\lambda} \in \Upsilon^n_{3; 0, c, -d, e} \cap \Upsilon^n_{2; 1, 0, -c}$ for $t \in [0, 1]$. First, for $\Upsilon^n_{2; 1, 0, -c}$, we get
\begin{align*}
\sigma_2 \bigl (   (t \lambda + (1-t) \tilde{\lambda}  )^{(2)}_s   \bigr) &=  ( t \lambda_{s(1)} + (1-t) \tilde{\lambda}_{s(1)}    )    ( t \lambda_{s(2)} + (1-t) \tilde{\lambda}_{s(2)}    )   \\
&= t^2 \sigma_2 \bigl ( \lambda^{(2)}_s   \bigr) + t(1-t)  ( \lambda_{s(1)} \tilde{\lambda}_{s(2)} +  {\lambda}_{s(2)} \tilde{\lambda}_{s(1)}  )  + (1-t)^2 \sigma_2 \bigl ( \tilde{\lambda}^{(2)}_s   \bigr).
\end{align*}There are two cases, if $s(i) = 1$ for some $i \in \{1, 2\}$, say $s(1) = 1$ for convenience, then we get
\begin{align*}
\sigma_2 \bigl (   (t \lambda + (1-t) \tilde{\lambda}  )^{(2)}_s   \bigr) &=  t \sigma_2 \bigl ( \lambda^{(2)}_s   \bigr)  + (1-t) \sigma_2 \bigl ( \tilde{\lambda}^{(2)}_s   \bigr) >  ct + c(1-t) = c.
\end{align*}If $s(i) \neq 1$ for all $i \in \{1, 2, 3\}$, then we have
\begin{align*}
\sigma_2 \bigl (   (t \lambda + (1-t) \tilde{\lambda}  )^{(2)}_s   \bigr) &= \sigma_2 \bigl (   \lambda^{(3)}_s   \bigr) >   c.
\end{align*}So $t \lambda + (1-t) \tilde{\lambda} \in   \Upsilon^n_{2; 1, 0, -c}$ for $t \in [0, 1]$. For $\Upsilon^n_{3; 0, c, -d, e}$, we see
\begin{align*}
\kern1em
&\kern-1em c\sigma_2 \bigl (   (t \lambda + (1-t) \tilde{\lambda}  )^{(3)}_s   \bigr) \\
&=   c   ( t \lambda_{s(1)} + (1-t) \tilde{\lambda}_{s(1)}    )    ( t \lambda_{s(2)} + (1-t) \tilde{\lambda}_{s(2)}    ) + c   ( t \lambda_{s(1)} + (1-t) \tilde{\lambda}_{s(1)}    )    ( t \lambda_{s(3)} + (1-t) \tilde{\lambda}_{s(3)}    ) \\
&\kern2em + c \bigl ( t \lambda_{s(2)} + (1-t) \tilde{\lambda}_{s(2)}  \bigr )  \bigl ( t \lambda_{s(3)} + (1-t) \tilde{\lambda}_{s(3)}  \bigr ). 
\end{align*}Similarly, we consider these two cases, if $s(i) = 1$ for some $i \in \{1, 2, 3\}$, say $s(1) = 1$ for convenience, then we get
\begin{align*}
c \sigma_2 \bigl (   (t \lambda + (1-t) \tilde{\lambda}  )^{(3)}_s   \bigr) 
&> d \sigma_1 \bigl(    (t \lambda + (1-t) \tilde{\lambda}  )^{(3)}_s  \bigr) - e.
\end{align*}If $s(i) \neq 1$ for all $i \in \{1, 2, 3\}$, then we have
\begin{align*}
c \sigma_2 \bigl (   (t \lambda + (1-t) \tilde{\lambda}  )^{(3)}_s   \bigr) &= c \sigma_2 \bigl (   \lambda^{(3)}_s   \bigr) >   d \sigma_1 \bigl(    \lambda^{(3)}_s  \bigr) -e = d \sigma_1 \bigl(    (t \lambda + (1-t) \tilde{\lambda}  )^{(3)}_s  \bigr) -e.
\end{align*}
Similar as before, we can do this process to connect any two points in $\Upsilon^n_{3; 0, c, -d, e} \cap \Upsilon^n_{2; 1, 0, -c} \cap \Upsilon^n_{1; 1, 0}$, so this set has only one connected component. \bigskip



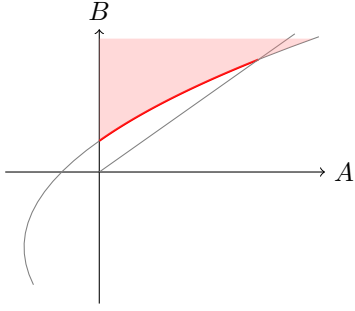
\begin{figure}
\centering
\begin{tikzpicture}[scale=0.5]
  \draw[->] (-2.5,0) -- (6,0) node[right] {$A$};
  \draw[->] (0,-3.5) -- (0,3.8) node[above] {$B$};
  \draw[color=gray,domain=-3:3.6]    plot ({(\x+2)*(\x+2)/4 -2},\x)             ;
  \draw[color=red,thick,domain={2*sqrt(2)-2}:{2*sqrt(2)-2+sqrt(16-8*sqrt(2))},samples=101]    plot ({(\x+2)*(\x+2)/4 -2},\x)             ;
  \draw[color=gray,domain=0:5.2]    plot (\x,{\x/sqrt(2)})             ;
\begin{scope}
    \fill[color=red!50,opacity=0.3,thick,domain={2*sqrt(2)-2}:3.55,samples=101]
    (0, 3.55) -- plot ({(\x+2)*(\x+2)/4 -2},\x) ;
  \end{scope}
\end{tikzpicture}
\caption{The defining region for $A$ and $B$}
\label{fig:2.1}
\end{figure}

For (d), the proof is straightforward. To prove (e), we prove by contradiction. Say there exists a $\lambda \in \Upsilon^n_{3; 1, 0, -c, d} \cap \Upsilon^n_{2; 1, 0, -c} \cap \Upsilon^n_{1; 1, 0}$ satisfying $c\sigma_2 \bigl( \lambda^{(3)}_s \bigr) = d \sigma_1 \bigl( \lambda^{(3)}_s \bigr) -e$ for some $s \in S_n$. That is,
\begin{align*}
c \bigl(  \lambda_{s(1)} \lambda_{s(2)} +   \lambda_{s(1)} \lambda_{s(3)} + \lambda_{s(2)} \lambda_{s(3)}  \bigr) = d \bigl(    \lambda_{s(1)}   +   \lambda_{s(2)} + \lambda_{s(3)} \bigr) - e,
\end{align*}which implies that
\begin{align*}
\label{eq:2.2}
\lambda_{s(1)} = \frac{-c \lambda_{s(2)} \lambda_{s(3)} + d  ( \lambda_{s(2)} + \lambda_{s(3)}  ) -e }{c   ( \lambda_{s(2)} + \lambda_{s(3)}  ) -d }. \tag{2.2}
\end{align*}By equation (\ref{eq:2.2}) and the hypothesis that $\lambda \in \Upsilon^n_{3; 1, 0, -c, d}$, we get
\begin{align*}
\label{eq:2.3}
0 < \sigma_3 \bigl( \lambda^{(3)}_s \bigr) - c \sigma_1 \bigl( \lambda^{(3)}_s \bigr)  +d &= \lambda_{s(1)} \lambda_{s(2)} \lambda_{s(3)} - c \bigl( \lambda_{s(1)} + \lambda_{s(2)} + \lambda_{s(3)}  \bigr) + d \tag{2.3} \\
&= \frac{1}{c^2 B} \bigl( cdAB+d^2 A-c^2 A^2-c^3A - ceA  -c^3 B^2  \bigr) \\
&= \frac{-1}{c^2 B} \bigl( c^2 A^2 -  cdAB + c^3 B^2 - CA     \bigr),
\end{align*}where we denote $A \coloneqq \lambda_{s(2)} \lambda_{s(3)} - c$, $B \coloneqq \bigl( \lambda_{s(2)} + \lambda_{s(3)} \bigr) - d/c \geq 2\sqrt{\lambda_{s(2)}  \lambda_{s(3)}}- d/c = 2 \sqrt{A+c} -d/c$, and $C \coloneqq -c^3  -ce + d^2$. Thus, $C > c^2 A -cd B + c^3  {B^2}/{A}$. In particular, if $2 c^{3/2} -d >0$, then the pink shaded region in Figure~\ref{fig:2.1} is the defining domain for $A$ and $B$.\bigskip

By checking the partial derivatives of $c^2 A -cd B + c^3  {B^2}/{A}$ with respect to $A$ and $B$, the infimum of $c^2 A -cd B + c^3  {B}/{A}$ will be obtained on the curve $B = 2 \sqrt{A+c} -  {d}/{c}$, where $B \in \bigl(2\sqrt{c} -d/c, 2\sqrt{c} -d/c + 2 \sqrt{2c -  d / c^{1/2}} \bigr]$. To find the minimum, we substitute $\tilde{B} = B + d/c$ and $A = \tilde{B}^2/4 - c$, so $\tilde{B} \in \bigl(2\sqrt{c}, 2\sqrt{c} + 2 \sqrt{2c -  d / c^{1/2}} \bigr]$. We define
\begin{align*}
h(\tilde{B}) &\coloneqq c^2 A -cd B + c^3  \frac{B}{A} =  c^2 \bigl(  {\tilde{B}^2}/{4} - c   \bigr) -cd \bigl(\tilde{B} -   {d}/{c} \bigr) + c^3 \frac{\tilde{B} -   {d}/{c}}{ {\tilde{B}^2}/{4} -c} \\
&=   \frac{c^2 \tilde{B}^4 - 4cd \tilde{B}^3 + 4(2c^3 +d^2)\tilde{B}^2 -16c^2d \tilde{B} + 16c^4 }{4 \tilde{B}^2 - 16c}.
\end{align*}The derivative withe respect to $\tilde{B}$ will give us
\begin{align*}
h'(\tilde{B}) 
&= \frac{ c \cdot  (c \tilde{B}^2  -2d \tilde{B} + 4c^2  ) \cdot  (  \tilde{B}^3 -12c \tilde{B} + 8d    )  }{2 \tilde{B}^4 - 16c \tilde{B}^2 +32c^2}.
\end{align*}By checking that
$2 \tilde{B}^4 - 16c \tilde{B}^2 +32c^2 > 0$ and $c \tilde{B}^2 + 4c^2 -2d \tilde{B} \geq c (2 \sqrt{c})^2 + 4c^2 -2d \cdot 2 \sqrt{c} = 4 \sqrt{c}(2 c^{3/2} -d ) > 0$. So the extremum happens when $\tilde{B}^3 -12c \tilde{B} + 8d = 0$. For a cubic equation of the form $t^3 + pt +q = 0$ (depressed cubic), the discriminant will be $-\bigl( 4p^3 + 27q^2  \bigr)$, that is,
\begin{align*}
-4 (-12c)^3 -27 (8d)^2 = 2^6 3^3 (4c^3 -d^2) > 0
\end{align*}in our case. This is the case {\it casus irreducilis}, where there will be three real roots. If there is no rational roots, then by Galois theory, this cubic equation cannot be solved in radicals in terms of real quantities. But we can still solve it trigonometrically in terms of real quantities as follows: 
\begin{align*}
\tilde{B}_k &= 4 \sqrt{c} \cos \biggl [ \frac{1}{3} \arccos \Bigl(  \frac{-d}{2c^{3/2}}  \Bigr)  - k \frac{2\pi}{3}    \biggr ]  
\end{align*}are the roots of this cubic equation, where $k \in \{0, 1, 2\}$.\bigskip

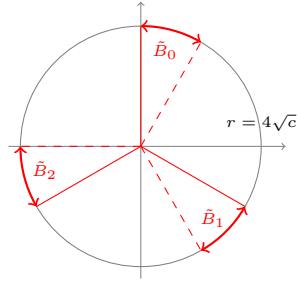
\begin{figure}
\centering
\begin{tikzpicture}[scale=1.6]
\clip (-1.3,-1.3) rectangle (1.3,1.3);
\draw[gray,ultra thin,->] (-1.1,0) -- (1.2,0) coordinate (x axis);
\draw[gray,ultra thin,->] (0,-1.1) -- (0,1.2) coordinate (y axis);
\draw[thin,gray] (0,0) circle (1cm);
    \draw[dashed,red] (0,0) -- (1/2,{sqrt(3)/2});
    \draw[red] (0,0) -- (0,1);
    \draw[red] (0,0) -- ({sqrt(3)/2},-1/2);
    \draw[dashed,red] (0,0) -- (1/2,{-sqrt(3)/2});
    \draw[dashed,red] (0,0) -- ({-1},0);
    \draw[red] (0,0) -- ({-sqrt(3)/2},-1/2);
\node[font=\tiny] at (1,0.2) {$r = 4 \sqrt{c}$};
\draw[<->, thick, red] (0, 0) +(180:1cm) arc (180:210:1cm) node[font=\tiny] at (-0.8,-0.2) {$\tilde{B}_2$};
\draw[<->, thick, red] (0, 0) +(60:1cm) arc (60:90:1cm) node[font=\tiny] at (0.2,0.8) {$\tilde{B}_0$};
\draw[<->, thick, red] (0, 0) +(300:1cm) arc (300:330:1cm) node[font=\tiny] at (0.6,-0.6) {$\tilde{B}_1$};
\end{tikzpicture}
\caption{The angle range of $\tilde{B}_i$}
\label{fig:2.2}
\end{figure}

By hypothesis, we get $-1 <  {-d}/{2c^{3/2}} \leq 0$, so for convenience in later use, we specify the branch so that $\arccos \bigl(   {-d}/{2c^{3/2}}  \bigr) \in \bigl( \pi, 3\pi/2 \bigr]$. Hence, $ \arccos  (   \bullet    )/3 \in  (  {\pi}/{3},   {\pi}/{2}  ]$, which implies that $2\sqrt{3c} \geq \tilde{B}_1 > 2\sqrt{c} > \tilde{B}_0 \geq 0 > - 2\sqrt{3c} \geq \tilde{B}_2 > -4\sqrt{c}$. See Figure~\ref{fig:2.2} for $\tilde{B}_i$'s angle range.\bigskip

We observe that $\tilde{B}_1$ will be the only critical point in $\bigl(2\sqrt{c}, 2\sqrt{c} + 2 \sqrt{2c -  d / c^{1/2}} \bigr]$ and will be the local minimum. Hence,
\begin{align*}
h(\tilde{B}_1) &=  \frac{c^2 \tilde{B}_1^4 - 4cd \tilde{B}_1^3 + 4(2c^3 +d^2)\tilde{B}_1^2 -16c^2d \tilde{B}_1 + 16c^4 }{4 \tilde{B}_1^2 - 16c} \\
&=  \frac{ \bigl( c^2  \tilde{B}_1 - 4cd \bigr) \bigl(  \tilde{B}_1^3 -12c \tilde{B}_1 + 8d   \bigr) + 4(5c^3 +d^2) \tilde{B}_1^2 -72c^2d \tilde{B}_1  +16c(c^3+2d^2) }{4 \tilde{B}_1^2 - 16c} \\
&=  \frac{   4(5c^3 +d^2) \tilde{B}_1^2 -72c^2d \tilde{B}_1  +16c(c^3+2d^2) }{4 \tilde{B}_1^2 - 16c}  = 5c^3 +d^2 + \frac{-18c^2d \tilde{B}_1 +12c (2c^3 + d^2) }{ \tilde{B}_1^2 - 4c} \\
&= 5c^3 +d^2 + \frac{-18 c^{3/2} d \cos(\theta_{c,d}) +3 (2c^3 + d^2) }{4 \cos^2(\theta_{c, d}) - 1},
\end{align*} where $\theta_{c, d} \coloneqq    \arccos  (   {-d}/{2c^{3/2}}   ) \big/3  - {2\pi}/{3}$. So by inequality (\ref{eq:2.3}), we get
\begin{align*}
C \coloneqq -c^3  -ce + d^2 > c^2 A -cd B + c^3  {B^2}/{A} \geq 5c^3 +d^2 +\frac{-18 c^{3/2} d \cos(\theta_{c,d}) +3 (2c^3 + d^2) }{4 \cos^2(\theta_{c, d}) - 1},
\end{align*}which implies that
\begin{align*}
\label{eq:2.4}
e &< -6c^2 + \frac{18 c^{3/2} d \cos(\theta_{c,d}) -3 (2c^3 + d^2) }{c \bigl(4 \cos^2(\theta_{c, d}) - 1\bigr)} = -3 \frac{ 8c^3 \cos^2(\theta_{c, d}) -6 c^{3/2} d \cos(\theta_{c,d})  +d^2 }{c\bigl(4 \cos^2(\theta_{c, d}) - 1\bigr)} \tag{2.4} \\
&= -3 \frac{ \bigl ( 4c^{3/2} \cos(\theta_{c, d}) -d \bigr) \bigl ( 2c^{3/2} \cos(\theta_{c, d}) -d \bigr)  }{c\bigl(4 \cos^2(\theta_{c, d}) - 1\bigr)}. 
\end{align*}

We can simplify inequality (\ref{eq:2.4}). By taking a closer look at the denominator, we apply the triple angle identity to $\cos(\theta_{c, d})$ and get
\begin{align*}
4 \cos^2(\theta_{c, d}) - 1 &= \frac{1}{\cos(\theta_{c, d})} \bigl(  4 \cos^3(\theta_{c, d}) - \cos(\theta_{c, d})    \bigr)  = \frac{1}{\cos(\theta_{c, d})}  \bigl(  \cos(3 \theta_{c, d})  + 2 \cos(\theta_{c,d})      \bigr) \\
&= \frac{1}{\cos(\theta_{c, d})}  \Bigl(  \frac{-d}{2c^{3/2}}  + 2 \cos(\theta_{c,d})      \Bigr),
\end{align*}thus inequality (\ref{eq:2.4}) becomes
\begin{align*}
e &< -3 \frac{ \bigl ( 4c^{3/2} \cos(\theta_{c, d}) -d \bigr) \bigl ( 2c^{3/2} \cos(\theta_{c, d}) -d \bigr)  }{c\bigl(4 \cos^2(\theta_{c, d}) - 1\bigr)} = -6 c^{1/2} \cos(\theta_{c, d}) \bigl ( 2c^{3/2} \cos(\theta_{c, d}) -d \bigr) \\
&= -12 c^2 \cos(\theta_{c, d}) \Bigl (  \cos(\theta_{c, d}) - \frac{d}{2c^{3/2}} \Bigr) =  -12 c^2 \cos(\theta_{c, d}) \Bigl (  \cos(\theta_{c, d}) + \cos(3\theta_{c, d})  \Bigr) \\
&= -24 c^2 \cos^2(\theta_{c, d})\cos(2\theta_{c, d}),
\end{align*}which contradicts to our hypothesis.\bigskip

To prove (f), similar as before, we can prove that $\Upsilon^n_{4; 1, 0, -c, d, -e} \cap \Upsilon^n_{3; 1, 0, -c, d} \cap \Upsilon^n_{2; 1, 0, -c} \cap \Upsilon^n_{1; 1, 0}$ only has one connected component. The difficulty will be trying to show that $\Upsilon^n_{4; 1, 0, -c, d, -e} \cap \Upsilon^n_{3; 1, 0, -c, d} \cap \Upsilon^n_{2; 1, 0, -c} \cap \Upsilon^n_{1; 1, 0}$ is indeed one of the connected component of $\Upsilon^n_{4; 1, 0, -c, d, -e}$. If not, consider the connected component of $\Upsilon^n_{4; 1, 0, -c, d, -e}$ containing $\Upsilon^n_{4; 1, 0, -c, d, -e} \cap \Upsilon^n_{3; 1, 0, -c, d} \cap \Upsilon^n_{2; 1, 0, -c} \cap \Upsilon^n_{1; 1, 0}$, then there exists a $\lambda$ in this connected component such that $\sigma_3 \bigl (   \lambda^{(3)}_s   \bigr) - c \sigma_1 \bigl (   \lambda^{(3)}_s   \bigr) + d = 0$ for some $s \in S_n$. By fixing this $\lambda$ and $s \in S_n$, we have
\begin{align*}
0 &< \sigma_4 \bigl (   \lambda^{(4)}_s   \bigr) -c \sigma_2 \bigl (   \lambda^{(4)}_s   \bigr) + d \sigma_1 \bigl (   \lambda^{(4)}_s   \bigr) - e \\
&= \lambda_{s(4)} \bigl(  \sigma_3 \bigl (   \lambda^{(3)}_s   \bigr) - c \sigma_1 \bigl (   \lambda^{(3)}_s   \bigr) + d  \bigr)   -c  \sigma_2 \bigl(  \lambda_s^{(3)}   \bigr)   + d \sigma_1 \bigl( \lambda_s^{3} \bigr) - e \\
&=  -c  \sigma_2 \bigl(  \lambda_s^{(3)}   \bigr)   + d \sigma_1 \bigl( \lambda_s^{3} \bigr) - e.
\end{align*}That is, $c  \sigma_2 \bigl(  \lambda_s^{(3)}   \bigr)   - d \sigma_1 \bigl( \lambda_s^{3} \bigr) + e < 0$, which contradicts to part (e) that $\Upsilon^n_{3; 1, 0, -c, d} \cap \Upsilon^n_{2; 1, 0, -c} \cap \Upsilon^n_{1; 1, 0}$ is contained in $\Upsilon^n_{3; 0, c, -d, e} \cap \Upsilon^n_{2; 1, 0, -c} \cap \Upsilon^n_{1; 1, 0}$.
\end{proof}

\subsection{\texorpdfstring{$C$}{}-Subsolution and \texorpdfstring{$\Upsilon$}{}-cones}
\label{sec:2.3}



\hypertarget{D:2.2}{\begin{fdefi}[\texorpdfstring{$C$}{}-Subsolution. Székelyhidi \cite{szekelyhidi2018fully} (see also Guan \cite{guan2014second} and Trudinger \cite{trudinger1995dirichlet})]} Consider an equation $F(\Lambda) = h$ on a compact connected Hermitian manifold of dimension $n$, where $\Lambda$ is the Hermitian Endomorphism $\omega^{i \bar{k}} \bigl ( X_u   \bigr )_{j \bar{k}}$ and $F(\Lambda) = f(\lambda_1, \cdots, \lambda_n)$ is a smooth symmetric function of the eigenvalues of $\Lambda$. We assume that $f$ is defined in an open symmetric cone $\Gamma_f \subset \mathbb{R}^n$ satisfying $f > 0$ on $\Gamma_f$, $\partial f/\partial \lambda_i < 0$ for all $i$, and $\inf_{\partial \Gamma_f}  f > \sup_M h$. We say that a smooth function $\ubar{u} \colon M   \rightarrow \mathbb{R}$ is a $C$-subsolution to the equation $F( \Lambda ) = h$ if the following holds: At each point $p \in M$ define the matrix $U^i_j \coloneqq \omega^{i \bar{k}} \bigl (  X + \sqrt{-1} \partial \bar{\partial} \ubar{u}   \bigr )_{j \bar{k}}$, then we require that the set
\begin{align*}
\label{eq:2.5}
F^{h}(\mu) (p) \coloneqq \bigl \{  \lambda  \colon  {f}(\lambda) = h(p) \text{ and } \lambda - \mu(p) \in \Gamma_n  \bigr \} \tag{2.5}
\end{align*}is bounded, where $\mu(p)$ is the $n$-tuple of eigenvalues of $U(p)$. 
\end{fdefi}

\begin{frmk}
In our paper, compared to the sign of $\partial f/\partial \lambda_i$ in Székelyhidi \cite{szekelyhidi2018fully}, we reverse the sign for convenience. In addition, we do not assume all the structural conditions; for example, we do not assume that $f$ is convex nor $\Gamma_f$ contains the positive orthant. Last, we assume $f > 0$ on $\Gamma_f$ which naturally gives us a lower bound, the reason is that in this paper, we are only considering equations satisfying these properties.
\end{frmk}

\hypertarget{D:2.3}{\begin{fdefi}[Alternative definition of Definition~\hyperlink{D:2.2}{2.2}. Székelyhidi \cite{szekelyhidi2018fully} (see also Trudinger \cite{trudinger1995dirichlet})]} 
Suppose that $f$ is defined in an open symmetric cone $\Gamma_f \subset \mathbb{R}^n$ satisfying $f > 0$ on $\Gamma_f$, $\partial f/\partial \lambda_i < 0$ for all $i$, and $\inf_{\partial \Gamma_f}  f > \sup_M h$. Define 
\begin{align*}
\label{eq:2.6}
\Gamma_f^h \coloneqq \bigl \{ \lambda \in \Gamma_f \colon f(\lambda) < h  \bigr \}. \tag{2.6}
\end{align*}For $\mu \in \mathbb{R}^n$, the set (\ref{eq:2.5}) $F^{h}(\mu)$ is bounded if and only if $\lim_{t \rightarrow \infty} f(\mu + t e_i) < h$ for all $i$, where $e_i$ is the $i$-th standard vector. We denote by ${\Gamma}_f^{n-1, h}$ the projection of $\Gamma_f^h$ onto $\mathbb{R}^{n-1}$ by dropping the last entry. Then for any $\mu' = (\mu_1, \cdots, \mu_{n-1} ) \in {\Gamma}_f^{n-1, h}$, define the function $f^{(n-1)}$ on ${\Gamma}_f^{n-1, h}$ by the following limit 
\begin{align*}
f^{(n-1)} (\mu_1, \cdots, \mu_{n-1}) &\coloneqq \lim_{\lambda_n \rightarrow \infty} f(\mu_1, \cdots, \mu_{n-1}, \lambda_n) \geq 0.
\end{align*}
First, the set $F^{h}(\mu)$ is bounded if and only if $f^{(n-1)} \bigl(\mu_{s(1)}, \cdots, \mu_{s(n-1)} \bigr) < h$ for every $s \in S_n$, where $S_n$ is the permutation group. This is well-defined since $f$ is a symmetric function. We can show that for $\mu \in \mathbb{R}^n$, $F^h(\mu)$ is bounded if and only if $\bigl(\mu_{s(1)}, \cdots, \mu_{s(n-1)}\bigr) \in \Gamma^{n-1, h}_f$ for every $s \in S_n$.

\end{fdefi}


Inspired by the work of Trudinger \cite{trudinger1995dirichlet} on the Dirichlet problem (over the reals) for equations of the eigenvalues of the Hessian (see also the results of Caffarlli--Nirenberg--Spruck \cite{caffarelli1985dirichlet} and Collins--Székelyhidi \cite{collins2017convergence}). We not only consider the projection onto a codimension one plane, we actually keep doing this process to record as much informations as we can. We introduce the following $\Upsilon$-cones, which will well describe how the $C$-subsolution cone of the dHYM equation changes along the continuity path. Roughly speaking, we will see how the cone, the cone of the cone, the cone of the cone of the cone, etc change.

\hypertarget{D:2.4}{\begin{fdefi}[$\Upsilon$-cones]} Consider an equation $F(\Lambda) = h$ on a compact connected Hermitian manifold of dimension $n$, where $\Lambda$ is the Hermitian Endomorphism $\omega^{i \bar{k}} \bigl ( X_u   \bigr )_{j \bar{k}}$ and $F(\Lambda) = f(\lambda_1, \cdots, \lambda_n)$ is a smooth symmetric function of the eigenvalues of $\Lambda$. We assume that $f$ is defined in an open symmetric cone $\Gamma_f \subset \mathbb{R}^n$ satisfying $f > 0$ on $\Gamma_f$, $\partial f/\partial \lambda_i < 0$ for all $i$, and $\inf_{\partial \Gamma_f}  f > \sup_M h$. At $p \in M$, we define $\Upsilon_1(p)$ to be
\begin{align*}
\Upsilon_1(p) \coloneqq \bigl \{  \mu \in \mathbb{R}^n  \colon    \bigl(\mu_{s(1)}, \cdots, \mu_{s(n-1)}\bigr) \in \Gamma^{n-1, h(p)}_f,\quad \forall s \in S_n \bigr \},  
\end{align*}where $\Gamma^{n-1, h}_f$ is defined in Definition~\hyperlink{D:2.3}{2.3} and for $n-1 \geq k \geq 2$, we define the following $\Upsilon$-cones
\begin{align*}
\Upsilon_k (p) \coloneqq \bigl \{  \mu \in \mathbb{R}^n  \colon    \bigl(\mu_{s(1)}, \cdots, \mu_{s(n-k)}\bigr) \in \Gamma^{n-k, h(p)}_f,\quad \forall s \in S_n  \bigr \},  
\end{align*}where we define $ \Gamma^{n-k, h(p)}_f$ inductively by the projection of $ \Gamma^{n+1-k, h(p)}_f$ onto $\mathbb{R}^{n-k}$ by dropping the last entry.
\end{fdefi}

\hypertarget{R:2.3}{\begin{frmk}}
$\Upsilon_1$ is the $C$-subsolution cone introduced by Székelyhidi \cite{szekelyhidi2018fully}.
\end{frmk}



By our definition, we get this Noetherian ascending chain structure:
\hypertarget{P:2.2}{\begin{fprop}}For an equation $F(\Lambda) = h$ on a compact connected Hermitian manifold of dimension $n$, where $\Lambda$ is the Hermitian Endomorphism $\omega^{i \bar{k}} \bigl ( X_u   \bigr )_{j \bar{k}}$ and $F(\Lambda) = f(\lambda_1, \cdots, \lambda_n)$ is a smooth symmetric function of the eigenvalues of $\Lambda$. We assume that $f$ is defined in an open symmetric cone $\Gamma_f \subset \mathbb{R}^n$ satisfying $f > 0$ on $\Gamma_f$, $\partial f/\partial \lambda_i < 0$ for all $i$ and $\inf_{\partial \Gamma_f}  f > \sup_M h$. At each point $p \in M$, we have
\begin{align*}
F^h(p) \subset \Upsilon_1(p) \subseteq \Upsilon_2(p) \subseteq \cdots  \subseteq \Upsilon_{n-1}(p),
\end{align*}where $F^h(p) \coloneqq \bigl\{  \lambda \in \Gamma_f \colon f(\lambda) < h(p)   \bigr\}$ is an open subset of $\mathbb{R}^n$. Moreover, the solution set $ \partial {F^h(p)} \subset \Upsilon_1(p)$ and $\Upsilon_k(p)$ are open subsets of $\mathbb{R}^n$ for any $k \in \{1, \cdots, n-1\}$.
\end{fprop}

\begin{proof}
The proof should be straightforward.
\end{proof}

 

\begin{figure}
	\centering
	\includegraphics[width=0.8\textwidth]{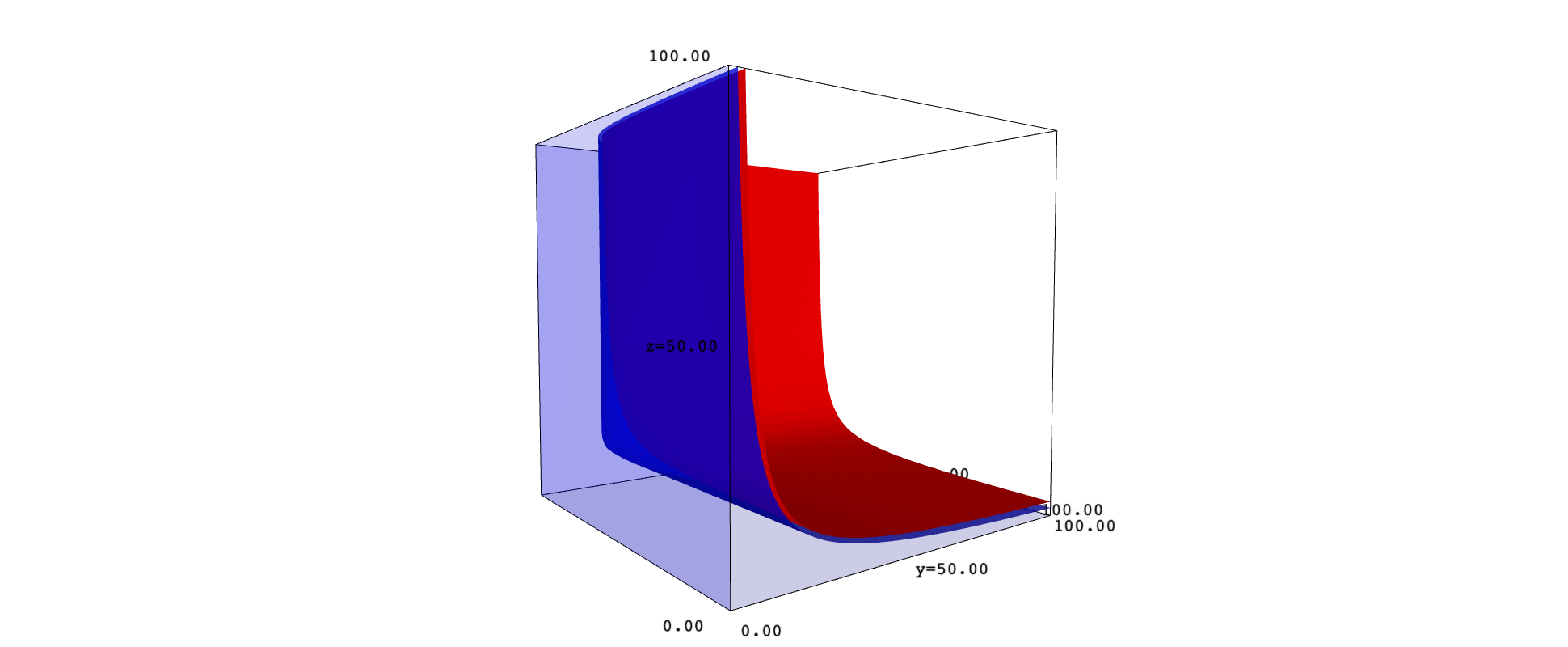}
	\caption{$\Upsilon$-cones for the three dimensional dHYM equation}
	\label{fig:2.3}
\end{figure}

Above Figure \ref{fig:2.3} is an example of the $\Upsilon$-cones for the three dimensional dHYM equation. The red hyperplane is the solution set $\{f(\lambda) =h\}$, the darker blue cone in between is the boundary of the $\Upsilon_1$-cone. By Remark~\hyperlink{R:2.3}{2.3}, the $\Upsilon_1$-cone is the $C$-subsolution cone introduced by Székelyhidi \cite{szekelyhidi2018fully}. Last, the outermost lighter blue cone is the boundary of the $\Upsilon_2$-cone. In fact, the $\Upsilon_2$-cone will be the positive orthant in this case.


\section{Constraints on Continuity Path}
\label{sec:3}

In this section, we will introduce the continuity paths we are considering and explain the reasons why we consider these continuity paths. The idea is to deform the original dHYM equation to some other type of equation of which we already knew the solvability. Along this continuity path, the $C$-subsolution will be preserved and hence the $C$-subsolution will provide us a priori estimates all over the continuity path. A small drawback is that we lose some convexity when we are moving along the path, but we get the solvability.

\subsection{When \texorpdfstring{$n =3$}{}}
\label{sec:3.1}
Here, we only consider the phase $\hat{\theta} \in \bigl( \pi/2, 5 \pi/6    \bigr)$ and the dHYM equation will be 
\begin{align*}
\Im   \bigl (  \omega + \sqrt{-1} \chi   \bigr )^3      = \tan(\hat{\theta}) \cdot  \Re   \bigl (  \omega +  \sqrt{-1} \chi    \bigr )^3   .
\end{align*}By doing a substitution $X \coloneqq \chi -\ta \omega$, the dHYM equation becomes
\begin{align*}
\label{eq:3.1}
X^3 - 3 \s \omega^2 \wedge X -2 \ta \s \omega^3 = 0. \tag{3.1}
\end{align*}We multiply equation (\ref{eq:3.1}) by $\css$ and get
\begin{align*}
\label{eq:3.2}
\css X^3 - 3   \omega^2 \wedge X -2 \ta   \omega^3 = 0. \tag{3.2}
\end{align*}The issue is that for our phase $\hat{\theta}$, $\ta < 0$, otherwise we can solve it using the generalized inverse $\sigma_k$ equation, see Collins--Székelyhidi \cite{collins2017convergence}. We consider the following continuity path inspired by Pingali \cite{pingali2019deformed}, that is,
\begin{align*}
\label{eq:3.3}
\css X^3 - 3 c_1(t)   \omega^2 \wedge X -2 c_0(t) \ta   \omega^3 = 0. \tag{3.3}
\end{align*}Here $t \in [0, 1]$ and $c_0(t)$ and $c_1(t)$ are smooth functions in $t$ which satisfy all the following \hypertarget{dim 3 cons}{4-dimensional version of the four constraints}:
\begin{enumerate}[leftmargin=4.5cm]
	\setlength\itemsep{+0.4em}
\item[Topological constraint:] $\cos^2(\hat{\theta}) \Omega_0 - 3c_1(t)    \Omega_2 - 2 c_0(t) \ta   \Omega_3  =0$.
\item[Boundary constraints:] $c_1(1) = c_0(1) = 1;\quad c_1(0)>0; \quad c_0(0) = 0$.
\item[Positivstellensatz constraint:] $c_1(t)^{ {3}/{2}} > c_0(t) \sin(\hat{\theta})$.
\item[$\Upsilon$-cone constraint:] $c_1'(t) > 0$.
\end{enumerate}\smallskip

Here we denote $\Omega_i \coloneqq \int_M \omega^i \wedge X^{3-i}$. \bigskip



From now on, we will automatically assume the pair $\bigl(c_1(t), c_0(t)\bigr)$ satisfy these \hyperlink{dim 3 cons}{3-dimensional four constraints} unless further notice. In the rest of this section, we will explain why we consider these constraints, the first two are not surprising. In Section~\ref{sec:5.1}, we will show that we can find a pair $\bigl(c_1(t), c_0(t)\bigr)$ such that these constraints will be satisfied. The first constraint, the topological constraint, should always be the same, that is, the integral should always be zero along the continuity path. For the second constraint, when $t = 1$, it is the original equation (\ref{eq:3.2}) and for $t = 0$, it becomes a general complex Hessian equation, which is solvable by the results of Fang--Lai--Ma \cite{fang2011class} provided that there exists a $C$-subsolution to equation (\ref{eq:3.2}).

\hypertarget{R:3.1}{\begin{frmk}}At any point $p \in M$, by picking the coordinates in Lemma~\hyperlink{L:2.2}{2.2}, we can write equation (\ref{eq:3.3}) as
\begin{align*}
\label{eq:3.4}
\css \lambda_1 \lambda_2 \lambda_3  - c_1(t)    \bigl( \lambda_1 + \lambda_2 + \lambda_3 \bigr)  - 2 c_0(t) \ta      = 0. \tag{3.4}
\end{align*}We can rewrite equation (\ref{eq:3.4}) as the following format
\begin{align*}
\label{eq:3.5}
f_t(\lambda_1, \lambda_2, \lambda_3) \coloneqq \frac{c_1(t) (\lambda_1+ \lambda_2 + \lambda_3) + 2c_0(t) \ta}{\lambda_1 \lambda_2 \lambda_3} = \css,   \tag{3.5}
\end{align*}
where $\lambda_i$ are the eigenvalues of $\omega^{-1}X$ at $p$.
\end{frmk}

\hypertarget{R:3.2}{\begin{frmk}}
For convenience, from mow on, we will denote the solution set as $\{f_t = \css\}$. There might be some ambiguities because this set has more than one connected component. But we pick the one such that the enclosed region contains $(R, R, R) \in \mathbb{R}^3$ for $R > 0$ large.
\end{frmk}


\hypertarget{R:3.3}{\begin{frmk}[Positivstellensatz constraint]}
For equation (\ref{eq:3.4}), we can rewrite it as
\begin{align*}
\css \lambda_{s(1)} \bigl(    \lambda_{s(2)} \lambda_{s(3)} - c_1(t) \s    \bigr) = c_1(t) \bigl (\lambda_{s(2)} + \lambda_{s(3)} \bigr ) + 2c_0(t) \ta,
\end{align*}where $s \in S_3$ is a permutation on three objects. By the Positivstellensatz Theorem~\hyperlink{T:2.1}{(a)}, we know $\Upsilon^3_{2; 1, 0, -c_1(t) \s } \cap \Upsilon^3_{1; 1, 0}$ is one of the connected component of $\Upsilon^3_{2; 1, 0, -c_1(t) \s }$. To have a nice $\Upsilon$-cone structure, the $C$-subsolution cone must contain the solution set $\{f_t = \css\}$, so we require both quantities $\lambda_{s(1)}$ and $\lambda_{s(2)} \lambda_{s(3)} - c_1(t) \s$ to be positive. This is true when the quantity $c_1(t) \bigl (\lambda_{s(2)} + \lambda_{s(3)} \bigr ) + 2c_0(t) \ta$ is positive (of course, we omit the case when both $\lambda_{s(1)}$ and $\lambda_{s(2)} \lambda_{s(3)} - c_1(t) \s$ are negative).\bigskip

$c_1(t) \bigl (\lambda_{s(2)} + \lambda_{s(3)} \bigr ) + 2c_0(t) \ta$ is always positive if $\Upsilon^3_{2; 1, 0, -c_1(t) \s} \cap \Upsilon^3_{1; 1, 0}$ is contained in $\Upsilon^3_{2; 0, c_1(t) \s, 2c_0(t) \ta \s} \cap \Upsilon^3_{1; 1, 0}$. By the Positivstellensatz Theorem~\hyperlink{T:2.1}{(d)}, this statement holds when $c_1(t)^{ {3}/{2}} > c_0(t) \sin(\hat{\theta})$.


\end{frmk}

\hypertarget{L:3.1}{\begin{flemma}}
A smooth function $\ubar{u} \colon M   \rightarrow \mathbb{R}$ is a $C$-subsolution to equation (\ref{eq:3.5}) if and only if at each point $p \in M$, we have
\begin{align*}
\mu \in  \Upsilon^3_{2; 1, 0, -c_1(t) \s }  \cap \Upsilon^3_{1; 1, 0},
\end{align*}where $  \{  \mu    \}$ are the eigenvalues of the Hermitian endomorphism $\omega^{i \bar{k}}    (  X + \sqrt{-1} \partial \bar{\partial} \ubar{u}       )_{j \bar{k}}(p)$. Note that the $\Upsilon$-cones for equation (\ref{eq:3.5}) at $t$ will be $\Upsilon_1(t) =  \Upsilon^3_{2; \css, 0, -c_1(t) } \cap \Upsilon^3_{1; 1, 0}$ and $\Upsilon_2(t) = \Upsilon^3_{1; 1, 0}$.
\end{flemma}

\begin{proof}
For convenience, we assume that $\mu_1 \geq \mu_2 \geq \mu_3$. Suppose that for any $p \in M$, we have 
\begin{align*}
\mu \in  \Upsilon^3_{2; 1, 0, -c_1(t) \s }  \cap \Upsilon^3_{1; 1, 0}.
\end{align*}Then since $M$ is a compact manifold and the functions $c_1(t)$, $c_0(t)$, $\mu_1$, $\mu_2$, $\mu_3$ are all continuous functions on $M$, there exists a $\delta > 0$ such that for any $p \in M$
\begin{align*}
\mu_2 \mu_3 >  c_1(t) \s + \delta.
\end{align*} 

If $(\lambda_1, \lambda_2, \lambda_3)$ is a solution to equation (\ref{eq:3.4}) with $\lambda_1 \geq \lambda_2 \geq \lambda_3$ satisfying $\lambda_i \geq \mu_i > 0$, then we have
\begin{align*}
\lambda_2 = \frac{c_1(t)  (\lambda_1 + \lambda_3)  + 2c_0(t) \tan(\hat{\theta})   }{ \css \lambda_1 \lambda_3 - c_1(t)   }, 
\end{align*}which gives 
\begin{align*}
\lambda_2 \lambda_3 &= \frac{c_1(t)   (\lambda_1 
\lambda_3 + \lambda_3^2)   + 2c_0(t) \tan(\hat{\theta})   \lambda_3 }{ \css \lambda_1 \lambda_3 - c_1(t)  } \\
&= c_1(t) \sec^2(\hat{\theta}) + \frac{c_1(t)    \lambda_3^2  + 2c_0(t) \tan(\hat{\theta})   \lambda_3  + c_1^2(t) \sec^2(\hat{\theta}) }{ \css \lambda_1 \lambda_3 - c_1(t)  }.
\end{align*}

Since we assume that $\lambda_3$ is bounded from below by $\mu_3$, we claim that $\lambda_2$ is bounded from above. If not, for any $N > 0$, assume we can find a solution $(\lambda_1, \lambda_2, \lambda_3)$ with $\lambda_1 \geq \lambda_2 \geq N$. By picking $N$ sufficiently large, one can prove the following inequality 
\begin{align*}
\lambda_3 = \frac{c_1(t)   (\lambda_1 + \lambda_2)  + 2c_0(t) \tan(\hat{\theta}) }{ \css \lambda_1 \lambda_2 - c_1(t)  }  \leq \frac{3c_1(t)\sec^2(\hat{\theta})}{N},
\end{align*}which gives us a contradiction by picking $N >  {3c_1(t) \sec^2(\hat{\theta})}/{\mu_3}$.\bigskip

On the other hand, if $\ubar{u}$ is a $C$-subsolution, we first suppose that there exists $p \in M$ with $\mu_3(p) = 0$. We may pick $\epsilon > 0$ sufficiently small so that if we set $\lambda_3 \coloneqq \epsilon$ and $\lambda_2 \coloneqq   {c_1(t)} \sec^2(\hat{\theta})/{\epsilon} + 1$, then
\begin{align*}
\lambda_2 \geq \mu_2; \quad \lambda_3 \geq \mu_3.
\end{align*}
For this choice, we can also check that
\begin{align*}
\lambda_2 \lambda_3 &= c_1(t) \sec^2(\hat{\theta}) + \epsilon > c_1(t) \sec^2(\hat{\theta})
\end{align*}and 
\begin{align*}
c_1(t)  (\lambda_2 + \lambda_3)  &\geq 2 c_1(t)   \sqrt{\lambda_2 \lambda_3} > - 2 c_1^{3/2}(t)   \sec(\hat{\theta})      
\geq -2   \sec^3(\hat{\theta}) \bigl( c_{\hat{\theta}}   + c_0(t) \sin(\hat{\theta}) \bigr).
\end{align*}Here, by the Positivstellensatz constraint, we define $c_{\hat{\theta}} \coloneqq \min_{t \in [0, 1]} \bigl( c_1^{3/2}(t) - c_0(t) \sin(\hat{\theta}) \bigr)  > 0$. So, we may set
\begin{align*}
\lambda_1 \coloneqq \frac{c_1(t)  (\lambda_2 + \lambda_3) + 2c_0(t) \tan(\hat{\theta})  }{ \css \lambda_2 \lambda_3 - c_1(t)  } \geq \frac{-2 c_{\hat{\theta}}   \sec^3(\hat{\theta})    }{\epsilon} > 0.
\end{align*}Again, as long as $\epsilon > 0$ is sufficiently small, we get $\lambda_1 \geq \mu_1$. Hence, the set
\begin{align*}
   \{  \lambda  \colon  {f}_t(\lambda) = \css \text{ and } \lambda - \mu(p)  \in \Gamma_n      \}
\end{align*}is not bounded, which contradicts our assumption that $\ubar{u}$ is a $C$-subsolution.\bigskip

To check the second condition, suppose that there exists $p \in M$ such that $c_1(t) \css \geq \mu_2 \mu_3$. We may pick $\lambda_3 = \mu_3$ and $\lambda_2 = \bigl( {c_1(t) \sec^2(\hat{\theta}) + \epsilon}\bigr)/{ \mu_3}$, where $\epsilon > 0$ will be chosen sufficiently small. We claim that $\lambda_2 > \mu_2 > 0$, with this claim we obtain 
\begin{align*}
\lambda_1 &\coloneqq   \frac{c_1(t)    (\lambda_2 + \lambda_3) + 2 c_0(t) \tan(\hat{\theta})   }{ \css \lambda_2 \lambda_3 - c_1(t)    } = \s \frac{c_1(t)    (\lambda_2 + \lambda_3)   + 2c_0(t) \tan(\hat{\theta})  }{\epsilon} \\
&\geq \s \frac{2c_1(t)    \sqrt{\lambda_2 \lambda_3} + 2c_0(t) \tan(\hat{\theta})  }{\epsilon}  \geq -\frac{2 c_{\hat{\theta}}  \sec^3(\hat{\theta})   }{\epsilon},
\end{align*}where $c_{\hat{\theta}} \coloneqq \min_{t \in [0, 1]} \bigl ( c_1^{3/2}(t) - c_0(t) \sin(\hat{\theta})  \bigr) > 0$. This implies that $\lambda_1$ can be unbounded by picking $\epsilon$ sufficiently small, which contradicts to our assumption that $\ubar{u}$ is a $C$-subsolution. 
\end{proof}


\hypertarget{R:3.4}{\begin{frmk}[$\Upsilon$-cone constraint]} 
Along the continuity path, we hope that the $C$-subsolution condition will be maintained, so we can use this $C$-subsolution to obtain a priori estimates. That is, we are checking whether the $\Upsilon$-cones will be shrinking when $t$ goes from $0$ to $1$. By Lemma~\hyperlink{L:3.1}{3.1}, the $\Upsilon$-cones at $t$ will be 
\begin{align*}
\Upsilon_1(t) = \Upsilon^3_{2; 1, 0, -c_1(t) \s }  \cap \Upsilon^3_{1; 1, 0} \quad\text{and}\quad \Upsilon_2(t) = \Upsilon^3_{1; 1, 0}.
\end{align*}If $c_1'(t) > 0$, then for $1 \geq t > t' \geq 0$, we have $\Upsilon_1(1) \subset \Upsilon_1(t) \subset \Upsilon_1(t') \subset \Upsilon_1(0)$. This is the reason why we require the $\Upsilon$-cone constraint.
\end{frmk}\bigskip


Now with a pair $(c_1(t), c_0(t))$ satisfying the \hyperlink{dim 3 cons}{3-dimensional four constraints}, we consider a more general equation 
\begin{align*}
\label{eq:3.6}
f_t(\lambda_1, \lambda_2, \lambda_3) \coloneqq \frac{c_1(t) (\lambda_1+ \lambda_2 + \lambda_3) + 2c_0(t) \ta}{\lambda_1 \lambda_2 \lambda_3} = h(p),   \tag{3.6}
\end{align*}where $\lambda_i$ are the eigenvalues of $\omega^{-1}X_u$ at $p$ and $h \colon M   \rightarrow \mathbb{R}$ with $c_1^3(t) \ctt/ c_0^2(t) > h(p) > 0$ for all $p \in M$. We will explain why we restrict the range of $h$ in Section~\ref{sec:4.1}. Roughly speaking, we want to preserve the $\Upsilon$-cone structure. In particular, when $h = \css$, equation (\ref{eq:3.6}) is the original equation (\ref{eq:3.5}). The proof of the following Lemma is similar to the proof of Lemma~\hyperlink{L:3.1}{3.1}, the $C$-subsolution to equation (\ref{eq:3.6}) can be described as follows.
\hypertarget{L:3.2}{\begin{flemma}}
A smooth function $\ubar{u} \colon M   \rightarrow \mathbb{R}$ is a $C$-subsolution to equation (\ref{eq:3.6}) if and only if at each point $p \in M$, we have
\begin{align*}
\mu \in  \Upsilon^3_{2; 1, 0, -c_1(t) /h(p) }  \cap \Upsilon^3_{1; 1, 0},
\end{align*}where $  \{  \mu    \}$ are the eigenvalues of the Hermitian endomorphism $\omega^{i \bar{k}}    (  X + \sqrt{-1} \partial \bar{\partial} \ubar{u}       )_{j \bar{k}}(p)$. Note that the $\Upsilon$-cones for equation (\ref{eq:3.6}) at $t$ at point $p$ will be $\Upsilon_1(t, p) =  \Upsilon^3_{2; h(p), 0, -c_1(t) } \cap \Upsilon^3_{1; 1, 0} $ and $\Upsilon_2(t, p) = \Upsilon^3_{1; 1, 0}$.
\end{flemma}

Now, if $c_1^3(t) \ctt/ c_0^2(t) > h > 0$ for all $t \in [0, 1]$, then the $\Upsilon$-cone constraint will automatically be satisfied. So we have the following Lemma.

\hypertarget{L:3.3}{\begin{flemma}}
For $p \in M$, if $c_1^3(t) \ctt/ c_0^2(t) > h(p) > 0$ for all $t \in [0, 1]$, then 
\begin{align*}
 \{ f_1 = h(p) \} \subset \Upsilon_1(t, p) \subset \Upsilon_1(t', p) \text{ and } \Upsilon_2(t, p) \subset \Upsilon_2(t', p)
\end{align*}for $1 \geq t > t' \geq 0$. Here, the solution set $\{ f_t = h(p) \}$ has some ambiguities, because it has more than one connected component. The workaround is the same as Remark~\hyperlink{R:3.2}{3.2}.
\end{flemma}

\subsection{When \texorpdfstring{$n=4$}{}}
\label{sec:3.2}
Here, we only consider the phase $\hat{\theta} \in \bigl( \pi, 5 \pi/4    \bigr)$ and the dHYM equation will be \begin{align*}
\ct \cdot \Im   \bigl (  \omega + \sqrt{-1} \chi   \bigr )^4     =   \Re   \bigl (  \omega +  \sqrt{-1} \chi    \bigr )^4    .
\end{align*}By doing a substitution $X \coloneqq \chi + \ct \omega$, the dHYM equation becomes
\begin{align*}
\label{eq:3.7}
X^4 - 6 \cs \omega^2 \wedge X^2 + 8 \ct \cs \omega^3 \wedge X  - \cc \cs \omega^4 = 0.  \tag{3.7}
\end{align*}
We multiply equation (\ref{eq:3.7}) by $\sii$ and get
\begin{align*}
\label{eq:3.8}
\sii X^4 - 6   \omega^2 \wedge X^2 + 8 \ct   \omega^3 \wedge X  - \cc   \omega^4 = 0. \tag{3.8}
\end{align*}The issue is that for our phase $\hat{\theta}$, $\ct > 0$, otherwise we can solve it using the generalized inverse $\sigma_k$ equation, see Collins--Székelyhidi \cite{collins2017convergence}. We consider the following continuity path:
\begin{align*}
\label{eq:3.9}
\sii X^4 - 6 c_2(t)   \omega^2 \wedge X^2 + 8 c_1(t) \ct   \omega^3 \wedge X - c_0(t) \cc   \omega^4 = 0. \tag{3.9}
\end{align*}
Here $t \in [0, 1]$ and the triple $(c_2(t), c_1(t), c_0(t))$ are smooth functions in $t$ which satisfy the following \hypertarget{dim 4 cons}{4-dimensional version of the four constraints}:
\begin{enumerate}[leftmargin=4.5cm]
	\setlength\itemsep{+0.4em}
\item[Topological constraint:] $\sii \Omega_0 - 6c_2(t) \Omega_2 + 8c_1(t) \ct \Omega_3 - c_0(t) \cc   \Omega_4 =0$.
\item[Boundary constraints:] $c_2(1) = c_1(1) = c_0(1) = 1;\quad c_2(0)>0; \quad c_1(0) = 0$.
\item[Positivstellensatz constraint:] $\cc c_0(t) >  -24   {c_2^2(t) \cs \cos^2(\theta_{c_1, c_2})\cos(2\theta_{c_1, c_2})} $.
\item[$\Upsilon$-cone constraints:] $\frac{d}{dt} \bigl( c_2^{3/2}(t) \bigr) \geq -\cos(\hat{\theta}) c_1'(t); \quad c_2'(t) > 0; \quad c_1'(t) > 0$.
\end{enumerate}Here we denote $\Omega_i \coloneqq \int_M \omega^i \wedge X^{4-i}$ and we define $\theta_{c_1, c_2} \coloneqq  \arccos \bigl(   {  c_1(t) \coo}/{c_2^{3/2}(t)}  \bigr) / 3 -  {2\pi}/{3}$, where we specify the branch so that $\arccos \bigl(   {c_1(t) \cos(\hat{\theta})}/{c_2^{3/2}(t)}  \bigr) \in   \bigl( \pi, 3\pi/2 \bigr]$.\bigskip





From now on, we will automatically assume that the triple $\bigl(c_2(t), c_1(t), c_0(t)\bigr)$ satisfy these \hyperlink{dim 4 cons}{4-dimensional four constraints} unless further notice. In the rest of this section, we will explain why we consider these constraints, the first two are not surprising. The first constraint, the topological constraint, should always be the same, that is, the integral should always be zero along the continuity path. For the second constraint, when $t = 1$, equation (\ref{eq:3.9}) is the original equation (\ref{eq:3.8}). When $t = 0$, equation (\ref{eq:3.9}) becomes
\begin{align*}
\label{eq:3.10}
\sii X^4 - 6 c_2(0)   \omega^2 \wedge X^2   - c_0(0) \cc   \omega^4 = 0. \tag{3.10}
\end{align*}If $c_0(0) > 0$ as well, then equation (\ref{eq:3.10}) belongs to the generalized inverse $\sigma_k$ equation, we can use the results of Collins--Székelyhidi \cite{collins2017convergence} to solve it. Unfortunately, it is possible that $c_0(0) < 0$, but we are still able to fix this. In Section~\ref{sec:5.2}, we will show that we can find a triple $\bigl(\tilde{c}_2(t), \tilde{c}_1(t), \tilde{c}_0(t)\bigr)$ such that all \hyperlink{dim 4 cons}{4-dimensional four constraint} will be satisfied. Moreover, for this triple $\bigl(\tilde{c}_2(t), \tilde{c}_1(t), \tilde{c}_0(t)\bigr)$, we have $\tilde{c}_2(0) > 0$, $\tilde{c}_1(0) = 0$, and $\tilde{c}_0(0) \geq 0$.

\hypertarget{R:3.5}{\begin{frmk}}At any point $p \in M$, by picking the coordinates in Lemma~\hyperlink{L:2.2}{2.2}, we can write equation (\ref{eq:3.9}) as
\begin{align*}
\label{eq:3.12}
0 &= \sii \lambda_1 \lambda_2 \lambda_3 \lambda_4  - c_2(t)    \bigl( \lambda_1 \lambda_2 + \lambda_1 \lambda_3 + \lambda_1 \lambda_4 + \lambda_2 \lambda_3 + \lambda_2 \lambda_4 + \lambda_3 \lambda_4 \bigr)    \tag{3.11} \\
&\kern2em +2 c_1(t) \ct  (\lambda_1 + \lambda_2 + \lambda_3 + \lambda_4)    - c_0(t) \cc.
\end{align*}We can rewrite equation (\ref{eq:3.12}) as the following format
\begin{align*}
\label{eq:3.13}
f_t(\lambda_1, \lambda_2, \lambda_3, \lambda_4) =  \frac{c_2(t) \sigma_2(\lambda)  - 2 c_1(t) \cot(\hat{\theta}) \sigma_1(\lambda) + c_0(t) \cc  }{\lambda_1 \lambda_2 \lambda_3 \lambda_4} = \sii, \tag{3.12}   
\end{align*}
where $\lambda_i$ are the eigenvalues of $\omega^{-1}X$ at $p$ and we denote $\lambda =\{ \lambda_1, \lambda_2, \lambda_3, \lambda_4\}$.
\end{frmk}

\hypertarget{R:3.6}{\begin{frmk}}
For convenience, from mow on, we will denote the solution set as $\{f_t = \sii\}$. There might be some ambiguities because this set has more than one connected component. But we pick the one such that the enclosed region contains $(R, R, R, R) \in \mathbb{R}^4$ for $R > 0$ large.
\end{frmk}

For convenience, we abbreviate $c_2 = c_2(t)$, $c_1 = c_1(t)$, and $c_0 = c_0(t)$ unless specify otherwise.

\hypertarget{R:3.7}{\begin{frmk}[Positivstellensatz constraint]}
For equation (\ref{eq:3.12}), we can rewrite it as
\begin{align*}
\kern1em
&\kern-1em  \lambda_{s(1)} \Bigl(  \sii   \lambda_{s(2)} \lambda_{s(3)} \lambda_{s(4)} - c_2     \bigl(\lambda_{s(2)} + \lambda_{s(3)}  + \lambda_{s(4)} \bigr) + 2c_1  \ct    \Bigr) \\
&= c_2  \bigl (\lambda_{s(2)} \lambda_{s(3)} + \lambda_{s(2)} \lambda_{s(4)} + \lambda_{s(3)} \lambda_{s(4)} \bigr ) - 2c_1  \ct \bigl (\lambda_{s(2)}   +   \lambda_{s(3)} +   \lambda_{s(4)} \bigr )  + c_0  \cc,
\end{align*}where $s \in S_4$ is a permutation on four objects and we abbreviate $c_i = c_i(t)$ for $i \in \{0, 1, 2\}$. By the Positivstellensatz Theorem~\hyperlink{T:2.1}{(b)}, we know $\Upsilon^4_{3; 1, 0, -c_2 \cs, 2c_1 \ct \cs} \cap \Upsilon^4_{2; 1, 0, -c_2 \cs} \cap \Upsilon^4_{1; 1, 0}$ is one of the connected component of $\Upsilon^4_{3; 1, 0, -c_2 \cs, 2c_1 \ct \cs}$. To have a nice $\Upsilon$-cone structure, the $C$-subsolution cone must contain the solution set $\{f_t = \sii\}$, so we require both quantities $\lambda_{s(1)}$ and $\sii \lambda_{s(2)} \lambda_{s(3)} \lambda_{s(4)} - c_2      (\lambda_{s(2)} + \lambda_{s(3)}  + \lambda_{s(4)}  ) + 2c_1  \ct  $ to be positive. This is true when the quantity $c_2    (\lambda_{s(2)} \lambda_{s(3)} + \lambda_{s(2)} \lambda_{s(4)} + \lambda_{s(3)} \lambda_{s(4)}   ) - 2c_1  \ct   (\lambda_{s(2)}   +   \lambda_{s(3)} +   \lambda_{s(4)}   )  + c_0  \cc$ is positive (of course, we omit the case when both $\lambda_{s(1)}$ and $ \lambda_{s(2)} \lambda_{s(3)} \lambda_{s(4)} - c_2  \cs   (\lambda_{s(2)} + \lambda_{s(3)}  + \lambda_{s(4)}  ) + 2c_1  \ct \cs$ are negative).\bigskip

$c_2    (\lambda_{s(2)} \lambda_{s(3)} + \lambda_{s(2)} \lambda_{s(4)} + \lambda_{s(3)} \lambda_{s(4)}   ) - 2c_1  \ct   (\lambda_{s(2)}   +   \lambda_{s(3)} +   \lambda_{s(4)}   )  + c_0  \cc$ is always positive if $\Upsilon^4_{3; 1, 0, -c_2 \cs, 2c_1 \ct \cs} \cap \Upsilon^4_{2; 1, 0, -c_2 \cs} \cap \Upsilon^4_{1; 1, 0}$ is contained in \[\Upsilon^4_{3; 0, c_2\cs, -2c_1 \ct\cs, c_0\cs (3\cs-4)} \cap \Upsilon^4_{2; 1, 0, -c_2 \cs} \cap \Upsilon^4_{1; 1, 0}.\] By the Positivstellensatz Theorem~\hyperlink{T:2.1}{(e)}, this statement holds when $c_2^{ {3}/{2}} > -c_1 \coo$ and $\cc c_0  >  -24   {c_2^2 \cs \cos^2(\theta_{c_1, c_2})\cos(2\theta_{c_1, c_2})}$. 




\end{frmk}

\hypertarget{L:3.4}{\begin{flemma}}
A smooth function $\ubar{u} \colon M   \rightarrow \mathbb{R}$ is a $C$-subsolution to equation (\ref{eq:3.13}) if and only if at each point $p \in M$, we have
\begin{align*}
\mu \in  \Upsilon^4_{3; 1, 0, -c_2(t) \cs, 2c_1(t) \ct \cs} \cap \Upsilon^4_{2; 1, 0, -c_2(t) \cs} \cap \Upsilon^4_{1; 1, 0},
\end{align*}where $  \{  \mu    \}$ are the eigenvalues of the Hermitian endomorphism $\omega^{i \bar{k}}    (  X + \sqrt{-1} \partial \bar{\partial} \ubar{u}       )_{j \bar{k}}(p)$. Note that the $\Upsilon$-cones for equation (\ref{eq:3.13}) at $t$ will be $\Upsilon_1(t) =  \Upsilon^4_{3; \sii, 0, -c_2(t)  , 2c_1(t) \ct  } \cap \Upsilon^4_{2; \sii, 0, -c_2(t)  } \cap \Upsilon^4_{1; 1, 0}$, $\Upsilon_2(t) = \Upsilon^4_{2; \sii, 0, -c_2(t)  } \cap \Upsilon^4_{1; 1, 0}$, and $\Upsilon_3(t) = \Upsilon^4_{1; 1, 0}$.
\end{flemma}

\begin{proof}
The proof is similar to the proof of Lemma~\hyperlink{L:3.1}{3.1}, so we will skip the proof here.
\end{proof}

\hypertarget{R:3.8}{\begin{frmk}[$\Upsilon$-cone constraints]} 
Along the continuity path, we hope that the $C$-subsolution condition will be maintained, so we can use this $C$-subsolution to obtain a priori estimates. That is, we are checking whether the $\Upsilon$-cones will be shrinking when $t$ goes from $0$ to $1$. By Lemma~\hyperlink{L:3.4}{3.4}, the $\Upsilon$-cones at $t$ will be 
\begin{align*}
\Upsilon_1(t) &= \Upsilon^4_{3; 1, 0, -c_2(t) \cs, 2c_1(t) \ct \cs} \cap \Upsilon^4_{2; 1, 0, -c_2(t) \cs} \cap \Upsilon^4_{1; 1, 0}; \\
\Upsilon_2(t) &= \Upsilon^4_{2; 1, 0, -c_2(t) \cs} \cap \Upsilon^4_{1; 1, 0};\quad \Upsilon_3(t) = \Upsilon^4_{1; 1, 0}.
\end{align*}If $c_2'(t) > 0$, then for $1 \geq t > t' \geq 0$, we have $\Upsilon_2(1)  \subseteq \Upsilon_2(t)  \subseteq \Upsilon_2(t')   \subseteq \Upsilon_2(0)$. Moreover, if $\frac{d}{dt} \bigl( c_2^{3/2}(t) \bigr) \geq -\cos(\hat{\theta}) c_1'(t)$, then for any $\mu \in \Upsilon_1(t)$, we have
\begin{align*}
\frac{d}{dt} c_2(t) \bigl( \mu_{s(2)} + \mu_{s(3)} + \mu_{s(4)} \bigr) &= c_2'(t)  \bigl( \mu_{s(2)} + \mu_{s(3)} + \mu_{s(4)} \bigr) \\ 
&\geq   c_2'(t) \bigl( \sqrt{\mu_{s(2)} \mu_{s(3)}} + \sqrt{\mu_{s(2)} \mu_{s(4)}} + \sqrt{\mu_{s(3)} \mu_{s(4)}} \bigr) \\
&\geq  -3 c_2^{1/2}(t) c_2'(t) \cso = -2 \frac{d}{dt}  \bigl( c_2^{3/2}(t) \bigr) \cso \\
&= 2  c_1'(t)\ct.
\end{align*}Thus,
\begin{align*}
&\kern-1em \frac{d}{dt} \Bigl( \mu_{s(2)} \mu_{s(3)} \mu_{s(4)} - c_2(t) \cs \bigl( \mu_{s(2)} + \mu_{s(3)} + \mu_{s(4)} \bigr) + 2c_1(t) \ct \cs \Bigr) \\
&= -c_2'(t) \cs \bigl( \mu_{s(2)} + \mu_{s(3)} + \mu_{s(4)} \bigr) + 2  c_1'(t) \ct \cs \leq 0.
\end{align*}Then for $t> t'$, we obtain 
\begin{align*}
0<&\mu_{s(2)} \mu_{s(3)} \mu_{s(4)} - c_2(t) \cs \bigl( \mu_{s(2)} + \mu_{s(3)} + \mu_{s(4)} \bigr) + 2c_1(t) \ct \cs \\
\leq &\mu_{s(2)} \mu_{s(3)} \mu_{s(4)} - c_2(t') \cs \bigl( \mu_{s(2)} + \mu_{s(3)} + \mu_{s(4)} \bigr) + 2c_1(t') \ct \cs.
\end{align*}This implies that $\mu \in \Upsilon_1(t')$. The constraint that $\frac{d}{dt} \bigl( c_2^{3/2}(t) \bigr) \geq -\cos(\hat{\theta}) c_1'(t)$ also implies that 
\begin{align*}
\int_0^{s}  \frac{d}{dt} \bigl( c_2^{3/2}(t) \bigr) dt \geq - \int_0^{s} \cos(\hat{\theta}) c_1'(t) dt &   \Longrightarrow  c_2^{3/2}(s) \geq c_2^{3/2}(0) -  \coo c_1(s).
\end{align*}So if we assume $c_2(0) > 0$ and $\frac{d}{dt} \bigl( c_2^{3/2}(t) \bigr) \geq -\cos(\hat{\theta}) c_1'(t)$, then for the \hyperlink{dim 4 cons}{Positivstellensatz constraint}, we do not need to assume $c_2^{3/2}(t) > -\coo c_1(t)$.
\end{frmk}


Now with a triple $(c_2(t), c_1(t), c_0(t))$ satisfying the \hyperlink{dim 4 cons}{4-dimensional four constraints}, we consider a more general equation 
\begin{align*}
\label{eq:3.14}
f_t(\lambda_1, \lambda_2, \lambda_3, \lambda_4) =  \frac{c_2(t) \sigma_2(\lambda)  - 2 c_1(t) \cot(\hat{\theta}) \sigma_1(\lambda) + c_0(t) \cc  }{\lambda_1 \lambda_2 \lambda_3 \lambda_4}  = h(p),   \tag{3.13}
\end{align*}where $\lambda_i$ are the eigenvalues of $\omega^{-1}X_u$ at $p$ and $h \colon M   \rightarrow \mathbb{R}$ is a smooth function satisfying $c_2^3 (t) \taa/ c_1^2(t) > h(p) > 0$ and \[c_0(t) \cc h(p) + 24 c_2^2(t) \cos^2(\theta_{h(p), c_1(t), c_2(t)})\cos(2\theta_{h(p), c_1(t), c_2(t)}) > 0\] for all $p \in M$. Here $\theta_{h(p), c_1(t), c_2(t)} \coloneqq \arccos \bigl(   {- c_1(t) \ct h^{1/2}(p) }\big/{c_2^{3/2}(t) }  \bigr)\big/3  - {2\pi}/{3}$ and we specify the branch so that $ \arccos  (  \bullet ) \in  ( \pi, 3\pi/2  ]$. We will explain why we restrict the range of $h$ in Section~\ref{sec:4.2}. Roughly speaking, we want to preserve the $\Upsilon$-cone structure. In particular, when $h = \sii$, equation (\ref{eq:3.14}) is the original equation (\ref{eq:3.13}). The proof of the following Lemma is similarly to the proof of Lemma~\hyperlink{L:3.4}{3.4}, the $C$-subsolution to equation (\ref{eq:3.14}) can be described as follows.
\hypertarget{L:3.5}{\begin{flemma}}
A smooth function $\ubar{u} \colon M   \rightarrow \mathbb{R}$ is a $C$-subsolution to equation (\ref{eq:3.14}) if and only if at each point $p \in M$, we have
\begin{align*}
\mu \in  \Upsilon^4_{3; 1, 0, -c_2(t) /h(p), 2c_1(t) \ct /h(p)} \cap \Upsilon^4_{2; 1, 0, -c_2(t) /h(p)} \cap \Upsilon^4_{1; 1, 0},
\end{align*}where $  \{  \mu    \}$ are the eigenvalues of the Hermitian endomorphism $\omega^{i \bar{k}}    (  X + \sqrt{-1} \partial \bar{\partial} \ubar{u}       )_{j \bar{k}}(p)$. Note that the $\Upsilon$-cones for equation (\ref{eq:3.14}) at $t$ at point $p$ will be $\Upsilon_1(t, p) = \Upsilon^4_{3; h(p), 0, -c_2(t), 2c_1(t) \ct } \cap \Upsilon^4_{2; h(p), 0, -c_2(t)} \cap \Upsilon^4_{1; 1, 0}$, $\Upsilon_2(t, p) = \Upsilon^4_{2; h(p), 0, -c_2(t)} \cap \Upsilon^4_{1; 1, 0}$, and $\Upsilon_3(t, p) =  \Upsilon^4_{1; 1, 0}$.
\end{flemma}

Now, if $c_2^3(t) \taa/ c_1^2(t) > h > 0$ and \[c_0(t) \cc h + 24 c_2^2(t) \cos^2(\theta_{h, c_1(t), c_2(t)})\cos(2\theta_{h, c_1(t), c_2(t)}) > 0\] for all $t \in [0, 1]$. Here  $\theta_{h, c_1(t), c_2(t)} \coloneqq \arccos \bigl(   {- c_1(t) \ct h^{1/2} }\big/{c_2^{3/2}(t) }  \bigr)\big/3  - {2\pi}/{3}$ and we specify the branch so that $ \arccos  (  \bullet   ) \in  ( \pi, 3\pi/2  ]$. Then the $\Upsilon$-cone constraint will automatically be satisfied. So we have the following Lemma.

\hypertarget{L:3.6}{\begin{flemma}}
For $p \in M$, if $c_2^3(t) \taa/ c_1^2(t) > h(p) > 0$ and \[c_0(t) \cc h(p) + 24 c_2^2(t) \cos^2(\theta_{h(p), c_1(t), c_2(t)})\cos(2\theta_{h(p), c_1(t), c_2(t)}) > 0\] for all $t \in [0, 1]$. Here  $\theta_{h(p), c_1(t), c_2(t)} \coloneqq \arccos \bigl(   {- c_1(t) \ct h^{1/2}(p) }\big/{c_2^{3/2}(t) }  \bigr)\big/3  - {2\pi}/{3}$ and we specify the branch so that $ \arccos  (  \bullet ) \in \bigl( \pi, 3\pi/2 \bigr]$. Then 
\begin{align*}
 \{ f_1 = h(p) \} \subset \Upsilon_1(t, p) \subset \Upsilon_1(t', p);\quad  \Upsilon_2(t, p) \subset \Upsilon_2(t', p);\quad \text{and } \Upsilon_3(t, p) \subset \Upsilon_3(t', p)
\end{align*}for $1 \geq t > t' \geq 0$. Here, the solution set $\{ f_t = h(p) \}$ has some ambiguities, because it has more than one connected component. The workaround is the same as Remark~\hyperlink{R:3.6}{3.6}.
\end{flemma}

\section{A Priori Estimate}
\label{sec:4}
The original a priori estimates for the dHYM equation on Kähler manifolds is due to Collins--Jacob--Yau \cite{collins20151} (See also Lin \cite{lin2020} for Hermitian manifolds). Here, by making a substitution, the dHYM equation has a different format (See equations (\ref{eq:3.1}) and (\ref{eq:3.7})). Moreover, we find a continuity path to connect the dHYM equation to a general inverse $\sigma_k$-type equation (See equations (\ref{eq:3.3}) and (\ref{eq:3.9})). The dHYM equation and general inverse $\sigma_k$-type equation are different equations, and indeeed they have different $C$-subsolution cone and $\Upsilon$-cones. Notice that a priori estimates depend on $C$-subsolution hence depend on the $C$-subsolution cone as well. So we need to obtain a priori estimates which can be applied all over the continuity path we choose.\bigskip

First, let us summarize the proof of a priori estimates. Under the assumption of $C$-subsolution and $\Upsilon$-cones, we apply the Alexandroff--Bakelman--Pucci estimate. The $C^0$ estimate can be obtained following the proof in Székelyhidi \cite{szekelyhidi2018fully}, which is based on the method that Błocki \cite{blocki2005uniform,blocki2011uniform} used in the case of the complex Monge--Ampère equation. We will skip the proof of the $C^0$ estimate because it follows verbatim.\bigskip

Second, we use the maximum principle to obtain that the $C^2$ norm can be bounded by the $C^1$ norm. The method is inspired by Hou--Ma--Wu \cite{hou2010second} for the complex Hessian equations and used by Székelyhidi \cite{szekelyhidi2018fully}. The interested reader is referred to \cite{collins20151, szekelyhidi2018fully} and the references therein. Once we have the above type inequality, by a blow-up argument due to Dinew--Kołodziej \cite{dinew2017liouville}, we can get an indirect $C^1$ estimate. \bigskip



Last, to get $C^{2, \alpha}$ estimate, we follow the proof of the complex version of the Evans--Krylov theory in Siu \cite{siu2012lectures}, we can exploit the convexity of the solution sets to obtain a $C^{2, \alpha}$ estimates by a blow-up argument. Furthermore, for the higher regularity, we apply the standard Schauder estimates and bootstrapping.

\subsection{When \texorpdfstring{$n=3$}{}}
\label{sec:4.1}
In this subsection, first, we always assume that $\hat{\theta} \in \bigl( \pi/2, 5 \pi/6    \bigr)$ and there exists a $C$-subsolution $\ubar{u} \colon M   \rightarrow \mathbb{R}$. We also call $X_{\ubar{u}}$ this $C$-subsolution and by changing representative, we may assume $X$ is this $C$-subsolution. We also abbreviate $\lambda = \{\lambda_1, \lambda_2, \lambda_3\}$ and we always assume $\lambda_1 \geq \lambda_2 \geq \lambda_3$ unless further notice. Most of the time, to save spaces, we will abbreviate $f = f_t =  f_t(\lambda)$, $f_i = \partial f /\partial \lambda_i, f_{ij} = \partial^2 f/\partial \lambda_i \partial \lambda_j$ for $i, j \in \{1, 2, 3\}$ for notational convention. Last, unless specify otherwise, we always abbreviate $c_1 = c_1(t)$ and $c_0 = c_0(t)$. We assume $(c_1, c_0)$ satisfy all the \hyperlink{dim 3 cons}{3-dimensional four constraints} in Section~\ref{sec:3.1} and consider equation (\ref{eq:3.5})
\begin{align*}
\label{eq:4.1}
f_t(  {\lambda}_1,   {\lambda}_2,   {\lambda}_3) =  \frac{c_1(t) \sigma_1( {\lambda}) + 2 c_0(t) \tan(\hat{\theta}) }{  {\lambda}_1   {\lambda}_2   {\lambda}_3} = h, \tag{4.1}
\end{align*}where $\lambda_i$ are the eigenvalues of $\omega^{-1}X_u$ and $h = \css$ is a constant. By the \hyperlink{dim 3 cons}{Positivstellensatz constraint}, it automatically satisfies $c_1^3(t) \ctt/ c_0^2(t) > h > 0$. 

\begin{frmk}
The author thinks that a priori estimates hold for smooth function $h \colon M   \rightarrow \mathbb{R}$ satisfying $c_1^3(t) \ctt/ c_0^2(t) > h > 0$. The difficulty is that if $h$ is not a constant, then the vector $(X_u)_{i \bar{i}, k}$ in equation (\ref{eq:4.14}) will not be a vector on the tangent plane of $\{ f = h\}$. Thus we can not apply the convexity of the solution set to get $\sum_{i, j} f_{ij}   (  X_u     )_{i \bar{i}, \bar{k}}    (  X_u     )_{j \bar{j}, {k}}  \geq 0$.
\end{frmk}

\subsubsection{The $C^2$ Estimates}
\label{sec:4.1.1}
Define a Hermitian endomorphism $\Lambda \coloneqq \omega^{-1}X_u$, where $X_u = X + \sqrt{-1} \partial \bar{\partial} u$, and let $\lambda = \{ \lambda_1, \lambda_2, \lambda_3\}$ be the eigenvalues of $\Lambda$. We consider the following function $G(\Lambda) = \log(1 + \lambda_1) =g(\lambda_1, \lambda_2, \lambda_3)$ and the following test function
\begin{align*}
U \coloneqq - Au + G(\Lambda), 
\end{align*}where $A \gg 0$ will be determined later. We want to apply the maximum principle to $U$, but since the eigenvalues of $\Lambda$ might not be distinct at the maximum point $p \in M$ of $U$, we do a perturbation here. The perturbation here, though not necessarily, is made to preserve the $\Upsilon$-cone structure for convenience. Assume $\lambda_1$ is large, otherwise we are done, then 

\begin{itemize}
\hypertarget{pert for dim 3}{\item} we pick the constant matrix $B$ to be a diagonal matrix with real entries 
\begin{align*}
B_{11} = \epsilon;\quad B_{22} = \epsilon/2;\quad B_{33} = 0
\end{align*}
such that $\tilde{\lambda}_i = \lambda_i + B_{ii}$ with $\lambda_1 + \epsilon = \tilde{\lambda}_1 > \tilde{\lambda}_2  = \lambda_2 + \epsilon/2> \tilde{\lambda}_3 = \lambda_3 > 0$ and assume $\epsilon > 0$ is sufficiently small.
\end{itemize}

By defining $\tilde{\Lambda} = \Lambda + B$, then $\tilde{\Lambda}$ has distinct eigenvalues near $p \in M$, which are $\{\tilde{\lambda}_1, \tilde{\lambda}_2, \tilde{\lambda}_3\}$. The eigenvalues of $\tilde{\Lambda}$ define smooth functions near the maximum point $p$. And we can check $p$ is still the maximum point of the following locally defined test function
\begin{align*}
\label{eq:4.2}
\tilde{U} \coloneqq -Au + G(\tilde{\Lambda}). \tag{4.2}
\end{align*}


Near the maximum point $p$ of $\tilde{U}$, we always use the coordinates in Lemma~\hyperlink{L:2.4}{2.4} unless otherwise noted. We instantly get the following lemma. 

\hypertarget{L:4.1}{\begin{flemma}}
At the maximum point $p$ of $\tilde{U}$, by taking the first derivative of $\tilde{U}$ at $p$, we get
\begin{align*}
\label{eq:4.3}
0 &= -A u_k(p) + \frac{1}{1+ \tilde{\lambda}_1} ( X_u )_{1\bar{1}, k}, \tag{4.3}
\end{align*}where we denote $u_k = \partial u/\partial z_k$ and $( X_u )_{1\bar{1}, k} = \partial (X_u)_{1 \bar{1}} / \partial z_k$.
\end{flemma}

\begin{proof}
First, since $\tilde{U} = -Au + G(\tilde{\Lambda})$, if we take the first derivative, we obtain
\begin{align*}
\frac{\partial}{\partial z_k} \tilde{U} &= - A u_k + \frac{\partial G}{\partial \Lambda^j_i} (\tilde{\Lambda}) \frac{\partial \tilde{\Lambda}^j_i}{\partial z_k}.
\end{align*}By Lemma~\hyperlink{L:2.1}{2.1}, Lemma~\hyperlink{L:2.2}{2.2}, and Lemma~\hyperlink{L:2.4}{2.4}, at the maximum point $p \in M$, we have
\begin{align*}
0 = - A u_k(p) + \frac{1}{1+ \tilde{\lambda}_1} ( X_u )_{1\bar{1}, k}.
\end{align*}This finishes the proof.
\end{proof}

We may define the following operator \hypertarget{dim 3 operator}{$\mathcal{L}_t$} by
\begin{align*}
\label{eq:4.4}
\mathcal{L}_t \coloneqq - \sum_{i, j, k} \frac{\partial F_t}{\partial \Lambda^k_i} ( {\Lambda}) \omega^{k \bar{j}}  \frac{\partial^2}{\partial z_i \partial \bar{z}_j}, \tag{4.4}
\end{align*}where $F_t = F_t( {\Lambda}) = f_t( {\lambda}_1,  {\lambda}_2,  {\lambda}_3)$ is defined by $f_t( {\lambda}) =  \bigl( c_1(t) \sigma_1( {\lambda}) + 2c_0(t) \ta \bigr)/  {\lambda}_1   {\lambda}_2   {\lambda}_3$. We immediately have the following Lemmas.

\hypertarget{L:4.2}{\begin{flemma}}
By taking $f_t( {\lambda}) =  \bigl( c_1(t) \sigma_1( {\lambda}) + 2c_0(t) \ta \bigr)/  {\lambda}_1   {\lambda}_2   {\lambda}_3$ and $g( {\lambda})=\log(  1 +  {\lambda}_1)$, we have 
\begin{align*}
f_i &= \frac{    - c_1  \sigma_1(  {\lambda}_{;i}) - 2c_0  \tan(\hat{\theta})}{   {\lambda}_1   {\lambda}_2   {\lambda}_3   {\lambda}_i }; &f_{ij}& = \frac{ -c_1  (  {\lambda}_i +   {\lambda}_j) + \bigl [  c_1 \sigma_1( {\lambda})  + 2c_0  \tan(\hat{\theta})  \bigr ] (1 + \delta_{ij})}{  {\lambda}_1   {\lambda}_2   {\lambda}_3   {\lambda}_i   {\lambda}_j}; \\
g_i &= \delta_{1i} \frac{1}{ 1 +   {\lambda}_1}; &g_{ij}&=-\delta_{1i}\delta_{1j}\frac{1}{( 1 +  {\lambda}_1)^2}. 
\end{align*}Here, $\lambda_{;i}$ means we exclude $\lambda_i$ from $\lambda = \{ \lambda_1, \lambda_2, \lambda_3\}$ and we denote $f_i \coloneqq  {\partial  f_t}/{\partial \lambda_i}, \, g_i \coloneqq  {\partial g}/{\partial \lambda_i}$, $(f_t)_{ij} \coloneqq  {\partial^2 f_t}/{\partial \lambda_i \partial \lambda_j}$, and $g_{ij} \coloneqq  {\partial^2 g}/{\partial \lambda_i \partial \lambda_j}.$
\end{flemma}

\hypertarget{L:4.3}{\begin{flemma}}
If $c_1^3 \ctt/ c_0^2 > h > 0$, then for any point on the solution set $\{ f_t = h\}$, we have
\begin{align*}
-f_i =  \frac{    c_1  \sigma_1(  {\lambda}_{;i}) + 2c_0  \tan(\hat{\theta})}{   {\lambda}_1   {\lambda}_2   {\lambda}_3   {\lambda}_i } > 0
\end{align*}for any $i \in \{1, 2, 3\}$ at this point. Here $f_i = \partial f_t / \partial \lambda_i$, where $i \in \{1, 2, 3\}$.
\end{flemma}

\begin{proof}
By the Positivstellensatz Theorem~\hyperlink{T:2.1}{(d)}, for $c > 0$, $d \geq 0$ and $2 c^{3/2} > d$, then $\Upsilon^3_{2; 1, 0, -c} \cap \Upsilon^3_{1; 1, 0}$ is contained in $\Upsilon^3_{2; 0, c, -d} \cap \Upsilon^3_{1; 1, 0}$. By letting $c = c_1/h$ and $d = -2c_0 \ta/h$, if $2 c^{3/2} > d$, then 
\begin{align*}
0< c \sigma_1(  {\lambda}_{; i}) -d = c_1 \sigma_1(  {\lambda}_{; i}) /h + 2 c_0 \ta/h 
\end{align*}for all $i \in \{1, 2, 3\}$. By checking $2 c^{3/2} - d$, we get
\begin{align*}
2 c^{3/2} - d = \frac{2 c_1^{3/2}}{h^{3/2}} + \frac{2c_0 \ta}{h} =  \frac{2}{h^{3/2}} \bigl( c_1^{3/2}  + c_0 \ta h^{1/2} \bigr) > 0
\end{align*}by our hypothesis. This finishes the proof.
\end{proof}

With Lemma~\hyperlink{L:4.3}{4.3}, if we assume $\bigl(c_1(t), c_0(t)\bigr)$ satisfy all \hyperlink{dim 3 cons}{3-dimensional four constraints}, then we get that the operator \hyperlink{dim 3 operator}{$\mathcal{L}_t $} is indeed an elliptic operator on the solution set $\{ f_t = h\}$. Here $c_1^3 \ctt/ c_0^2 > h > 0$.

\hypertarget{L:4.4}{\begin{flemma}}
If $c_1^3 \ctt/ c_0^2 > h > 0$, then the solution set $\{ f_t = h\}$ is convex.  
\end{flemma}


\begin{proof}
Let $V = (V_1, V_2, V_3) \in T_{ {\lambda}} \bigl\{ f_t = h   \bigr\}$ be a tangent vector, which gives, $\sum_i f_i V_i =0$. Then we are trying to show that the following quantity 
\begin{align*}
\label{eq:4.5}
\sum_{i, j} f_{ij} V_i  {V}_{\bar{j}}  \tag{4.5}
\end{align*}is positive. First, since $V$ is a tangent vector, we can write $V_3 = - \bigl({f_1 V_1 + f_2 V_2 }\bigr)/{f_3}$. By plugging in quantity (\ref{eq:4.5}), we obtain
\begin{align*}
\label{eq:4.6}
\sum_{i, j} f_{ij} V_i V_{\bar{j}} &= f_{11} |V_1|^2 + f_{22} |V_2|^2  + f_{33} |V_3|^2 \tag{4.6} \\
&\kern2em + f_{12} \bigl( V_{1} V_{\bar{2}}  + V_{\bar{1}} V_{{2}}  \bigr)  + f_{13} \bigl( V_{1} V_{\bar{3}}  + V_{\bar{1}} V_{{3}}  \bigr)  + + f_{23} \bigl( V_{2} V_{\bar{3}}  + V_{\bar{2}} V_{{3}}  \bigr)  \\
&= \Bigl(  f_{11} + f_{33} \frac{f_{1}^2}{f_{3}^2}  -2 f_{13} \frac{f_{1}}{f_{3}}  \Bigr) |V_1|^2 + \Bigl(  f_{22} + f_{33} \frac{f_{2}^2}{f_{3}^2}  -2 f_{23} \frac{f_{2}}{f_{3}}  \Bigr) |V_2|^2 \\
&\kern2em + \Bigl (  f_{12}  + f_{33} \frac{f_{1} f_{2}}{f_{3}^2}  - f_{13} \frac{f_{2}}{f_{3}}  - f_{23} \frac{f_{1}}{f_{3}}  \Bigr) \bigl( V_1 V_{\bar{2}} + V_{\bar{1}} V_2  \bigr) \\
&= 2\frac{c_1 \sigma_1(  {\lambda}_{;1}) + 2c_0  \ta}{  {\lambda}_1^3   {\lambda}_2   {\lambda}_3} \frac{c_1 \sigma_1( {\lambda}) + 2c_0 \ta}{ c_1 \sigma_1(  {\lambda}_{;3}) + 2c_0  \ta} |V_1|^2 \\
&\kern2em + 2\frac{c_1 \sigma_1(  {\lambda}_{;2}) + 2c_0  \ta}{  {\lambda}_1   {\lambda}_2^3   {\lambda}_3}\frac{c_1 \sigma_1( {\lambda}) + 2c_0 \ta}{ c_1 \sigma_1(  {\lambda}_{;3}) + 2c_0  \ta} |V_2|^2 \\
&\kern2em + 2\frac{c_1    {\lambda}_3 + c_0  \ta}{  {\lambda}_1^2   {\lambda}_2^2   {\lambda}_3} \frac{c_1 \sigma_1( {\lambda}) + 2c_0 \ta}{ c_1 \sigma_1(  {\lambda}_{;3}) + 2c_0  \ta} \bigl( V_1 V_{\bar{2}} + V_{\bar{1}} V_2  \bigr).
\end{align*}By checking the discriminant of above quadratic form (\ref{eq:4.6}), we have
\begin{align*}
&\kern-1em \Bigl ( 2\frac{c_1    {\lambda}_3 + c_0  \ta}{  {\lambda}_1^2   {\lambda}_2^2   {\lambda}_3}   \Bigr )^2 - 2\frac{c_1 \sigma_1(  {\lambda}_{;1}) + 2c_0  \ta}{  {\lambda}_1^3   {\lambda}_2   {\lambda}_3} \cdot 2\frac{c_1 \sigma_1(  {\lambda}_{;2}) + 2c_0  \ta}{  {\lambda}_1   {\lambda}_2^3   {\lambda}_3} \\
&= \frac{-4}{  {\lambda}_1^4   {\lambda}_2^4   {\lambda}_3^2} \bigl (  c_1^2 \sigma_2( {\lambda}) + 2c_0  c_1  \ta \sigma_1( {\lambda})  + 3c_0^2 \tan^2 (\hat{\theta} )    \bigr).
\end{align*}If $c_1^2 \sigma_2( {\lambda}) + 2c_0  c_1  \ta \sigma_1( {\lambda})  + 3c_0^2 \tan^2 (\hat{\theta} )  \geq 0$, then quantity (\ref{eq:4.5}) will be non-negative, which proves that the solution set $\{ f_t = h\}$ is convex.\bigskip


%
%
%
%


By the \hyperlink{T:2.1}{Positivstellensatz Theorem}, assume $c > 0$, $2 c^{3/2} > d$, and $e > -24 c^2 \cos^2(\theta_{c, d})\cos(2\theta_{c, d})$. Then $\Upsilon_{3; 1, 0, -c, d} \cap \Upsilon_{2; 1, 0, -c} \cap \Upsilon_{1; 1, 0}$ is contained in $\Upsilon_{3; 0, c, -d, e} \cap \Upsilon_{2; 1, 0, -c} \cap \Upsilon_{1; 1, 0}$. Here, by setting $c = c_1/h$, $d= -2c_0 \ta/h$, and $e =  {3c_0^2  \tan^2 (\hat{\theta} )}/{(hc_1)}$, if the above statement holds, then 
\begin{align*}
c_1\sigma_2( {\lambda})/h + 2c_0 \ta \sigma_1( {\lambda})/h  + 3c_0^2 \tan^2 (\hat{\theta} )/(hc_1)  > 0,
\end{align*}which implies that $c_1^2 \sigma_2( {\lambda}) + 2c_0 c_1 \ta \sigma_1( {\lambda}) + 3 c_0^2 \ta > 0$.\bigskip

Thus, we check whether $e + 24 c^2 \cos^2(\theta_{c, d})\cos(2\theta_{c, d})$ is positive, which is equivalent to whether
\begin{align*}
h c_0^2 \taa + 8 c_1^3 \cos^2(\theta_{h, c_0, c_1}) \cos(2\theta_{h, c_0, c_1})
\end{align*}is positive. Here $\theta_{h, c_0, c_1} \coloneqq \arccos \bigl(   {c_0 \ta h^{1/2} }\big/{c_1^{3/2} }  \bigr)\big/3  - {2\pi}/{3}$ and we specify the branch so that $\arccos  ( \bullet )  \in  ( \pi, 3\pi/2  ]$. We can write 
\begin{align*}
&\kern-1em h c_0^2 \taa + 8 c_1^3 \cos^2(\theta_{h, c_0, c_1}) \cos(2\theta_{h, c_0, c_1}) \\
&= h c_0^2 \taa + 4c_1^3 \cos(\theta_{h, c_0, c_1}) \bigl(    \cos(\theta_{h, c_0, c_1}) + \cos(3\theta_{h, c_0, c_1}) \bigr) \\
&=  h c_0^2 \taa + 4c_1^3 \cos(\theta_{h, c_0, c_1}) \bigl( \cos(\theta_{h, c_0, c_1}) + c_0\ta h^{1/2}c_1^{-3/2} \bigr) \\
&= \bigl( 2 c_1^{3/2}  \cos(\theta_{h, c_0, c_1})   + h^{1/2} c_0 \ta \bigr)^2 > 0.
\end{align*}The above inequality holds due to the choice of branch, so $\cos(\theta_{h, c_0, c_1}) > 1/2$. Also, by our hypothesis, we have $c_1^{3/2}      + h^{1/2} c_0 \ta > 0$. Thus the quantity will always be positive, this finishes the proof.
\end{proof}

Now, by taking the first and second derivatives of equation (\ref{eq:4.1}), we have the following Lemma. The proof should be straightforward; we apply Lemma~\hyperlink{L:2.1}{2.1}, Lemma~\hyperlink{L:2.2}{2.2}, Lemma~\hyperlink{L:2.3}{2.3}, Lemma~\hyperlink{L:2.4}{2.4}, and Lemma~\hyperlink{L:4.2}{4.2}. Or one can check the following reference Lin \cite{lin2020} for more details.

\hypertarget{L:4.5}{\begin{flemma}}
Let $F_t(  {\Lambda} )= {h}(p)$, then we have
\begin{align*}
 \frac{\partial {h}}{\partial  {z}_k}  = \sum_{i ,j}  \frac{\partial F_t (   {\Lambda}     ) }{\partial  {\Lambda}^j_i}    \frac{\partial  {\Lambda}^j_i }{\partial  {z}_k}; \quad  \frac{\partial^2 {h}}{\partial z_k \partial \bar{z}_k} =  \sum_{i ,j}  \frac{\partial^2 F_t (   {\Lambda}     ) }{\partial \Lambda^j_i \partial \Lambda_r^s}    \frac{\partial  {\Lambda}_i^j }{\partial \bar{z}_k} \frac{\partial  {\Lambda}_r^s}{\partial z_k} + \sum_{i ,j}  \frac{\partial F_t (   {\Lambda}     )}{\partial \Lambda_i^j }    \frac{\partial^2  {\Lambda}_i^j }{\partial z_k \partial \bar{z}_k}.
\end{align*}In particular, at the maximum point $p \in M$ of $\tilde{U}$, we have
\begin{align*}
\label{eq:4.7}
h_k &=  \sum_i f_i   (  X_u     )_{i \bar{i},  {k}};     \tag{4.7}   \\
\label{eq:4.8}
h_{k \bar{k}} &= \sum_{i, j}  f_{ij}   (  X_u     )_{i \bar{i}, \bar{k}}    (  X_u     )_{j \bar{j}, {k}} + \sum_{i \neq j}  \frac{f_i - f_j}{\lambda_i - \lambda_j}    | (X_u)_{j \bar{i},k}    |^2     +  \sum_i \Bigl( f_i  ( X_u   )_{i \bar{i},k \bar{k}} -f_i \lambda_i \omega_{i \bar{i},k \bar{k}} \Bigr) \tag{4.8}  \\
&=\sum_{i, j}  f_{ij}   (  X_u     )_{i \bar{i}, \bar{k}}    (  X_u     )_{j \bar{j}, {k}} + \sum_{i \neq j}  \frac{ h }{\lambda_i \lambda_j}   | (X_u)_{j \bar{i},k}    |^2    +  \sum_i \Bigl( f_i  ( X_u   )_{i \bar{i},k \bar{k}} -f_i  {\lambda}_i \omega_{i \bar{i},k \bar{k}} \Bigr). 
\end{align*}
\end{flemma}

\begin{proof}
The first and second derivatives should be straightforward. At the maximum point, suppose the eigenvalues are pairwise distinct satisfying $\lambda_1 > \lambda_2 > \lambda_3$. Since $\Lambda$ is a diagonal matrix, then
\begin{align*}
h_k = \sum_{i, j} \frac{\partial F_t (   {\Lambda}     ) }{\partial  {\Lambda}^j_i}    \frac{\partial  {\Lambda}^j_i }{\partial  {z}_k} = \sum_i f_i (  X_u     )_{i \bar{i},  {k}}. 
\end{align*}This is also true when the eigenvalues are not pairwise distinct. For the second derivative, if the eigenvalues at the maximum point $p$ are pairwise distinct, then 
\begin{align*}
h_{k \bar{k}}
&= \sum_{i, j}  f_{ij}   (  X_u     )_{i \bar{i}, \bar{k}}    (  X_u     )_{j \bar{j}, {k}} + \sum_{i \neq j}  \frac{f_i - f_j}{ {\lambda}_i -  {\lambda}_j}    | (X_u)_{j \bar{i},k}    |^2     +  \sum_i \Bigl( f_i  ( X_u   )_{i \bar{i},k \bar{k}} -f_i  {\lambda}_i \omega_{i \bar{i},k \bar{k}} \Bigr) \\
&= \sum_{i, j}  f_{ij}   (  X_u     )_{i \bar{i}, \bar{k}}    (  X_u     )_{j \bar{j}, {k}} + \sum_{i \neq j}  \frac{c_1 \sigma_1(\lambda) +2c_0 \ta}{\lambda_1 \lambda_2 \lambda_3 \lambda_i \lambda_j}   | (X_u)_{j \bar{i},k}    |^2     +  \sum_i f_i \Bigl(   ( X_u   )_{i \bar{i},k \bar{k}} -   {\lambda}_i \omega_{i \bar{i},k \bar{k}} \Bigr).
\end{align*}This is also true when the eigenvalues are not pairwise distinct.
\end{proof}

For the remainder of this subsection, we let $O_i$ be the Big $O$ notation that describes the limiting behavior when $\lambda_I$ approaches infinity. So $O_i(1)$ means the quantity will be bounded by a uniform constant if $\lambda_i$ is sufficiently large. 

\hypertarget{L:4.6}{\begin{flemma}}
There exists uniform constants $N > 0$ and $\kappa > 0$, which are independent of $t \in [0, 1]$, such that if $\lambda_1 > N$,  
\begin{align*}
\sum_i f_i u_{i \bar{i}} \geq -\kappa \sum_i f_i.
\end{align*}
\end{flemma}

\begin{proof}
First, by Lemma~\hyperlink{L:3.2}{3.2}, if $X$ is a $C$-subsolution to equation (\ref{eq:4.1}), then $h X^2 - \omega^2 > 0$, where $h$ is a constant satisfying $c_1^3 \ctt/ c_0^2 > h > 0$. We may fix $\delta > 0$ sufficiently small such that
\begin{align*}
(1-\delta)hX^2 - \omega^2  > 0 \text{ for all } p \in M.
\end{align*}
Now, unless further notice, the following inequalities hold for any point $p \in M $. Choose $\kappa > 0$ to be also small enough so that 
\begin{align*}
(1-\delta) h \bigl ( X - \kappa \omega \bigr )^2 > \omega^2; \quad X - \kappa \omega > 0.
\end{align*}Due to our \hyperlink{dim 3 cons}{Boundary constraints} that $c_1$ is increasing and $c_1(0) > 0$, for $t \in [0, 1]$ we have
\begin{align*}
\label{eq:4.9}
(1-\delta)h \bigl ( X - \kappa \omega \bigr )^2 >   \omega^2 \geq  c_1(t)  \omega^2. \tag{4.9}
\end{align*}
%

Note that $u_{ i \bar{i}} = \lambda_i - X_{i \bar{i}}$. We can write

\begin{align*}
\label{eq:4.10}
\sum_i  f_i   ( u_{i \bar{i}}  + \kappa     ) &= \sum_i f_i  \bigl (  {\lambda}_i - X_{i \bar{i}}    + \kappa   \bigr ) = \sum_i  \frac{- c_1 \sigma_1(  {\lambda}_{;i}) - 2 c_0 \ta }{   {\lambda}_1   {\lambda}_2   {\lambda}_3     {\lambda}_i }     (   {\lambda}_i - X_{i \bar{i}}     + \kappa     ) \tag{4.10} \\
&= -\frac{ \sum_i A_i }{  {\lambda}_1   {\lambda}_2   {\lambda}_3  }  +  \sum_i   ( X_{i \bar{i}}   - \kappa     ) \frac{ A_i }{   {\lambda}_1   {\lambda}_2   {\lambda}_3     {\lambda}_i },
\end{align*}where we denote $A_i = c_1 \sigma_1( {\lambda}_{;i})  + 2 c_0 \ta$ for $i \in \{1, 2, 3\}$. Since $  {\lambda}_1 \geq   {\lambda}_2 \geq   {\lambda}_3$, we have $A_3 \geq A_2 \geq A_1 > 0$. \bigskip



There are two cases to be considered:\bigskip

$\bullet$ If $0 <   {\lambda}_3 \leq \frac{ X_{3 \bar{3}} - \kappa}{3}$, then 
\begin{align*}
( X_{3 \bar{3}} - \kappa     ) \frac{ A_3 }{   {\lambda}_1   {\lambda}_2   {\lambda}_3   {\lambda}_3 } \geq 3 \frac{ A_3}{  {\lambda}_1   {\lambda}_2   {\lambda}_3 }.
\end{align*}Hence inequality (\ref{eq:4.10}) becomes
\begin{align*}
\sum_i f_i   ( u_{i \bar{i}}  + \kappa     ) &\geq -3 \frac{ A_3 }{  {\lambda}_1   {\lambda}_2   {\lambda}_3 }   +  \sum_i   ( X_{i \bar{i}} - \kappa     ) \frac{ A_i }{   {\lambda}_1   {\lambda}_2   {\lambda}_3    {\lambda}_i }   \geq 0.
\end{align*}

$\bullet$ If $  {\lambda}_3 \geq \frac{ X_{3 \bar{3}} - \kappa}{3}$, similar as the proof of Lemma~\hyperlink{L:3.1}{3.1}, then we can show that $  {\lambda}_2$ is bounded from above. To be more precise, we get
\begin{align*}
\frac{6c_1 }{h(X_{3 \bar{3}} - \kappa)} \geq       {\lambda}_2 \geq   {\lambda}_3 \geq \frac{ X_{3 \bar{3}} - \kappa}{3} > 0.
\end{align*}With this, we can do a better estimate for $  {\lambda}_2   {\lambda}_3$. We have
\begin{align*}
  {\lambda}_2   {\lambda}_3 = \frac{c_1(  {\lambda}_1  {\lambda}_2 +   {\lambda}_2^2)  + 2c_0 \tan(\hat{\theta})    {\lambda}_2 }{  h  {\lambda}_1   {\lambda}_2 - c_1   } \xrightarrow[\text{ as }   {\lambda}_1   \rightarrow \infty]{} \frac{c_1}{h},
\end{align*}since $  {\lambda}_2$ and $  {\lambda}_3$ are both bounded from above and below. Thus for $  {\lambda}_1$ sufficiently large, we get
\begin{align*}
\label{eq:4.11}
\sqrt{  {\lambda}_2   {\lambda}_3} < (1 +  {\delta}/{4}) {c^{1/2}_1 } h^{-1/2}. \tag{4.11}
\end{align*}\bigskip

On the other hand, by inequality (\ref{eq:4.9}), we get the following,
\begin{align*}
\label{eq:4.12}
 (X_{2 \bar{2}} - \kappa) \lambda_3 + (X_{3 \bar{3}} - \kappa) \lambda_2  \geq 2  \sqrt{\lambda_2  \lambda_3 } {\sqrt{(X_{2 \bar{2}} - \kappa) (X_{3 \bar{3}} - \kappa)}} \geq   \frac{2c_1^{1/2}   h^{-1/2} \sqrt{\lambda_2  \lambda_3 }}{\sqrt{1- \delta}}. \tag{4.12}
\end{align*}

By combining inequalities (\ref{eq:4.10}), (\ref{eq:4.11}), and (\ref{eq:4.12}), we may write
\begin{align*}
\sum_i f_i   ( u_{i \bar{i}}  + \kappa     ) &= \frac{-2c_1 \sigma_1(\lambda) - 6c_0 \ta }{\lambda_1 \lambda_2 \lambda_3  }  +  \sum_i   ( X_{i \bar{i}} - \kappa     ) \frac{ A_i }{ \lambda_1 \lambda_2 \lambda_3   \lambda_i } \\
&\geq \frac{-2c_1}{\lambda_2 \lambda_3} + (X_{2 \bar{2}} - \kappa) \frac{c_1}{\lambda_2^2 \lambda_3} + (X_{3 \bar{3}} - \kappa) \frac{c_1}{\lambda_2 \lambda_3^2}  + \lambda_1^{-1} \cdot O_1(1) \\
&=   c_1 \frac{(X_{2 \bar{2}} - \kappa)\lambda_3 + ( X_{3 \bar{3}} - \kappa) \lambda_2 - 2 \lambda_2 \lambda_3 }{\lambda_2^2 \lambda_3^2}    + \lambda_1^{-1} \cdot O_1(1) \\
&\geq   c_1 \frac{ 2 (1-\delta)^{-1/2} c_1^{1/2} h^{-1/2} \sqrt{\lambda_2 \lambda_3}  - 2 (1+ \delta/4) c_1^{1/2} h^{-1/2} \sqrt{\lambda_2 \lambda_3} }{\lambda_2^2 \lambda_3^2}    + \lambda_1^{-1} \cdot O_1(1) \\
&= 2c_1^{3/2} h^{-1/2} \frac{   (1- \delta)^{-1/2}  -  (1+ \delta/4)   }{\lambda_2^{3/2} \lambda_3^{3/2}}    + \lambda_1^{-1} \cdot O_1(1)  \\
&\geq c_1^{3/2}(0)  \cdot  h^{-1/2} \cdot  \frac{   \delta/2  }{\lambda_2^{3/2} \lambda_3^{3/2}}    + \lambda_1^{-1} \cdot O_1(1).
\end{align*}Here, because in this case $\lambda_3$ has a lower bound, otherwise we will not get a lower order term $\lambda_1^{-1} \cdot O_1(1)$. Now since $c_1$ is increasing with $c_1(0) > 0$ and $\lambda_2, \lambda_3$ have a positive lower bound, thus for sufficiently large $\lambda_1$, $\sum_i f_i   ( u_{i \bar{i}}  + \kappa     )$ will be non-negative. \bigskip

In conclusion, from above, we can find a uniform $N>0$ such that if $\lambda_1 > N$, we have
\begin{align*}
\sum_i f_i   ( u_{i \bar{i}}  + \kappa     ) \geq 0 \Longrightarrow \sum_i f_i   u_{i \bar{i}}  \geq - \sum_i f_i \kappa.
\end{align*}
\end{proof}

\hypertarget{L:4.7}{\begin{flemma}}
There exists a uniform $N > 0$ independent of $t \in [0, 1]$ such that if $\lambda_1 > N$, then we have the following estimate
\begin{align*}
-f_2 - f_3 > \frac{  h^{3/2}}{5} > 0.
\end{align*}
\end{flemma}

\begin{proof}
We have
\begin{align*}
-f_2 - f_3 &= \frac{c_1 \sigma_1(\lambda_{;2}) + 2c_0 \ta }{\lambda_1 \lambda_2^2 \lambda_3}  +  \frac{c_1 \sigma_1(\lambda_{;3}) + 2c_0 \ta}{\lambda_1 \lambda_2 \lambda_3^2} \\
&= \frac{c_1(\lambda_1 \lambda_2 + \lambda_1 \lambda_3 + \lambda_2^2 + \lambda_3^2)  + 2c_0 \ta (\lambda_2 + \lambda_3)}{\lambda_1 \lambda_2^2 \lambda_3^2} \geq    \frac{c_1(\lambda_1 \lambda_2 + \lambda_1 \lambda_3 - 2 \lambda_2 \lambda_3)   }{\lambda_1 \lambda_2^2 \lambda_3^2}  \\
&\geq \frac{c_1(\lambda_2 + \lambda_3)}{\lambda_2^2 \lambda_3^2} + \lambda_1^{-1} \cdot O_1(1) \geq \frac{2c_1 }{\lambda_2^{3/2} \lambda_3^{3/2}} + \lambda_1^{-1} \cdot O_1(1) \geq   \frac{2c_1 h^{3/2}}{3^{3/2} c_1^{3/2}} + \lambda_1^{-1} \cdot O_1(1) \\
&> \frac{  h^{3/2}}{5 c^{1/2}_1(1)}  = \frac{  h^{3/2}}{5 } > 0,
\end{align*}provided that $\lambda_1$ is sufficiently large. Here we use the fact that if $\lambda_1$ is sufficiently large, then $3c_1/h \geq \lambda_2  \lambda_3$.
\end{proof}

Now we let $C$ be a constant depending only on the stated data, but which may change from line to line. We can finish the proof of the following $C^2$ estimate.

\hypertarget{T:4.1}{\begin{fthm}}
Suppose $X$ is a $C$-subsolution to equation (\ref{eq:4.1}) and $u \colon M   \rightarrow \mathbb{R}$ is a smooth function solving equation (\ref{eq:4.1}). Then there exists a constant $C$ independent of $t$ such that 
\begin{align*}
|\partial \bar{\partial} u | \leq C \bigl ( 1 + \sup_M \bigl|\nabla u\bigr|^2  \bigr), 
\end{align*}where $C = C(\hat{\theta}, c_0, c_1, \osc_M u, M, X, \omega)$ is a constant, $h = \css$ is a constant, and $\nabla$ is the Levi-Civita connection with respect to $\omega$.\bigskip
\end{fthm}




\begin{proof}

First, by applying the operator \hyperlink{dim 3 operator}{$\mathcal{L}_t$} to $G(\tilde{\Lambda})$, at the maximum point, we obtain
\begin{align*}
\label{eq:4.13}
\mathcal{L}_t  \bigl (  G( \tilde{\Lambda} )   \bigr )   \tag{4.13}
&= - \sum_{i, j, k}f_k   g_{ij}  \frac{\partial \tilde{\Lambda}_i^i}{\partial z_k} \frac{\partial \tilde{\Lambda}_j^j}{\partial \bar{z}_k} - \sum_k f_k   \sum_{i \neq j} \frac{g_i - g_j}{ \tilde{\lambda}_i - \tilde{\lambda}_j}   \frac{\partial \tilde{\Lambda}_j^i}{\partial z_k} \frac{\partial \tilde{\Lambda}_i^j}{\partial \bar{z}_k} - \sum_{i,k} f_k   g_i   \frac{\partial^2 \tilde{\Lambda}_i^i}{\partial z_k \partial \bar{z}_k}    \\
&=  \sum_k f_k \frac{1}{(1 + \tilde{\lambda}_1)^2}   \bigl|    (  X_u     )_{1 \bar{1},k }    \bigr |^2 + \sum_k f_k  \frac{ {\lambda}_1}{1 + \tilde{\lambda}_1} \omega_{1\bar{1},k \bar{k}}  -  \sum_k f_k \frac{1}{1+ \tilde{\lambda}_1}   ( X_u   )_{1 \bar{1},k \bar{k}} \\
&\kern2em  - \sum_{k}f_k \sum_{j \neq 1}  \frac{1}{(1 + \tilde{\lambda}_1)( \tilde{\lambda}_1 - \tilde{\lambda}_j)}   \Bigl(    \bigl |    (  X_u    )_{j \bar{1},k}     \bigr |^2   +      \bigl |    (  X_u    )_{1 \bar{j},k}    \bigr  |^2      \Bigr ) \\
&\geq   \sum_i f_i \frac{1}{(1 + \tilde{\lambda}_1)^2}  \bigl |    (  X_u     )_{1 \bar{1}, i}    \bigr |^2 +  C \sum_i f_i  - \sum_i f_i \frac{1}{1+ \tilde{\lambda}_1}   (  X_u   )_{1 \bar{1}, i \bar{i}}.
\end{align*}Here we change the index from $k$ to $i$ for convenience. Then by equation (\ref{eq:4.8}), we have
\begin{align*}
\label{eq:4.14}
0 = h_{k \bar{k}} &= \sum_{i, j}  f_{ij}   (  X_u     )_{i \bar{i}, \bar{k}}    (  X_u     )_{j \bar{j}, {k}} + \sum_{i \neq j}  \frac{f_i - f_j}{\lambda_i - \lambda_j}    | (X_u)_{j \bar{i},k}    |^2     +  \sum_i \Bigl( f_i  ( X_u   )_{i \bar{i},k \bar{k}} -f_i \lambda_i \omega_{i \bar{i},k \bar{k}} \Bigr) \tag{4.14} \\
&\geq \sum_{i \neq j}  \frac{ h }{  \lambda_i \lambda_j}    | (X_u)_{j \bar{i},k}    |^2     +  \sum_i \Bigl( f_i  ( X_u   )_{i \bar{i},k \bar{k}} -f_i \lambda_i \omega_{i \bar{i},k \bar{k}} \Bigr) \\
&\geq \sum_{i \neq j}  \frac{h}{\lambda_i \lambda_j}    | (X_u)_{j \bar{i},k}    |^2     +  \sum_i  f_i  ( X_u   )_{i \bar{i},k \bar{k}} + C \sum_i f_i \lambda_i,
\end{align*}where the inequality on the second line is due to Lemma~\hyperlink{L:4.4}{4.4}. Since the solution set $\{f_t = h\}$ is convex and $0 = h_k = \sum_i f_i (X_u)_{i \bar{i}, k}$ implies $(X_u)_{i \bar{i}, k}$ is a tangent vector, we obtain that $\sum_{i, j} f_{ij}   (  X_u     )_{i \bar{i}, \bar{k}}    (  X_u     )_{j \bar{j}, {k}}  \geq 0$. Hence by setting $k = 1$, inequality (\ref{eq:4.14}) gives
\begin{align*}
\label{eq:4.15}
-\sum_i f_i   (  X_u     )_{1 \bar{1},i \bar{i}} &= - \sum_i f_i (X_u  )_{i  \bar{i},1 \bar{1}}  + \sum_i f_i \bigl (    (X_u  )_{i  \bar{i},1 \bar{1}} -   ( X_u   )_{1 \bar{1}, i \bar{i}}  \bigr ) \tag{4.15} \\
&= - \sum_i f_i (X_u  )_{i  \bar{i},1 \bar{1}}  + \sum_i f_i \bigl (    (X   )_{i  \bar{i},1 \bar{1}} -   ( X    )_{1 \bar{1}, i \bar{i}}  \bigr ) \\
&\geq   C \sum_i f_i (1 + \lambda_i)   +   \sum_{i \neq j}  \frac{ h }{  \lambda_i \lambda_j}    | (X_u)_{j \bar{i},1}    |^2 \\
&\geq   C \sum_i f_i (1 + \lambda_i)   +   \sum_{j \neq 1}  \frac{h }{  \lambda_1 \lambda_j}    | (X_u)_{j \bar{1},1}    |^2.
\end{align*}

%
%

Combining inequalities (\ref{eq:4.13}) and (\ref{eq:4.15}), at the maximum point $p \in M$, if $\lambda_1$ is sufficiently large, then we have
\begin{align*}
\label{eq:4.16}
0 &\geq \mathcal{L}_t   ( U   ) \geq   A   \sum_i f_i u_{i \bar{i}} + \sum_i f_i \frac{\bigl |    (  X_u     )_{1 \bar{1},i}    \bigr |^2}{(1 + \tilde{\lambda}_1)^2}  +  C \sum_i f_i  - \sum_i f_i  \frac{(  X_u   )_{1 \bar{1}, i\bar{i}}}{1+ \tilde{\lambda}_1}    \tag{4.16} \\
&\geq    \sum_i f_i (Au_{i \bar{i}} + C ) + \sum_i  f_i \frac{  \bigl |    (  X_u     )_{1 \bar{1},i}    \bigr |^2}{(1 + \tilde{\lambda}_1)^2}        +   \sum_{j \neq 1}  \frac{h | (X_u)_{j \bar{1},1}    |^2 }{(1+ \tilde{\lambda}_1)  \lambda_1 \lambda_j}      \\
&\geq \bigl( C-A \kappa \bigr)  \sum_i f_i    + \sum_i f_i \frac{ \bigl |    (  X_u     )_{1 \bar{1},i}    \bigr |^2}{(1 + \tilde{\lambda}_1)^2}          +   \sum_{j \neq 1}  \frac{ h | (X_u)_{j \bar{1},1}    |^2 }{(1+ \tilde{\lambda}_1)  \lambda_1 \lambda_j}     \\
&\geq \frac{A \kappa   h^{3/2}}{10}     +  f_1 \frac{\bigl |    (  X_u     )_{1 \bar{1},1}    \bigr |^2}{(1 + \tilde{\lambda}_1)^2}  +  \sum_{j \neq 1} f_j \frac{  \bigl |    (  X_u     )_{1 \bar{1},j}    \bigr |^2}{(1 + \tilde{\lambda}_1)^2}          + \frac{ h   | (X_u)_{j \bar{1},1}    |^2 }{(1+ \tilde{\lambda}_1)  \lambda_1 \lambda_j}. 
\end{align*}Here, provided that $A$ is sufficiently large such that $A \kappa - C > A \kappa/2$ and $\lambda_1$ is sufficiently large. The inequality on the third line is by Lemma~\hyperlink{L:4.6}{4.6} and the last inequality is by Lemma~\hyperlink{L:4.7}{4.7}. \bigskip

We can also simplify the last two terms in inequality (\ref{eq:4.16}): 
\begin{align*}
\label{eq:4.17}
&\kern-1em \sum_{j \neq 1} f_j  \frac{ \bigl |    (  X_u     )_{1 \bar{1},j}    \bigr |^2}{(1 + \tilde{\lambda}_1)^2}          + \frac{ h  | (X_u)_{j \bar{1},1}    |^2 }{(1+ \tilde{\lambda}_1)   \lambda_1 \lambda_j}    
=  \sum_{j \neq 1} f_j \frac{ \bigl |    (  X_u     )_{1 \bar{1},j}    \bigr |^2}{(1 + \tilde{\lambda}_1)^2}   +   \sum_{j \neq 1}  \frac{ h | (X_u)_{1 \bar{1},j} + S_j   |^2 }{(1+ \tilde{\lambda}_1)  \lambda_1 \lambda_j}      \tag{4.17} \\
&\geq   \sum_{j \neq 1}  f_j \frac{ \bigl |    (  X_u     )_{1 \bar{1},j}    \bigr |^2}{(1 + \tilde{\lambda}_1)^2}   +   \sum_{j \neq 1}  \frac{ h }{(1+ \tilde{\lambda}_1)   \lambda_1 \lambda_j}  \Bigl(  \frac{\lambda_1}{1+ \tilde{\lambda}_1} | (X_u)_{1 \bar{1},j} |^2 - \frac{\lambda_1 - B_{11}}{1+ B_{11}} | S_j   |^2 \Bigr) \\
&\geq    \sum_{j \neq 1}  \frac{ f_j \lambda_j +  h }{(1+ \tilde{\lambda}_1)^2    \lambda_j}      | (X_u)_{1 \bar{1},j} |^2    -  \sum_{j \neq 1}  \frac{ h | S_j   |^2 }{(1+ \tilde{\lambda}_1)    \lambda_j}         \\
&=   \sum_{j \neq 1}  \frac{\bigl |    (  X_u     )_{1 \bar{1},j}    \bigr |^2}{(1 + \tilde{\lambda}_1)^2 \lambda_1 \lambda_2 \lambda_3}   -   \sum_{j \neq 1}  \frac{ h | S_j   |^2 }{(1+ \tilde{\lambda}_1)    \lambda_j}         
\geq    -   \sum_{j \neq 1}  \frac{  h  | S_j   |^2 }{(1+ \tilde{\lambda}_1)     \lambda_j}         \geq -C,
\end{align*}where we denote $S_j \coloneqq  (X_u)_{j \bar{1}, 1}- (X_u)_{1 \bar{1}, j} = X_{j \bar{1}, 1}- X_{1 \bar{1}, j}$. Thus, by inequalities (\ref{eq:4.3}), (\ref{eq:4.16}), and (\ref{eq:4.17}), at the point $p$ we obtain

\begin{align*}
0 &\geq \mathcal{L}_t    ( U     ) \geq \frac{A \kappa   h^{3/2}}{10}     +  f_1 \frac{\bigl |    (  X_u     )_{1 \bar{1},1}    \bigr |^2}{(1 + \tilde{\lambda}_1)^2}  -C \geq \frac{A \kappa   h^{3/2}}{20}- A^2 |u_1|^2 \frac{c_1 (\lambda_2 + \lambda_3) + 2c_0  \tan(\hat{\theta})}{ \lambda_1^2 \lambda_2 \lambda_3 } ,
\end{align*}provided that $A$ is sufficiently large. This implies,
\begin{align*}
\frac{A \kappa   h^{3/2}}{20}  &\leq  A^2 |u_1|^2 \frac{c_1  (\lambda_2 + \lambda_3) + 2c_0  \tan(\hat{\theta})}{ \lambda_1^2 \lambda_2 \lambda_3 } \leq A^2 h |u_1|^2 \frac{c_1  (\lambda_2 + \lambda_3) + 2c_0  \tan(\hat{\theta})}{ \lambda_1 \bigl( c_1 \sigma_1(\lambda)+ 2c_0 \ta  \bigr) } \\
&\leq A^2 h |u_1|^2 \frac{1}{ \lambda_1} \leq A^2 h \cdot    \sup_M |\nabla u|^2  \cdot \frac{1}{ \lambda_1} .
\end{align*}Hence at the maximum point $p$ of $\tilde{U}$, we have
\begin{align*}
\lambda_1 \leq \frac{20 A   }{\kappa   h^{1/2}}   \sup_M |\nabla u|^2.
\end{align*}By plugging back to the original test function $U = -Au +G(\Lambda)$, we will obtain the $C^2$ estimate independent of $t$. This finishes the proof.
\end{proof}



\subsubsection{The $C^1$ Estimates}
\label{sec:4.1.2}
Here, we use a blow-up argument proved by Collins--Jacob--Yau \cite{collins20151} to obtain the $C^1$ estimate. One can also check a more general setting considered by Székelyhidi \cite{szekelyhidi2018fully}, or the complex Hessian equation studied by Dinew--Kołodziej \cite{dinew2017liouville}.

\hypertarget{P:4.1}{\begin{fprop}[Collins--Jacob--Yau \cite{collins20151}]}
Suppose $u \colon M   \rightarrow \mathbb{R}$ satisfies 
\begin{itemize}
\item[(a)] $X + \sqrt{-1} \partial \bar{\partial} u \geq -K \omega$,
\item[(b)] $\|u\|_{L^\infty(M)} \leq K$,
\item[(c)] $\| \partial \bar{\partial} u \|_{L^\infty(M)} \leq K \bigl( 1 + \sup_M |\nabla u |^2  \bigr)$,
\end{itemize}for a uniform constant $K < \infty$. Then there exists a constant $C$, depending only on $M, \omega$, $X$, and $K$ such that
\begin{align*}
\sup_M |\nabla u| \leq C.
\end{align*}
\end{fprop}

\subsubsection{Higher Order Estimates}
\label{sec:4.1.3}


The proof follows from Siu \cite{siu2012lectures}, here we use a standard blow-up argument inspired by Collins--Jacob--Yau \cite{collins20151}. The equation is elliptic and the solution set is convex when $h \in \bigl (  0, c_1^3(t) \ctt / c_0^2(t)      \bigr )$. As long as $h \in \bigl (  0, c_1^3(t) \ctt / c_0^2(t)      \bigr )$ for all $t \in [0, 1]$, we can exploit the convexity of the solution set to obtain $C^{2, \alpha}$ estimates by a blow-up argument. \bigskip


By shrinking the coordinate charts if necessary, we may assume that the manifold $M$ can be covered by finitely many coordinate charts $\bar{U}_a \subset V_a$ such that $X_u = \sqrt{-1} \partial \overline{\partial} u_a$ on $V_a$ for a smooth function $u_{\alpha}$ satisfying $\| u_a \|_{C^2(\bar{U}_a)}   \leq K$, where we use the standard Euclidean metric on $\mathbb{C}^3$ and $K$ is a uniform constant independent of $a$ and $t \in [0, 1]$. For convenience, we focus on a fixed coordinate chart $V_{a}$, we drop the subscript $a$. The function $u$ on $V$ satisfies 
\begin{align*}
\label{eq:4.18}
 {F}_t(x, \partial \bar{\partial} u  ) = F_t \bigl (\Lambda(x) \bigr)  = h, \text{ for } x \in V, \tag{4.18}
\end{align*}where $\Lambda^j_i(x) = \omega^{j \bar{k}}(x) u_{i \bar{k}}(x)$ with eigenvalues $\lambda \bigl ( \Lambda^j_i (x) \bigr) \in \Upsilon^3_{2; 1, 0, -c_1(t) /h }  \cap \Upsilon^3_{1; 1, 0}$ and $h = \css$ is a constant satisfying $c_1^3(t) \ctt / c_0^2(t)  > h > 0$ for all $t \in [0, 1]$ by our \hyperlink{dim 3 cons}{Positivstellensatz constraint}. Moreover, fix $\tilde{x} \in U$, we define the following operator which does not depend on $x \in V$, 
\begin{align*}
\tilde{F}_{t, \tilde{x}}( \partial \bar{\partial} u ) \coloneqq  {F}_t(    \omega^{j \bar{k}}(\tilde{x})  u_{i \bar{k}}   ).
\end{align*}

\bigskip


First, we prove a Hölder estimate for the second derivatives. We have the following.

\hypertarget{L:4.8}{\begin{flemma}}
Fix $t \in [0, 1]$, let $U \subset \mathbb{C}^3$ be a connected open set, and fix $\tilde{x} \in U$. Suppose $u \colon U \subset  \mathbb{C}^3   \rightarrow \mathbb{R}$ is a $C^3$ function such that $\| \partial \bar{\partial}u \|_{L^\infty( U )}   < \infty$ and $\lambda \bigl (\omega^{j \bar{k}}(\tilde{x}) u_{i \bar{k}}(\tilde{x})  \bigr )   \in \Upsilon^3_{2; 1, 0, -c_1(t) /h }  \cap \Upsilon^3_{1; 1, 0}$. If for all $x \in U$,
\begin{align*}
\tilde{F}_{t, \tilde{x}}   ( \partial \bar{\partial} u  )(x) = h,
\end{align*}then there exists a constant $\alpha \in (0, 1)$ such that for any $R > 0$ with $\overline{B_{2R}} \subset U$, the function $u$ satisfies 
\begin{align*}
\| \partial \bar{\partial} u \|_{C^\alpha(B_{R})} \leq C    \cdot R^{-\alpha}.
\end{align*}Here $C = C\bigl( \hat{\theta}, c_0, c_1, h, \| \partial \bar{\partial}u \|_{L^\infty( U )}  \bigr)$ is a constant, $h \in (0, c_1^3(t) \ctt / c_0^2(t))$ is a constant, and $\lambda \bigl (\omega^{j \bar{k}}(\tilde{x}) u_{i \bar{k}}(\tilde{x})  \bigr )$ are the eigenvalues of $\omega^{j \bar{k}}(\tilde{x}) u_{i \bar{k}}(\tilde{x})$. 
\end{flemma}

\hypertarget{R:4.1}{\begin{frmk}}
Notice that the constant $C$ in fact depends on $t$ indirectly as it directly depends on $\| \partial \bar{\partial}u \|_{L^\infty(U)}$. So as long as we find uniform bound for $\| \partial \bar{\partial}u \|_{L^\infty(U)}$, we can have a uniform $C^{2, \alpha}$ bound for $u$ independent of $t \in [0, 1]$.
\end{frmk}

\begin{proof}
First, we may prove that for all $x \in U$,
\begin{align*}
\lambda \bigl (\omega^{j \bar{k}}(\tilde{x}) u_{i \bar{k}}(x)  \bigr )   \in \Upsilon^3_{2; 1, 0, -c_1(t) /h }  \cap \Upsilon^3_{1; 1, 0}.
\end{align*}If not, then by the intermediate value theorem, say there exists $\tilde{x} \in U$ such that 
\begin{align*}
\tilde{F}_{t, \tilde{x}} \bigl (\partial \bar{\partial} u \bigr) (\tilde{x})    = \frac{c_1(t) \sigma_1( {\lambda}) + 2 c_0(t) \tan(\hat{\theta}) }{\lambda_1 \lambda_2 \lambda_3} = h  \text{ and } \lambda_2 \lambda_3 = \frac{c_1(t)}{h},
\end{align*}where $\{\lambda_1, \lambda_2, \lambda_3\}$ are the eigenvalues of $\omega^{j \bar{k}}(\tilde{x}) u_{i \bar{k}}(x)$, which will be a contradiction.\bigskip

Let $\gamma$ be an arbitrary vector of $\mathbb{C}^n$, differentiating $\tilde{F}_{t, \tilde{x}} (\partial \bar{\partial} u) = h$ with respect to $\gamma$ and then with respect to $\bar{\gamma}$ gives
\begin{align*}
\sum_{i, j} \frac{\partial \tilde{F}_{t, \tilde{x}}}{\partial u_{i \bar{j}}} (\partial \bar{\partial} u) u_{i \bar{j} \gamma} = h_\gamma = 0; \quad \sum_{i, j, k, l} \frac{\partial^2 \tilde{F}_{t, \tilde{x}}}{\partial u_{i \bar{j}} \partial u_{k \bar{l}}} (\partial \bar{\partial} u) u_{k \bar{l} \bar{\gamma}} u_{i \bar{j} \gamma} + \sum_{i, j} \frac{\partial \tilde{F}_{t, \tilde{x}}}{\partial u_{i \bar{j}}} ( \partial \bar{\partial} u) u_{i \bar{j} \gamma \bar{\gamma}} = h_{\gamma \bar{\gamma}} = 0.
\end{align*}

By Lemma~\hyperlink{L:4.3}{4.3} and Lemma~\hyperlink{L:4.4}{4.4}, we have
\begin{align*}
\sum_{i, j} \frac{\partial \tilde{F}_{t, \tilde{x}}}{\partial u_{i \bar{j}}} (  \partial \bar{\partial} u) u_{i \bar{j} \gamma \bar{\gamma}} \leq h_{\gamma \bar{\gamma}} = 0. 
\end{align*}Let $w = u_{\gamma \bar{\gamma}}$, then we may rewrite the equation as $-  \sum_{i, j} \tilde{F}_{t, \tilde{x}}^{i \bar{j}} (  \partial \bar{\partial} u)  \partial_i \bar{\partial}_j w  \geq -h_{\gamma \bar{\gamma}} = 0$, where we denote $\tilde{F}_{t, \tilde{x}}^{i \bar{j}} (  \partial \bar{\partial} u)  \coloneqq \partial \tilde{F}_{t, \tilde{x}} /\partial u_{i \bar{j}} (  \partial \bar{\partial} u)$. By the hypothesis that $\| \partial \bar{\partial}u \|_{L^\infty(U)} < \infty$, the eigenvalues of $\sqrt{-1}\partial \bar{\partial}u$ have an upper bound and thus a positive lower bound by equation (\ref{eq:4.18}) itself. Hence the operator
\begin{align*}
-\tilde{F}_{t, \tilde{x}}^{i \bar{j}} (  \partial \bar{\partial} u) \frac{\partial^2}{\partial z_i \bar{\partial} z_j}
\end{align*} is uniformly elliptic.

For $s= 1, 2$, let $M_s \coloneqq \sup_{B_{sR}} w$, where $B_{sR}$ is a ball of radius $sR$ contained in $U$ having the same center. By the Krylov--Safanov's weak Harnack inequality \cite{gilbarg2015elliptic}, there is a constant $p > 0$ and $C > 0$ such that
\begin{align*}
\Bigl(  \frac{1}{R^{6}} \int_{B_R} (M_2 - w)^p   \Bigr)^{1/p} \leq C \bigl(   M_2 - M_1 + R^{\frac{2(q-3)}{q}} \| h_{\gamma \bar{\gamma}} \|_{L^q(B_{2R})}     \bigr ),
\end{align*}where $q >3$. Then, by the smoothness and convexity of the solution set $\{f_t = h\}$, the tangent plane to the graph of $\tilde{F}_t$ at the point $\bigl(u_{i \bar{j}}(y)\bigr)$ is below the graph of $\tilde{F}_t$. Hence the tangent plane will be the supporting hyperplane, which implies


\begin{align*}
\tilde{F}_{t, \tilde{x}} ( \partial \bar{\partial} u(y))  - \tilde{F}_{t, \tilde{x}}^{i \bar{j}} (  \partial \bar{\partial} u(y)) \bigl( u_{i \bar{j}}(x) - u_{i \bar{j}}(y) \bigr) \geq \tilde{F}_{t, \tilde{x}} (  \partial \bar{\partial} u(x)), 
\end{align*}that is,
\begin{align*}
 0 = h(y) - h(x) \geq - \tilde{F}_{t, \tilde{x}}^{i \bar{j}} (  \partial \bar{\partial} u(y)) \bigl( u_{i \bar{j}}(y) - u_{i \bar{j}}(x) \bigr).
\end{align*}Lastely, the rest follows directly from the proof of the complex version of the Evans--Krylov theory in Siu \cite{siu2012lectures}.
\end{proof}

Then, with the above Lemma~\hyperlink{L:4.8}{4.8}, we can prove a Louiville-type result.

\hypertarget{P:4.2}{\begin{fprop}}
Fix $t \in [0, 1]$ and $\tilde{x} \in \mathbb{C}^3$. Suppose $u \colon \mathbb{C}^3   \rightarrow \mathbb{R}$ is a $C^3$ function such that $\| \partial \bar{\partial}u \|_{L^\infty( \mathbb{C}^3 )} < \infty$ and $\lambda \bigl (\omega^{j \bar{k}}(\tilde{x}) u_{i \bar{k}}(\tilde{x})  \bigr )   \in \Upsilon^3_{2; 1, 0, -c_1(t) /h }  \cap \Upsilon^3_{1; 1, 0}$. If for all $x \in \mathbb{C}^3$,
\begin{align*}
\tilde{F}_{t, \tilde{x}}   ( \partial \bar{\partial} u  )(x) = h,
\end{align*}then $u$ is a quadratic polynomial. Here $\lambda \bigl (\omega^{j \bar{k}}(\tilde{x}) u_{i \bar{k}}(\tilde{x})  \bigr )$ are the eigenvalues of $\omega^{j \bar{k}}(\tilde{x}) u_{i \bar{k}}(\tilde{x})$ and $h \in (0, c_1^3(t) \ctt / c_0^2(t))$ is a constant.
\end{fprop}

\begin{proof}
The proof follows from Lemma~\hyperlink{L:4.8}{4.8} by letting $R   \rightarrow \infty$.
\end{proof}

%


\hypertarget{L:4.9}{\begin{flemma}}
For $r > 0$, suppose $u \colon B_{2r} \subset \mathbb{C}^3   \rightarrow \mathbb{R}$ is a smooth function satisfying 
\begin{align*}
{F}_t (x,  \partial \bar{\partial} u ) = h,
\end{align*}where $h \in \bigl[  \epsilon, c_1^3(t) \ctt / c_0^2(t) -\epsilon \bigr]$ is a constant and $\epsilon > 0$ small. Then, for every $\alpha \in (0, 1)$, we have the estimate
\begin{align*}
\|\partial \bar{\partial} u\|_{C^{\alpha} (B_{r/2})} \leq C( \alpha, \hat{\theta}, c_0, c_1, h, \epsilon, \|\partial \bar{\partial} u\|_{L^{\infty} (B_{2r})}).
\end{align*}
\end{flemma}



\begin{proof}
%

For each $x \in B_r$, we consider the following quantity 
\begin{align*}
N_u \coloneqq \sup_{x \in B_r} d_x |\partial \partial \bar{\partial} u(x) |, 
\end{align*}where $d_x \coloneqq \dist(x, \partial B_r)$. Suppose the supremum is achieved at $x_0 \in B_r$, then we consider the following smooth function $\tilde{u} \colon B_{N_u}(0)   \rightarrow \mathbb{R}$ defined by 
\begin{align*}
\tilde{u}(z) \coloneqq  u \bigl ( x_0 + {d_{x_0}}z/N_u \bigr){N_u^2}/{d_{x_0}^2} - A - A_i z_i,
\end{align*}where $A$, $A_i$ are chosen so that $\tilde{u}(0) = 0 = \partial \tilde{u}(0)$. Notice that
\begin{align*}
\partial \bar{\partial} \tilde{u} (z) = \partial \bar{\partial} u(x_0 + d_{x_0} z/N_u); \quad \| \partial \partial \bar{\partial} \tilde{u} \|_{L^{\infty} (B_{N_u} (0)) } = 1 = | \partial \partial \bar{\partial} \tilde{u} (0) |.
\end{align*}In particular, we have $\|\partial \bar{\partial} \tilde{u} \|_{C^{\alpha} (B_r)} \leq r$ for every $\alpha \in (0, 1)$ and $\tilde{u}$ solves 
\begin{align*}
{F}_t \bigl  ( x_0  + {d_{x_0}}{}z/N_u, \partial \bar{\partial} \tilde{u}  \bigr ) (z) = h,\quad z \in B_{N_u}(0).
\end{align*}
By the hypothesis that $\| \partial \bar{\partial}u \|_{L^\infty(B_{2r})} < \infty$, the eigenvalues of $\sqrt{-1}\partial \bar{\partial}u$ have an upper bound and thus a positive lower bound, so $F_t(x, \cdot)$ is uniformly elliptic. The Schauder theory for fully nonlinear uniformly elliptic operators of the form $F(x, \partial \bar{\partial} u)$ implies that $\partial \tilde{u}$ is bounded in $C^{2, \alpha}(B_{N_u /2} (0))$, and so $\tilde{u}$ is controlled in $C^{3, \alpha}(B_{N_u /2} (0))$. Now, we prove by contradiction. Suppose we have a sequence $\{u_n\}$ satisfying ${F}_t (x,  \partial \bar{\partial} u_n ) = h_n$, where $u_n \colon B_{2r}   \rightarrow \mathbb{R}$ such that $\| \partial \bar{\partial} u_n \|_{L^{\infty}(B_{2r})} \leq K$ but $N_{u_n} \geq n$. For each $n$, we let $x_n \in B_r$ be a point where $N_{u_n}$ is achieved. Since $\overline{B_r}$ is compact, by passing to a subsequence, we may assume that $x_n   \rightarrow x_{\infty} \in \overline{B_r}$ and $h_n$ converges to $h_{\infty} \in \bigl[  \epsilon, c_1^3(t) \ctt / c_0^2(t) -\epsilon \bigr]$. \bigskip



Thus, we have functions $\tilde{u}_n \colon B_{N_{u_n}}(0)   \rightarrow \mathbb{R}$ such that
\begin{align*}
\| \tilde{u}_n \|_{C^{3, \alpha}(B_{N_{u_n}}(0) )} \leq C \text{ and } {F}_t \bigl  ( x_n  + {d_{x_n}}{}z/N_{u_n}, \partial \bar{\partial} \tilde{u}_n  \bigr ) (z) = h_n   \text{ for } z \in B_{N_{u_n}}(0).
\end{align*}Since $N_{u_n} \geq n$, by a diagonal argument, there exists a function $\tilde{u}_{\infty} \colon \mathbb{C}^3   \rightarrow \mathbb{R}$ and a subsequence such that $\{ \tilde{u}_n \}_{n \geq k}$ converges uniformly to $\tilde{u}_{\infty}$ in $C^{3, \alpha'}(B_k(0))$ for some $\alpha' \in (0, 1)$. In particular, we have
\begin{align*}
\tilde{F}_{t, x_\infty}( \partial \bar{\partial} \tilde{u}_{\infty})(x) = h_{\infty}   \text{ and }   |\partial \partial \bar{\partial} \tilde{u}_{\infty}(0)| = 1. 
\end{align*} Since $h_{\infty}  \in \bigl[  \epsilon, c_1^3(t) \ctt / c_0^2(t) - \epsilon     \bigr]$, then by the Proposition~\hyperlink{P:4.2}{4.2}, $\tilde{u}_\infty$ is a quadratic polynomial, which leads to a contradiction. 
\end{proof}

By arguing locally,  with Lemma~\hyperlink{L:4.9}{4.9} we have the following.

\hypertarget{C:4.1}{\begin{fcor}}
Suppose $X$ is a $C$-subsolution to equation (\ref{eq:4.1}) and $u \colon M  \rightarrow \mathbb{R}$ is a solution to equation (\ref{eq:4.1}), that is,
\begin{align*}
F_t \bigl(  \omega^{-1} X_u    \bigr) = h,
\end{align*}where $X_u \coloneqq X + \sqrt{-1} \partial \bar{\partial} u$ and $h  = \css$ is a constant. Then for every $\alpha \in (0, 1)$, we have
\begin{align*}
\| \partial \bar{\partial} u \|_{C^\alpha(M)} \leq C(M, X, \omega, \alpha, \hat{\theta}, c_0, c_1, \| \partial \bar{\partial} u\|_{L^{\infty}(M)}).
\end{align*}

\end{fcor}

\subsection{When \texorpdfstring{$n =4$}{}}
\label{sec:4.2}

In this subsection, first, we always assume that $\hat{\theta} \in \bigl( \pi, 5 \pi/4    \bigr)$ and there exists a $C$-subsolution $\ubar{u} \colon M   \rightarrow \mathbb{R}$. We also call $X_{\ubar{u}}$ this $C$-subsolution and by changing representative, we may assume $X$ is this $C$-subsolution. We also abbreviate $\lambda = \{\lambda_1, \lambda_2, \lambda_3, \lambda_4\}$ and we always assume $\lambda_1 \geq \lambda_2 \geq \lambda_3 \geq \lambda_4$ unless further notice. Most of the time, to save spaces, we will abbreviate $f = f_t =  f_t(\lambda)$, $f_i = \partial f /\partial \lambda_i, f_{ij} = \partial^2 f/\partial \lambda_i \partial \lambda_j$ for $i, j \in \{1, 2, 3, 4\}$ for notational convention. Last, unless specify otherwise, we always abbreviate $c_2 = c_2(t)$, $c_1 = c_1(t)$, and $c_0 = c_0(t)$. We assume $(c_2, c_1, c_0)$ satisfy all \hyperlink{dim 4 cons}{4-dimensional four constraints} in Section~\ref{sec:3.2} and consider equation (\ref{eq:3.13})
\begin{align*}
\label{eq:4.19}
f_t(\lambda_1, \lambda_2, \lambda_3, \lambda_4) =  \frac{c_2(t) \sigma_2(\lambda)  - 2 c_1(t) \cot(\hat{\theta}) \sigma_1(\lambda) + c_0(t) \cc  }{\lambda_1 \lambda_2 \lambda_3 \lambda_4} = h, \tag{4.19}
\end{align*}where $\lambda_i$ are the eigenvalues of $\omega^{-1}X_u$ and $h = \sii$ is a constant. By the \hyperlink{dim 4 cons}{$\Upsilon$-cone constraints} it automatically satisfies $c_2^3(t) \taa/ c_1^2(t) > h > 0$ and \[c_0(t) \cc h  + 24 c_2^2(t)  \cos^2(\theta_{h, c_1(t), c_2(t)})\cos(2\theta_{h, c_1(t), c_2(t)}) > 0.\] Here \[\theta_{h, c_1(t), c_2(t)} \coloneqq \arccos \bigl(   {- c_1(t) \ct h^{1/2} }\big/{c_2^{3/2}(t) }  \bigr)\big/3  - {2\pi}/{3}\] and we specify the branch so that $ \arccos  ( \bullet  ) \in \bigl( \pi, 3\pi/2 \bigr]$. \bigskip

\subsubsection{The $C^2$ Estimates}
\label{sec:4.2.1}
Define a Hermitian endomorphism $\Lambda \coloneqq \omega^{-1}X_u$, where $X_u = X + \sqrt{-1} \partial \bar{\partial} u$, and let $\lambda = \{ \lambda_1, \lambda_2, \lambda_3, \lambda_4\}$ be the eigenvalues of $\Lambda$. We consider the following function $G(\Lambda) = \log(1 + \lambda_1) =g(\lambda_1, \lambda_2, \lambda_3, \lambda_4)$ and the following test function
\begin{align*}
U \coloneqq - Au + G(\Lambda), 
\end{align*}where $A \gg 0$ will be determined later. We want to apply the maximum principle to $U$, but since the eigenvalues of $\Lambda$ might not be distinct at the maximum point $p \in M$ of $U$, we do a perturbation here. The perturbation here, though not necessarily, is made to preserve the $\Upsilon$-cone structure for convenience. Assume $\lambda_1$ is large, otherwise we are done, then 
\begin{itemize}
\hypertarget{pert for dim 4}{\item} we pick the constant matrix $B$ to be a diagonal matrix with real entries 
\begin{align*}
B_{11} = \epsilon;\quad B_{22} = \epsilon/2;\quad B_{33} = \epsilon/3;\quad B_{44} = 0
\end{align*}
such that $\tilde{\lambda}_i = \lambda_i + B_{ii}$ with $\lambda_1 + \epsilon = \tilde{\lambda}_1 > \tilde{\lambda}_2   > \tilde{\lambda}_3 > \tilde{\lambda}_4 = \lambda_4 > 0$ and assume $\epsilon > 0$ is sufficiently small.
\end{itemize}
By defining $\tilde{\Lambda} = \Lambda + B$, then $\tilde{\Lambda}$ has distinct eigenvalues near $p \in M$, which are $\{\tilde{\lambda}_1, \tilde{\lambda}_2, \tilde{\lambda}_3, \tilde{\lambda}_4\}$. The eigenvalues of $\tilde{\Lambda}$ define smooth functions near the maximum point $p$. And we can check $p$ is still the maximum point of the following locally defined test function
\begin{align*}
\label{eq:4.20}
\tilde{U} \coloneqq -Au + G(\tilde{\Lambda}). \tag{4.20}
\end{align*}

Near the maximum point $p$ of $\tilde{U}$, we always use the coordinates in Lemma~\hyperlink{L:2.4}{2.4} unless otherwise noted. We instantly get the following.
\hypertarget{L:4.10}{\begin{flemma}}
At the maximum point $p$ of $\tilde{U}$, by taking the first derivative of $\tilde{U}$ at $p$, we get
\begin{align*}
\label{eq:4.21}
0 &= -A u_k(p) + \frac{1}{1+ \tilde{\lambda}_1} ( X_u )_{1\bar{1}, k}, \tag{4.21}
\end{align*}where we denote $u_k = \partial u/\partial z_k$ and $( X_u )_{1\bar{1}, k} = \partial (X_u)_{1 \bar{1}} / \partial z_k$.
\end{flemma}

We may define the following operator \hypertarget{dim 4 operator}{$\mathcal{L}_t$} by
\begin{align*}
\label{eq:4.22}
\mathcal{L}_t \coloneqq - \sum_{i, j, k} \frac{\partial F_t}{\partial \Lambda^k_i} ( {\Lambda}) \omega^{k \bar{j}}  \frac{\partial^2}{\partial z_i \partial \bar{z}_j}, \tag{4.22}
\end{align*}where $F_t = F_t( {\Lambda}) = f_t( {\lambda}_1,  {\lambda}_2,  {\lambda}_3, \lambda_4)$ is defined by $f_t( {\lambda}) =  \bigl( c_2(t) \sigma_2(\lambda)  - 2 c_1(t) \cot(\hat{\theta}) \sigma_1(\lambda) + c_0(t) \cc \bigr)/  {\lambda}_1   {\lambda}_2   {\lambda}_3 \lambda_4$. We immediately have the following Lemmas.



\hypertarget{L:4.11}{\begin{flemma}}
By taking $f_t(\lambda) =    (  {c_2(t) \sigma_2(\lambda) -2c_1(t) \ct \sigma_1(\lambda) + c_0(t) \cc } )/{\sigma_4(\lambda)}$ and $g(\lambda)=\log(  1 +\lambda_1)$, we have 
\begin{align*}
f_i &= \frac{- c_2 \sigma_2(\lambda_{;i}) + 2 c_1  \cot(\hat{\theta}) \sigma_1(\lambda_{;i}) - c_0  (3 \csc^2(\hat{\theta}) -4) }{ \lambda_1 \lambda_2 \lambda_3 \lambda_4 \lambda_i };     \\
f_{ij} &= \frac{ -c_2  (\lambda_i \sigma_1(\lambda_{;i}) + \lambda_l \sigma_1(\lambda_{;i, j}))  +2c_1 \cot(\hat{\theta}) (\lambda_i + \lambda_j) }{\lambda_1 \lambda_2 \lambda_3 \lambda_i \lambda_j} \\
&\kern2em  + \frac{ \bigl [  c_2 \sigma_2(\lambda) - 2c_1 \cot(\hat{\theta}) \sigma_1 (\lambda)   + c_0  (3 \csc^2(\hat{\theta}) -4)  \bigr ] (1 + \delta_{ij})}{\lambda_1 \lambda_2 \lambda_3 \lambda_4 \lambda_i \lambda_j}; \\
g_i &= \delta_{1i} \frac{1}{ 1 + \lambda_1}; \quad  g_{ij}=-\delta_{1i}\delta_{1j}\frac{1}{( 1 +\lambda_1)^2}.
\end{align*}Here, $\lambda_{;i}$ means we exclude $\lambda_i$ from $\lambda = \{ \lambda_1, \lambda_2, \lambda_3, \lambda_4\}$, $\lambda_{;i, j}$ means we exclude both $\lambda_i$ and $\lambda_j$ from $\lambda = \{ \lambda_1, \lambda_2, \lambda_3, \lambda_4\}$, and we denote $f_i \coloneqq  {\partial  f_t}/{\partial \lambda_i}, g_i \coloneqq  {\partial g}/{\partial \lambda_i}$, $(f_t)_{ij} \coloneqq  {\partial^2 f_t}/{\partial \lambda_i \partial \lambda_j}$, and $g_{ij} \coloneqq  {\partial^2 g}/{\partial \lambda_i \partial \lambda_j}$.
\end{flemma}


\hypertarget{L:4.12}{\begin{flemma}}
If $ {   c_2^{3}} \taa / c_1^2 > h > 0$ and $c_0 \cc h  + 24 c_2^2  \cos^2(\theta_{h, c_1, c_2})\cos(2\theta_{h, c_1, c_2})$ $> 0$, where $\theta_{h, c_1, c_2} \coloneqq \arccos \bigl(   {- c_1 \ct h^{1/2} }\big/{c_2^{3/2} }  \bigr)\big/3  - {2\pi}/{3}$. Then for any point on the solution set $\{ f_t = h\}$, we have
\begin{align*}
-f_i =  \frac{c_2 \sigma_2(\lambda_{;i}) - 2 c_1  \cot(\hat{\theta}) \sigma_1(\lambda_{;i}) + c_0  (3 \csc^2(\hat{\theta}) -4) }{ \lambda_1 \lambda_2 \lambda_3 \lambda_4 \lambda_i }  > 0
\end{align*}for any $i \in \{1, 2, 3, 4\}$ at this point. Here $f_i = \partial f_t / \partial \lambda_i$, where $i \in \{1, 2, 3, 4\}$ and we specify the branch so that $ \arccos  ( \bullet ) \in  ( \pi, 3\pi/2  ]$.
\end{flemma}


\begin{proof}
By the Positivstellensatz Theorem~\hyperlink{T:2.1}{(e)}, the $\Upsilon$-cone $\Upsilon^4_{3; 1, 0, -c, d} \cap \Upsilon^4_{2; 1, 0, -c} \cap \Upsilon^4_{1; 1, 0}$ is contained in $\Upsilon^4_{3; 0, c, -d, e} \cap \Upsilon^4_{2; 1, 0, -c} \cap \Upsilon^4_{1; 1, 0}$, for $c > 0$, $d \geq 0$, $2 c^{3/2} > d$, and $e > -24 c^2 \cos^2(\theta_{c, d})\cos(2\theta_{c, d})$. Here $\theta_{c,d} \coloneqq   \arccos \bigl(   {-d}/{2c^{3/2}}  \bigr)/3  - {2\pi}/{3}$ and we specify the branch so that $\arccos  (  \bullet  ) \in  ( \pi, 3\pi/2  ]$. By letting $c = c_2/h$, $d = 2c_1 \ct/h$, and $e = c_0 \cc/h$. If $2 c^{3/2} > d$ and $e > -24 c^2 \cos^2(\theta_{c, d})\cos(2\theta_{c, d})$, then 
\begin{align*}
0< c \sigma_2(\lambda_{; i}) -d \sigma_1(\lambda_{;i}) + e = c_2  \sigma_2(\lambda_{; i}) /h -2 c_1 \ct \sigma_1(\lambda_{;i})/h + c_0 \cc/h 
\end{align*}for all $i \in \{1, 2, 3, 4\}$. By checking the quantity $2 c^{3/2} - d$, we get
\begin{align*}
2 c^{3/2} - d =  {2 c_2^{3/2}}{h^{-3/2}} -  {2c_1 \ct}{h^{-1}} =   {2}{h^{-3/2}} \bigl( c_2^{3/2}  - c_1 \ct h^{1/2} \bigr) > 0
\end{align*}by our hypothesis. In addition, for the quantity $e  + 24 c^2 \cos^2(\theta_{c, d})\cos(2\theta_{c, d})$, we have
\begin{align*}
e + 24 c^2 \cos^2(\theta_{c, d})\cos(2\theta_{c, d}) &=  {c_0 \cc}{h^{-1}}  + 24  {c_2^2 \cos^2(\theta_{c, d})\cos(2\theta_{c, d}) }{h^{-2}} \\
&=  {h^{-2}} \Bigl(  c_0 \cc h + 24 c_2^2 \cos^2(\theta_{h, c_1, c_2})\cos(2\theta_{h, c_1, c_2})    \Bigr), 
\end{align*}this quantity is positive, which finishes the proof.
\end{proof}

With Lemma~\hyperlink{L:4.12}{4.12}, if we assume $\bigl(c_2(t), c_1(t), c_0(t)\bigr)$ satisfy all \hyperlink{dim 4 cons}{4-dimensional four constraints}, then we get that the operator \hyperlink{dim 4 operator}{$\mathcal{L}_t $} is indeed an elliptic operator on the solution set $\{ f_t = h\}$. Here $ {   c_2^{3}} \taa / c_1^2 > h > 0$ and $c_0 \cc h  + 24 c_2^2  \cos^2(\theta_{h, c_1, c_2})\cos(2\theta_{h, c_1, c_2}) > 0$, where \[\theta_{h, c_1, c_2} \coloneqq \arccos \bigl(   {- c_1 \ct h^{1/2} }\big/{c_2^{3/2} }  \bigr)\big/3  - {2\pi}/{3}\]
and we specify the branch so that $ \arccos  (  \bullet ) \in  ( \pi, 3\pi/2  ]$.


\hypertarget{L:4.13}{\begin{flemma}}
If $ {   c_2^{3}} \taa / c_1^2 > h > 0$ and $c_0 \cc h  + 24 c_2^2  \cos^2(\theta_{h, c_1, c_2})\cos(2\theta_{h, c_1, c_2})$ $> 0$, then the solution set $\{ f_t = h\}$ is convex. 
\end{flemma}



\begin{proof}
For convenience, we assume $\lambda_4$ is the smallest eigenvalue and we drop the assumption that $\lambda_1 \geq \lambda_2 \geq \lambda_3$. Let $V = (V_1, V_2, V_3, V_4) \in T_{ {\lambda}} \bigl\{ f_t = h   \bigr\}$ be a tangent vector, which gives, $\sum_i f_i V_i =0$. Then we are trying to show that the following quantity
\begin{align*}
\label{eq:4.23}
\sum_{i, j} f_{ij} V_i  {V}_{\bar{j}}  \tag{4.23}
\end{align*}is positive. First, since $V$ is a tangent vector, we can write
$V_4 = - \bigl({f_1 V_1 + f_2 V_2 + f_3 V_3 }\bigr) \big/{f_4}$. By plugging in quantity (\ref{eq:4.23}), we obtain 
\begin{align*}
\label{eq:4.24}
\sum_{i, j} f_{ij} V_i V_{\bar{j}} &= f_{11} |V_1|^2 + f_{22} |V_2|^2  + f_{33} |V_3|^2  + f_{44} |V_4|^2  + \sum_{1 \leq i< j \leq 4} f_{ij} (V_i V_{\bar{j}} + V_{\bar{i}} V_j) \tag{4.24} \\
&= \Bigl(  f_{11} + f_{44} \frac{f_{1}^2}{f_{4}^2}  -2 f_{14} \frac{f_{1}}{f_{4}}  \Bigr) |V_1|^2  + \Bigl(  f_{22} + f_{44} \frac{f_{2}^2}{f_{4}^2}  -2 f_{24} \frac{f_{2}}{f_{4}}  \Bigr) |V_2|^2 \\
&\kern2em + \Bigl(  f_{33} + f_{44} \frac{f_{3}^2}{f_{4}^2}  -2 f_{34} \frac{f_{3}}{f_{4}}  \Bigr) |V_3|^2 \\
&\kern2em + \Bigl (  f_{12}  + f_{44} \frac{f_{1} f_{2}}{f_{4}^2}  - f_{14} \frac{f_{2}}{f_{4}}  - f_{24} \frac{f_{1}}{f_{4}}  \Bigr) \bigl( V_1 V_{\bar{2}} + V_{\bar{1}} V_2  \bigr) \\
&\kern2em + \Bigl (  f_{13}  + f_{44} \frac{f_{1} f_{3}}{f_{4}^2}  - f_{14} \frac{f_{3}}{f_{4}}  - f_{34} \frac{f_{1}}{f_{4}}  \Bigr) \bigl( V_1 V_{\bar{3}} + V_{\bar{1}} V_3  \bigr) \\
&\kern2em + \Bigl (  f_{23}  + f_{44} \frac{f_{2} f_{3}}{f_{4}^2}  - f_{24} \frac{f_{3}}{f_{4}}  - f_{34} \frac{f_{2}}{f_{4}}  \Bigr) \bigl( V_2 V_{\bar{3}} + V_{\bar{2}} V_3  \bigr). 
\end{align*}If we can show that
\begin{align*}
&\kern-1em \Bigl(  f_{ii} + f_{44} \frac{f_{i}^2}{f_{4}^2}  -2 f_{i4} \frac{f_{i}}{f_{4}}  \Bigr) |V_i|^2 + \Bigl(  f_{jj} + f_{44} \frac{f_{j}^2}{f_{4}^2}  -2 f_{j4} \frac{f_{j}}{f_{4}}  \Bigr) |V_j|^2 \\
&\geq -2 \Bigl (  f_{ij}  + f_{44} \frac{f_{i} f_{j}}{f_{4}^2}  - f_{i4} \frac{f_{j}}{f_{4}}  - f_{j4} \frac{f_{i}}{f_{4}}  \Bigr) \bigl( V_i V_{\bar{j}} + V_{\bar{i}} V_j  \bigr), 
\end{align*}for all $1 \leq i < j \leq 3$, then by summing over all pairs $(i, j)$, quantity (\ref{eq:4.24}) will be non-negative. Without loss of generality, we consider the case when $i = 1$ and $j = 2$. First, we have the following observation
\begin{align*}
\label{eq:4.25}
f_{ii} = \frac{2 c_2 \sigma_2(\lambda_{;i}) -4 c_1 \ct \sigma_1(\lambda_{;i}) + c_0 \cc }{\lambda_1 \lambda_2 \lambda_3  \lambda_4 \lambda_i^2} = \frac{-2}{\lambda_i} f_{i}. \tag{4.25}
\end{align*}

Then the coefficient of $|V_1|^2$ will be
\begin{align*}
f_{11} + f_{44} \frac{f_{1}^2}{f_{4}^2}  -2 f_{14} \frac{f_{1}}{f_{4}}  
= &\frac{1}{f_4^2} \bigl( f_{11} f_4^2 + f_{44}  f_{1}^2   -2 f_{14} f_{1} f_{4} \bigr) = \frac{1}{f_4^2} \Bigl( \frac{-2}{\lambda_1} f_1 f_4^2 +  \frac{-2}{\lambda_4} f_{4}  f_{1}^2   -2 f_{14} f_{1} f_{4} \Bigr) \\
= &\frac{-2 f_1}{ \lambda_1 \lambda_4 f_4} \Bigl( \lambda_4 f_4 +  \lambda_1  f_{1}   + \lambda_1 \lambda_4 f_{14}   \Bigr) \\
= &\frac{2 f_1}{ \lambda_1^2 \lambda_2 \lambda_3 \lambda_4^2 f_4}  \Bigl( c_2 \sigma_2(\lambda) - c_2 \lambda_1 \lambda_4  - 2 c_1 \ct \sigma_1(\lambda) + c_0 \cc    \Bigr)  \\
> &\frac{2 f_1}{ \lambda_1^2 \lambda_2 \lambda_3 \lambda_4^2  f_4}  \Bigl( c_2 (\lambda_2 \lambda_4 + \lambda_3 \lambda_4)  - 2 c_1 \ct \lambda_4    \Bigr)  \\
= &\frac{2 f_1}{ \lambda_1^2 \lambda_2 \lambda_3 \lambda_4  f_4}  \Bigl( c_2 (\lambda_2   + \lambda_3  )  - 2 c_1 \ct      \Bigr) > 0.
\end{align*}
The inequality on the second to last line is due to Lemma~\hyperlink{L:4.12}{4.12} and the last inequality is by Lemma~\hyperlink{L:3.4}{3.4}. 
%
%
%
%
Similarly, the coefficient of $|V_2|^2$ will also be positive. By checking the discriminant of the following quadratic form 
\begin{align*}
\label{eq:4.26}
&\kern-1em \Bigl(  f_{11} + f_{44} \frac{f_{1}^2}{f_{4}^2}  -2 f_{14} \frac{f_{1}}{f_{4}}  \Bigr) |V_1|^2 + \Bigl(  f_{22} + f_{44} \frac{f_{2}^2}{f_{4}^2}  -2 f_{24} \frac{f_{2}}{f_{4}}  \Bigr) |V_2|^2 \tag{4.26} \\
&\kern2em +2 \Bigl (  f_{12}  + f_{44} \frac{f_{1} f_{2}}{f_{4}^2}  - f_{14} \frac{f_{2}}{f_{4}}  - f_{24} \frac{f_{1}}{f_{4}}  \Bigr) \bigl( V_1 V_{\bar{2}} + V_{\bar{1}} V_2  \bigr)  \\
&=\frac{2 f_1}{ \lambda_1^2 \lambda_2 \lambda_3 \lambda_4^2 f_4}  \Bigl( c_2 \sigma_2(\lambda) - c_2 \lambda_1 \lambda_4  - 2 c_1 \ct \sigma_1(\lambda) + c_0 \cc    \Bigr)  |V_1|^2 \\
&\kern2em + \frac{2 f_2}{ \lambda_1 \lambda_2^2 \lambda_3 \lambda_4^2 f_4}  \Bigl( c_2 \sigma_2(\lambda) - c_2 \lambda_2 \lambda_4  - 2 c_1 \ct \sigma_1(\lambda) + c_0 \cc    \Bigr)  |V_2|^2 \\
&\kern2em + \frac{2}{\lambda_4 f_4} \Bigl (   \lambda_4 f_{12} f_4   -2{f_{1} f_{2}}  - \lambda_4 f_{14}  {f_{2}} - \lambda_4 f_{24}  {f_{1}}  \Bigr) \bigl( V_1 V_{\bar{2}} + V_{\bar{1}} V_2  \bigr),  
\end{align*}if the discriminant is non-positive, then the quadratic form will be non-negative. To save spaces, we do not expand the discriminant, the discriminant will be the following 
\begin{align*}
\label{eq:4.27}
&\kern-1em \biggl ( \frac{2}{\lambda_4 f_4} \Bigl (   \lambda_4 f_{12} f_4   -2{f_{1} f_{2}}  - \lambda_4 f_{14}  {f_{2}} - \lambda_4 f_{24}  {f_{1}}  \Bigr) \biggr)^2 \tag{4.27}  \\
&\kern2em - \frac{2 f_1}{ \lambda_1^2 \lambda_2 \lambda_3 \lambda_4^2 f_4}  \Bigl( c_2 \sigma_2(\lambda) - c_2 \lambda_1 \lambda_4  - 2 c_1 \ct \sigma_1(\lambda) + c_0 \cc    \Bigr)    \\
&\kern2em   \times \frac{2 f_2}{ \lambda_1 \lambda_2^2 \lambda_3 \lambda_4^2 f_4}  \Bigl( c_2 \sigma_2(\lambda) - c_2 \lambda_2 \lambda_4  - 2 c_1 \ct \sigma_1(\lambda) + c_0 \cc    \Bigr) \\
&= - \frac{4}{\lambda_1^6 \lambda_2^6 \lambda_3^4 \lambda_4^6 f_4^2 } r_1(\lambda) \cdot r_2(\lambda),
\end{align*}where 
\begin{align*}
\label{eq:4.28}
r_1(\lambda) &= c_2^2 \bigl( \lambda_3 \sigma_2(\lambda) + \lambda_1 \lambda_2 \lambda_4 \bigr)  - 2c_1 c_2 \ct \bigl( \lambda_3 \sigma_1(\lambda) + \sigma_2(\lambda)    \bigr) + 4 c_1^2 \cot^2\bigl(\hat{\theta}\bigr) \sigma_1(\lambda) \tag{4.28} \\
&\kern2em  + c_0c_2 \cc \lambda_3 - 2c_0 c_1 \ct \cc
\end{align*}and
\begin{align*}
\label{eq:4.29}
 r_2(\lambda) &= c_2^2 \bigl (  \lambda_3  (\sigma_2(\lambda_{;3}) -3 \lambda_4^2   ) \sigma_2(\lambda) + \lambda_1 \lambda_2 \lambda_4 (\sigma_2(\lambda) -3\lambda_4(\lambda_3 + \lambda_4)  )  \bigr ) \tag{4.29} \\
&\kern2em +4 c_1^2 \cot^2 (\hat{\theta} ) \sigma_1(\lambda)     (  \sigma_2(\lambda) -3 \lambda_4(\lambda_3 + \lambda_4)    )   + c_0^2 \cc^2   (\lambda_1 + \lambda_2 - 2 \lambda_4)     \\
&\kern2em- 2 c_1 c_2 \ct \Bigl(  \sigma_2(\lambda)    \bigl (  \sigma_2(\lambda) -3 \lambda_4(\lambda_3 + \lambda_4)  \bigr  )  + \sigma_1(\lambda) \lambda_3 \bigl (\sigma_2(\lambda_{;3}) -3 \lambda_4^2 \bigr  ) \\ 
&\kern22em + \lambda_1 \lambda_2 \lambda_4 (\lambda_1 + \lambda_2 - 2 \lambda_4) \Bigr) \\
&\kern2em +  c_0 c_2 \cc \bigl( \sigma_2(\lambda)(\lambda_1 + \lambda_2 -2 \lambda_4) + \lambda_3 (\sigma_2(\lambda_{;3}) -3 \lambda_4^2   )   \bigr) \\
&\kern2em - 2c_0 c_1 \ct \cc  \bigl(  \sigma_1(\lambda)(\lambda_1 + \lambda_2 - 2\lambda_4) + \sigma_2(\lambda) -3\lambda_4(\lambda_3+\lambda_4)     \bigr). 
\end{align*}We use the equation $f_t = h$ to simplify these expressions, for $r_1(\lambda)$, we have

\begin{align*}
\label{eq:4.30}
r_1(\lambda) 
&= hc_2   \lambda_3   \lambda_1 \lambda_2 \lambda_3 \lambda_4 - 2h c_1 \ct   \lambda_1 \lambda_2 \lambda_3 \lambda_4  + c_2^2  \lambda_1 \lambda_2 \lambda_4 \tag{4.30}\\
&= \lambda_1 \lambda_2 \lambda_4 \bigl ( h c_2 \lambda_3^2  -2 h c_1   \ct \lambda_3 + c_2^2    \bigr) \\
&= \lambda_1 \lambda_2 \lambda_4 \Bigl( h c_2 \bigl( \lambda_3 -  {c_1 \ct}/{c_2} \bigr)^2 +    \bigl({c_2^3 - h c_1^2 \ctt }\bigr)/{c_2}       \Bigr) > 0.
\end{align*}In addition, for $r_2(\lambda)$, by using the equation $f_t = h$ several times, we find

\begin{align*}
\label{eq:4.31}
&\kern-1em r_2(\lambda) \tag{4.31}  
 = c_2^2         \bigl (\sigma_2(\lambda_{;3}) -3 \lambda_4^2  \bigr  )      \lambda_1 \lambda_2 \lambda_4 +  h c_2     \lambda_3 \bigl (\sigma_2(\lambda_{;3}) -3 \lambda_4^2  \bigr  )    \lambda_1 \lambda_2 \lambda_3 \lambda_4 \\ 
&  \kern2em -2 h c_1 \ct    \bigl (\sigma_2(\lambda_{;3}) -3 \lambda_4^2  \bigr  )   \lambda_1 \lambda_2 \lambda_3 \lambda_4  + c_2^2       \lambda_3 (\lambda_1 + \lambda_2 -2\lambda_4)      \lambda_1 \lambda_2 \lambda_4    \\
& \kern2em     -2c_1 c_2          \ct   (\lambda_1 + \lambda_2 - 2 \lambda_4)    \lambda_1 \lambda_2 \lambda_4   -2 h c_1 \ct     \lambda_3 (\lambda_1 + \lambda_2 -2\lambda_4)    \lambda_1 \lambda_2 \lambda_3 \lambda_4 \\
&  \kern2em      + h c_0 \cc (\lambda_1 + \lambda_2 -2\lambda_4)\lambda_1 \lambda_2 \lambda_3 \lambda_4 \\
&= \lambda_1 \lambda_2 \lambda_4 \Bigl[  \bigl (\sigma_2(\lambda_{;3}) -3 \lambda_4^2  \bigr  ) \bigl(   h c_2 \lambda_3^2   -2 h c_1 \ct \lambda_3 + c_2^2  \bigr )   \\
&\kern5em   +   (\lambda_1 + \lambda_2 - 2\lambda_4) \bigl(    -2 h c_1 \ct \lambda_3^2 +  c_2^2       \lambda_3+  h c_0 \cc \lambda_3 \\ 
&\kern28em -2c_1 c_2 \ct  \bigr)  \Bigr] \\
&= \lambda_1 \lambda_2 \lambda_4 \Bigl[  A \bigl(   h c_2 \lambda_3^2   -2 h c_1 \ct \lambda_3 + c_2^2  \bigr )   \\
&\kern5em   +  B \bigl(    -2 h c_1 \ct \lambda_3^2 +  c_2^2       \lambda_3+  h c_0 \cc \lambda_3  -2c_1 c_2 \ct  \bigr)  \Bigr],
\end{align*}
where we denote $A = \sigma_2(\lambda_{;3}) -3 \lambda_4^2$ and $B= \lambda_1 + \lambda_2 - 2\lambda_4$. If we can show that the following quantity is always non-negative, then quantity (\ref{eq:4.31}) will also be non-negative
\begin{align*}   
\label{eq:4.32}
\kern1em
&\kern-1em \cc c_0  +  \Bigl(  \frac{c_2^2  }{h } -\frac{2  {A}c_1 \ct}{B} \Bigr) + \Bigl(  \frac{{A}c_2}{B }  -  2c_1 \ct     \Bigr) \lambda_3  +  \frac{  {A}{B^{-1}}    {     c_2^2} - {2c_1c_2 \ct}      }{h \lambda_3} \tag{4.32} \\
&= \cc c_0  +  \Bigl(  \frac{c_2^2  }{h } -\frac{2  {A}c_1 \ct}{B} \Bigr) + \Bigl(  \frac{{A}c_2}{B } - 2c_1 \ct    \Bigr) \bigl( \lambda_3  + \frac{c_2}{h \lambda_3} \bigr).
\end{align*}

We claim that $A \geq 2 h^{-1/2} c_2^{1/2}  B$. To prove this claim, for $k \geq 0$, define
\begin{align*}
h_k (\tilde{\lambda}_1, \tilde{\lambda}_2, \lambda_4) &=  \sigma_2(\lambda_{;3}) -3 \lambda_4^2 - k (\lambda_1 + \lambda_2 -2 \lambda_4) \\
&= \lambda_1 \lambda_2 + \lambda_1 \lambda_4 + \lambda_2 \lambda_4 -3 \lambda_4^2 -k(\lambda_1 + \lambda_2 -2 \lambda_4) \\
&= \tilde{\lambda}_1  \tilde{\lambda}_2 - 4 \lambda_4^2 -k (\tilde{\lambda}_1 + \tilde{\lambda}_2 -4 \lambda_4),
\end{align*}where we denote $\tilde{\lambda}_1 = \lambda_1 + \lambda_4$ and $\tilde{\lambda}_2 = \lambda_2 + \lambda_4$. They satisfy the following constraints 
\begin{align*}
\tilde{\lambda}_1 &= \lambda_1 + \lambda_4 \geq 2 \sqrt{\lambda_1 \lambda_4} > 2 h^{-1/2} \sqrt{c_2};  \\
\tilde{\lambda}_2 &= \lambda_1 + \lambda_4 \geq 2 \sqrt{\lambda_2 \lambda_4} > 2 h^{-1/2} \sqrt{c_2};  \\
 \lambda_4 &<  h^{-1/2} \sqrt{c_2}.
\end{align*}By taking the partial derivatives of $h_k$, we get
\begin{align*}
\frac{\partial h_k}{\partial \tilde{\lambda}_1} = \tilde{\lambda}_2 - k; \quad \frac{\partial h_k}{\partial \tilde{\lambda}_2} = \tilde{\lambda}_1 - k; \quad \frac{\partial h_k}{\partial  {\lambda}_4} = -8 \lambda_4 +4k.
\end{align*}So the infimum happens when $\tilde{\lambda}_1 = \tilde{\lambda}_2 = \max\{k, 2 h^{-1/2} \sqrt{c_2}  \}$ and $\lambda_4 = 0$ or $\lambda_4 = h^{-1/2} \sqrt{c_2} $. If $k > 2 h^{-1/2} \sqrt{c_2} $, then we may check that
\begin{align*}
h_k(k, k, 0) = -k^2 < 0;\quad h_k(k, k, h^{-1/2} \sqrt{c_2} ) = - \bigl(k - 2 h^{-1/2} \sqrt{c_2}  \bigr)^2 < 0.
\end{align*}On the other hand, if $k \leq 2 \sqrt{c_2} h^{-1/2}$, then we get
\begin{align*}
h_k\bigl(2 h^{-1/2} \sqrt{c_2}, 2 h^{-1/2} \sqrt{c_2}, 0\bigr) &=   4h^{-1} c_2  -  2k h^{-1/2}\sqrt{c_2}  \geq 0;    \\
h_k\bigl(2  h^{-1/2} \sqrt{c_2}, 2  h^{-1/2} \sqrt{c_2},  h^{-1/2}  \sqrt{c_2} \bigr) &= 0.
\end{align*}So by the arguments above, we always have $A - 2 \sqrt{c_2} h^{-1/2} B \geq 0$, which proves the claim. Thus, by the claim, quantity (\ref{eq:4.32}) becomes 
\begin{align*}
&\kern-1em \cc c_0  +  \Bigl(  \frac{c_2^2  }{h } -\frac{2  {A}c_1 \ct}{B} \Bigr) + \Bigl(  \frac{{A}c_2}{B } - 2c_1 \ct    \Bigr) \bigl( \lambda_3  + \frac{c_2}{h \lambda_3} \bigr) \\
&\geq \cc c_0  + \Bigl( \frac{c_2^2  }{h } - \frac{2  {A}c_1 \ct }{B} \Bigr) + 2 h^{-1/2} {c_2}^{1/2}   \Bigl(  \frac{{A}c_2}{B } - 2c_1 \ct    \Bigr) \\
&=  \cc c_0  + h^{-1} c_2^2  - 4 h^{-1/2} c_1  {c_2}^{1/2}  \ct + 2 \bigl( h^{-1/2} c_2^{3/2}  - c_1 \ct  \bigr) \frac{A}{B}  \\
&\geq  \cc c_0  +  5 h^{-1} c_2^2  - 8 h^{-1/2} c_1  {c_2}^{1/2}  \ct.
\end{align*}Then by our hypothesis, we obtain
\begin{align*}
\label{eq:4.33}
&\kern-1em \cc c_0  +  5 h^{-1} c_2^2  - 8 h^{-1/2} c_1  {c_2}^{1/2}  \ct \tag{4.33} \\
&> -24 h^{-1} c_2^2   \cos^2(\theta_{h, c_1, c_2})\cos(2\theta_{h, c_1, c_2})  +  5 h^{-1} c_2^2   - 8 h^{-1/2}  c_1  {c_2}^{1/2} \ct.
\end{align*}
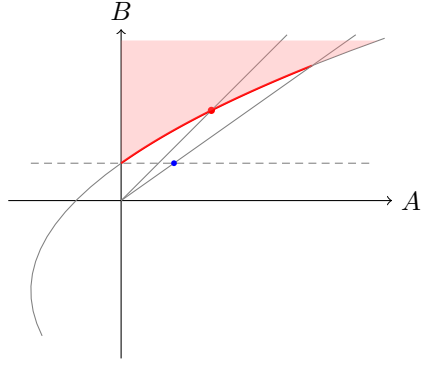
\begin{figure}
\centering
\begin{tikzpicture}[scale=0.6]
  \draw[->] (-2.5,0) -- (6,0) node[right] {$A$};
  \draw[->] (0,-3.5) -- (0,3.8) node[above] {$B$};
  \draw[color=gray,domain=-3:3.6]    plot ({(\x+2)*(\x+2)/4 -2},\x)             ;
  \draw[color=red,thick,domain={2*sqrt(2)-2}:{2*sqrt(2)-2+sqrt(16-8*sqrt(2))},samples=101]    plot ({(\x+2)*(\x+2)/4 -2},\x)             ;
  \draw[color=gray,domain=0:5.2]    plot (\x,{\x/sqrt(2)})             ;
  \draw[color=gray,domain=0:{5.2/sqrt(2)}]    plot (\x,{\x})             ;
  \draw[color=gray,densely dashed,domain=-2:5.5] plot (\x,{2*sqrt(2)-2})             ;
  \fill[blue] ({4-2*sqrt(2)},{2*sqrt(2)-2}) circle (0.064cm);
  \fill[red] (2,{2}) circle (0.08cm);
\begin{scope}
    \fill[color=red!50,opacity=0.3,thick,domain={2*sqrt(2)-2}:3.55,samples=101]
    (0, 3.55) -- plot ({(\x+2)*(\x+2)/4 -2},\x) ;
  \end{scope}
\end{tikzpicture}
\caption{Infimum on the defining region and another smaller value}
\label{fig:4.1}
\end{figure}

A clever way is to consider the test function $-c^2 + c^{-1}d^2 - c A + d B - c^2  {B^2}/{A}$ in the proof of \hyperlink{T:2.1}{Positivstellensatz Theorem}, where $c = c_2/h$ and $d = 2c_1 \ct/h$. By taking a pair $(A, B)$ outside the defining region, which is the blue dot in Figure \ref{fig:4.1}, we will get a value smaller than the infimum on the defining region. Here, by picking $A = c^{1/2} B$ and $B = 2c^{1/2} - c^{-1}d$, we get
\begin{align*}
-c^2 + c^{-1}d^2 - c A + d B - c^2  \frac{B^2}{A} &= -c^2 + c^{-1}d^2 - ( 2c^{3/2} -d)B = -c^2 + c^{-1}d^2 - c^{-1} ( 2c^{3/2} -d)^2 \\
&= -5c^2 + 4 c^{1/2}d = -5 h^{-2}c_2^2  + 8h^{-3/2} c_1 c_2^{1/2} \ct.
\end{align*}On the other hand, by the Positivstellensatz Theorem~\hyperlink{T:2.1}{(e)}, the infimum on the defining region will be $-24 c_2^2 h^{-1}  \cos^2(\theta_{h, c_1, c_2})\cos(2\theta_{h, c_1, c_2})$, which will be greater than $ -5 h^{-2}c_2^2  + 8h^{-3/2} c_1 c_2^{1/2} \ct$. So quantity (\ref{eq:4.33}) will always be non-negative, which implies quantity (\ref{eq:4.32}) will always be non-negative, which finishes the proof.
\end{proof}

Then, by taking the first and second derivatives of equation (\ref{eq:4.19}), we have the following Lemma. The proof should be straightforward, we consider Lemma~\hyperlink{L:2.1}{2.1}, Lemma~\hyperlink{L:2.2}{2.2}, Lemma~\hyperlink{L:2.3}{2.3}, Lemma~\hyperlink{L:2.4}{2.4}, and Lemma~\hyperlink{L:4.11}{4.11}. 

\hypertarget{L:4.14}{\begin{flemma}}
Let $F_t(  {\Lambda} )= {h}(p)$, then we have
\begin{align*}
 \frac{\partial {h}}{\partial  {z}_k}  = \sum_{i ,j}  \frac{\partial F_t (   {\Lambda}     ) }{\partial  {\Lambda}^j_i}    \frac{\partial  {\Lambda}^j_i }{\partial  {z}_k}; \quad  \frac{\partial^2 {h}}{\partial z_k \partial \bar{z}_k} =  \sum_{i ,j}  \frac{\partial^2 F_t (   {\Lambda}     ) }{\partial \Lambda^j_i \partial \Lambda_r^s}    \frac{\partial  {\Lambda}_i^j }{\partial \bar{z}_k} \frac{\partial  {\Lambda}_r^s}{\partial z_k} + \sum_{i ,j}  \frac{\partial F_t (   {\Lambda}     )}{\partial \Lambda_i^j }    \frac{\partial^2  {\Lambda}_i^j }{\partial z_k \partial \bar{z}_k}.
\end{align*}In particular, at the maximum point $p \in M$ of $\tilde{U}$, we have
\begin{align*}
\label{eq:4.34}
h_k &=  \sum_i f_i   (  X_u     )_{i \bar{i},  {k}};     \tag{4.34}   \\
\label{eq:4.35}
h_{k \bar{k}} &= \sum_{i, j}  f_{ij}   (  X_u     )_{i \bar{i}, \bar{k}}    (  X_u     )_{j \bar{j}, {k}} + \sum_{i \neq j}  \frac{f_i - f_j}{\lambda_i - \lambda_j}    | (X_u)_{j \bar{i},k}    |^2     +  \sum_i \Bigl( f_i  ( X_u   )_{i \bar{i},k \bar{k}} -f_i \lambda_i \omega_{i \bar{i},k \bar{k}} \Bigr) \tag{4.35}  \\
&=\sum_{i, j}  f_{ij}   (  X_u     )_{i \bar{i}, \bar{k}}    (  X_u     )_{j \bar{j}, {k}}  + \sum_{i \neq j} \Bigl (  \frac{h}{\lambda_i \lambda_j} - \frac{ c_2(t)   }{\lambda_1 \lambda_2 \lambda_3 \lambda_4  } \Bigr) | (X_u)_{j \bar{i},k}    |^2   \\
&\kern2em +  \sum_i \Bigl( f_i  ( X_u   )_{i \bar{i},k \bar{k}} -f_i  {\lambda}_i \omega_{i \bar{i},k \bar{k}} \Bigr) .  
\end{align*}
\end{flemma}

For the remainder of this subsection, we let $O_i$ be the Big $O$ notation that describes the limiting behavior when $\lambda_i$ approaches infinity. So $O_i(1)$ means the quantity will be bounded by a uniform constant if $\lambda_i$ is sufficiently large. 

\hypertarget{L:4.15}{\begin{flemma}}
There exists uniform constants $N > 0$ and $\kappa > 0$, which are independent of $t \in [0, 1]$, such that if $\lambda_1 > N$, then
\begin{align*}
\sum_i f_i u_{i \bar{i}} \geq -\kappa \sum_i f_i.
\end{align*}
\end{flemma}


\begin{proof}
First, by Lemma~\hyperlink{L:3.5}{3.5}, if $X$ is a $C$-subsolution to equation (\ref{eq:4.19}), then $h X^3 -3   \omega^2 \wedge X + 2 \ct   \omega^3 > 0;\quad   h X^2 -    \omega^2 > 0; \quad  X> 0$. We may fix $\delta > 0$ sufficiently small such that
\begin{align*}
(1-\delta) h X^2 -    \omega^2  > 0;\quad (1-\delta) h X^3 -3   \omega^2 \wedge X + 2 \ct   \omega^3 > 0.
\end{align*}
Also, we choose $\kappa > 0$ to be also small enough so that $ (1-\delta) h (X - \kappa \omega)^3 -3   \omega^2 \wedge (X - \kappa \omega ) + 2 \ct   \omega^3 > 0;   (1-\delta) h   ( X - \kappa \omega   )^2 >   \omega^2;   X - \kappa \omega > 0$. Due to the fact that $c_2$ is decreasing, $c_2(0) > 0$, and the choice of $\delta$, we get
\begin{align*}
(1-\delta) h ( X - \kappa \omega   )^2 > c_2(t)   \omega^2.
\end{align*}
%

Note that $u_{ i \bar{i}} = \lambda_i - X_{i \bar{i}}$, we can write


\begin{align*}
\label{eq:4.36}
\sum_i f_i   ( u_{i \bar{i}}  + \kappa     ) &= \sum_i f_i \bigl ( \lambda_i - X_{i \bar{i}}  + \kappa   \bigr ) \tag{4.36}\\ 
&= \sum_i  \frac{- c_2 \sigma_2(\lambda_{;i}) + 2 c_1 \cot(\hat{\theta}) \sigma_1(\lambda_{;i}) - c_0 (3 \csc^2(\hat{\theta}) -4) }{ \lambda_1 \lambda_2 \lambda_3 \lambda_4 \lambda_i }     ( \lambda_i - X_{i \bar{i}}  + \kappa     ) \\
&=  - \frac{ \sum_i A_i }{\lambda_1 \lambda_2 \lambda_3 \lambda_4}  +  \sum_i   ( X_{i \bar{i}} - \kappa     ) \frac{ A_i }{ \lambda_1 \lambda_2 \lambda_3 \lambda_4 \lambda_i },
\end{align*}where we denote $A_i = c_2 \sigma_2(\lambda_{;i}) - 2 c_1 \cot(\hat{\theta}) \sigma_1(\lambda_{;i}) + c_0 \cc$ for $i \in \{1, 2, 3, 4\}$. Since $\lambda_1 \geq \lambda_2 \geq \lambda_3 \geq \lambda_4$, which implies $A_4 \geq A_3 \geq A_2 \geq A_1 > 0$.\bigskip


There are two cases to be considered: \smallskip

$\bullet$ If $0 < \lambda_4 \leq \frac{X_{4 \bar{4}} - \kappa}{4}$, then 
\begin{align*}
( X_{4 \bar{4}} - \kappa     ) \frac{ A_4 }{ \lambda_1 \lambda_2 \lambda_3 \lambda_4 \lambda_4 } \geq 4 \frac{ A_4}{\lambda_1 \lambda_2 \lambda_3 \lambda_4}.
\end{align*}So
\begin{align*}
\sum_i f_i   ( u_{i \bar{i}}  + \kappa     ) &\geq -4 \frac{ A_4 }{\lambda_1 \lambda_2 \lambda_3 \lambda_4}   +  \sum_i   ( X_{i \bar{i}} - \kappa     ) \frac{ A_i }{ \lambda_1 \lambda_2 \lambda_3 \lambda_4 \lambda_i }   \geq 0.
\end{align*}

$\bullet$ If $\lambda_4 \geq \frac{ X_{4 \bar{4}} - \kappa}{4}$, similar as before, then we can show that $\lambda_3$ is bounded from above. We can simplify equation (\ref{eq:4.36}) to
\begin{align*}
\label{eq:4.37}
\kern1em
&\kern-1em \sum_i f_i   ( u_{i \bar{i}}  + \kappa     ) \tag{4.37}  \\
&= \frac{-2c_2 \sigma_2(\lambda) + 6c_1 \ct \sigma_1(\lambda) -4c_0 \cc }{\lambda_1 \lambda_2 \lambda_3 \lambda_4}  +  \sum_i   ( X_{i \bar{i}} - \kappa     ) \frac{ A_i }{ \lambda_1 \lambda_2 \lambda_3 \lambda_4 \lambda_i } \\
&\geq  \frac{ -2c_2 \bigl ( \lambda_2 + \lambda_3 + \lambda_4\bigr ) + 6c_1 \ct }{\lambda_2 \lambda_3 \lambda_4} + \lambda_1^{-1} \cdot O_1(1)   +  ( X_{2 \bar{2}} - \kappa  ) \frac{c_2 (\lambda_3 + \lambda_4) - 2c_1 \ct }{\lambda_2^2 \lambda_3 \lambda_4}         \\
&\kern2em         +  ( X_{3 \bar{3}} - \kappa  ) \frac{c_2 (\lambda_2 + \lambda_4) - 2c_1\ct }{\lambda_2 \lambda_3^2 \lambda_4}  +  ( X_{4 \bar{4}} - \kappa  ) \frac{c_2 (\lambda_2 + \lambda_3) - 2c_1 \ct }{\lambda_2 \lambda_3 \lambda_4^2}.
\end{align*}\bigskip

In this case, if $\lambda_2$ is also large, then we can have a better estimate for $\lambda_3 \lambda_4$, that is,
\begin{align*}
\lambda_3 \lambda_4 &= \frac{c_2 \lambda_3 \sigma_2(\lambda_{;4}) -2c_1 \ct \lambda_3 \sigma_1(\lambda_{;4}) + c_0 \cc \lambda_3 }{ h \lambda_1 \lambda_2 \lambda_3 - c_2 \sigma_1(\lambda_{;4}) + 2c_1 \ct} 
\xrightarrow[\text{ as } \lambda_2   \rightarrow \infty]{} \frac{c_2}{h}.
\end{align*}Here, since we assume $\lambda_1 \geq \lambda_2$, if $\lambda_2    \rightarrow \infty$, then $\lambda_1    \rightarrow \infty$ as well. If $\lambda_2$ is sufficiently large, then 
\begin{align*}
\label{eq:4.38}
\sqrt{\lambda_3 \lambda_4} < (1 +  {\delta}/{4})  {c^{1/2}_2 } h^{-1/2}. \tag{4.38}
\end{align*}

By combining inequalities (\ref{eq:4.37}) and (\ref{eq:4.38}), we may write
\begin{align*}
\sum_i f_i   ( u_{i \bar{i}}  + \kappa     ) &\geq  c_2 \frac{  ( X_{3 \bar{3}} - \kappa  ) \lambda_4 + ( X_{4 \bar{4}} - \kappa  ) \lambda_3  -2\lambda_3 \lambda_4}{\lambda_3^2 \lambda_4^2}  + \lambda_1^{-1} \cdot O_1(1) + \lambda_2^{-1} \cdot O_2(1) \\
&\geq 2c_2 \frac{ \sqrt{ ( X_{3 \bar{3}} - \kappa  )   ( X_{4 \bar{4}} - \kappa  ) }    - \sqrt{\lambda_3 \lambda_4}}{\lambda_3^{3/2} \lambda_4^{3/2}}  + \lambda_1^{-1} \cdot O_1(1) + \lambda_2^{-1} \cdot O_2(1) \\
&\geq 2c_2^{3/2} h^{-1/2} \frac{    (1- \delta)^{-1/2}      -  (1 +  {\delta}/{4}) }{\lambda_3^{3/2} \lambda_4^{3/2}}  + \lambda_1^{-1} \cdot O_1(1) + \lambda_2^{-1} \cdot O_2(1) \\
&\geq c_2^{3/2}(0) \cdot h^{-1/2} \cdot \frac{    \delta/2 }{\lambda_3^{3/2} \lambda_4^{3/2}}  + \lambda_1^{-1} \cdot O_1(1) + \lambda_2^{-1} \cdot O_2(1).
\end{align*}Here, because $\lambda_4$ has a lower bound, otherwise we will not get another lower order term $\lambda_2^{-1} \cdot O_2(1)$. Now since $c_2$ is increasing with $c_2(0) > 0$ and $\lambda_3, \lambda_4$ have a positive lower bound, for sufficiently large $\lambda_2$, $\sum_i f_i   ( u_{i \bar{i}}  + \kappa     )$ will be non-negative. \bigskip


If $\lambda_2$ is also bounded from above, then we need to estimate it more carefully. From equation (\ref{eq:4.36}), we have
\begin{align*}
\label{eq:4.39}
 \sum_i f_i   ( u_{i \bar{i}}  + \kappa     ) \geq  - ( X_{2 \bar{2}} - \kappa - \lambda_2 ) g_2 - ( X_{3 \bar{3}} - \kappa - \lambda_3 ) g_3  - ( X_{4 \bar{4}} - \kappa - \lambda_4 ) g_4 + \lambda_1^{-1} \cdot O_1(1), \tag{4.39}
\end{align*}where we denote $g_t = g_t(\lambda_2, \lambda_3, \lambda_4) = \bigl( c_2 (\lambda_2 + \lambda_3 + \lambda_4) -2c_1 \ct  \bigr)/\lambda_2 \lambda_3 \lambda_4$ and $g_j = \partial g_t/\partial \lambda_j$ with $j \in \{2, 3, 4\}$. We have


\begin{align*}
\lambda_1 = \frac{c_2 \sigma_2(\lambda_{;1} )  -2c_1 \ct \sigma_1(\lambda_{;1}) + c_0 \cc}{ h \lambda_2 \lambda_3 \lambda_4  - c_2   \sigma_1(\lambda_{;1}) + 2c_1 \ct}.
\end{align*}Since $\lambda_2, \lambda_3, \lambda_4$ are all bounded, say $S \geq \lambda_2, \lambda_3, \lambda_4 \geq s = (X_{4\bar{4}} -\kappa)/4 > 0$, the numerator will also be bounded from above. If $\lambda_1$ approaches infinity, then $h \lambda_2 \lambda_3 \lambda_4  - c_2   \sigma_1(\lambda_{;1}) + 2c_1 \ct$ will approach $0$. If $\lambda_1$ is sufficiently large and $\lambda_2, \lambda_3, \lambda_4$ are all bounded, then we obtain
\begin{align*}
\label{eq:4.40}
\delta h \lambda_2 \lambda_3 \lambda_4/2 > h \lambda_2 \lambda_3 \lambda_4  - c_2   \sigma_1(\lambda_{;1}) + 2c_1 \ct > 0. \tag{4.40}
\end{align*}On the other hand, since 
\begin{align*}
(1- \delta ) h (X - \kappa \omega)^3 -3   \omega^2 \wedge (X - \kappa \omega ) + 2 \ct   \omega^3 > 0,
\end{align*}which implies
\begin{align*}
\label{eq:4.41}
0  &< 6 (1-\delta)  h \det \bigl( \tilde{X} - \kappa \tilde{\omega}   \bigr) - 6   \sum_{j=2}^4 (X_{j \bar{j}} - \kappa ) +12 \ct  \tag{4.41} \\
&\leq  6 (1-\delta) h  (X_{2 \bar{2}} - \kappa ) (X_{3 \bar{3}} - \kappa ) (X_{4 \bar{4}} - \kappa )   - 6   \sum_{j=2}^4 (X_{j \bar{j}} - \kappa ) +12 \ct.
\end{align*}Here $\tilde{X} - \kappa \tilde{\omega}$ is a positive definite Hermitian matrix obtained by evaluating $X - \kappa \omega$ on the tangent subspace spanned by $\bigl\{ \partial/\partial z^2, \partial/\partial \bar{z}^2, \partial/\partial z^3, \partial/\partial \bar{z}^3, \partial/\partial z^4, \partial/\partial \bar{z}^4 \bigr\}$. Notice that the last inequality is due to Hadamard's inequality, that is, the determinant of a positive-semidefinite Hermitian matrix is less than or equal to the product of its diagonal entries. Inequality (\ref{eq:4.41}) gives us that
\begin{align*}
&\kern-1em (1- \tilde{\delta}) h (X_{2 \bar{2}} - \kappa ) (X_{3 \bar{3}} - \kappa ) (X_{4 \bar{4}} - \kappa )  - c_2   \sum_{j=2}^4 (X_{j \bar{j}} - \kappa ) + 2c_1 \ct \\
&>   ( \delta -\tilde{\delta}) h  (X_{2 \bar{2}} - \kappa ) (X_{3 \bar{3}} - \kappa ) (X_{4 \bar{4}} - \kappa ) + (1-c_2)  \sum_{j=2}^4 (X_{j \bar{j}} - \kappa )  + 2 (c_1-1)\ct  \\
&\geq  ( \delta -\tilde{\delta}) h  (X_{2 \bar{2}} - \kappa ) (X_{3 \bar{3}} - \kappa ) (X_{4 \bar{4}} - \kappa )  \geq  \delta   h  (X_{2 \bar{2}} - \kappa ) (X_{3 \bar{3}} - \kappa ) (X_{4 \bar{4}} - \kappa )/2 > 0,
\end{align*}where $\tilde{\delta} \in [0, \delta/2]$. The inequality on the second line is because of the quantity $(1-c_2)  \sum_{j=2}^4 (X_{j \bar{j}} - \kappa )  + 2 (c_1-1)\ct$ being decreasing, thus the quantity will obtain its minimum when $t = 1$. We can also check this fact by applying the Cauchy--Schwarz inequality, we get
\begin{align*}
c_2' \bigl(  ( X_{j \bar{j}} - \kappa ) + ( X_{k \bar{k}} - \kappa ) \bigr) &\geq 2 c_2' \sqrt{ ( X_{j \bar{j}} - \kappa )  ( X_{k \bar{k}} - \kappa )} \geq 2 (1-\delta)^{-1/2}\sqrt{c_2} c_2' h^{-1/2}   \geq 2 \sqrt{c_2} c_2' h^{-1/2},
\end{align*}where $j \neq k \in \{2, 3, 4\}$. Thus, by the \hyperlink{dim 4 cons}{$\Upsilon$-cone constraints}, we have
\begin{align*}
c_2' \sum_{j=2}^4 (X_{j \bar{j}} - \kappa ) \geq  3 \sqrt{c_2} c_2' h^{-1/2} \geq 2 c_1' \ct.
\end{align*}

So for $\tilde{\delta} \in [0, \delta/2]$, $(X_{2 \bar{2}}- \kappa, X_{3 \bar{3}}- \kappa, X_{4 \bar{4}}- \kappa)$ will be in the enclosed region of $\{ g_t = (1 - \tilde{\delta}) h \} \cap \Upsilon^3_{2; (1-\tilde{\delta})h, 0, -c_2} \cap \Upsilon^3_{1; 1, 0}$, where $g_t = g_t(y, z, w) = \bigl({ c_2(t) (y + z + w) -2c_1(t) \ct  }\bigr)/{y z w}$. Also, $(\lambda_2, \lambda_3, \lambda_4)$ will satisfy 
\begin{align*}
h > g_t(\lambda_2, \lambda_3, \lambda_4) > (1-\delta/2) h.
\end{align*} 
In fact, $g_t$ is the continuity path when we are solving the complex three dimensional dHYM equation. We can similarly show that the set $\{ g_t = (1 - \tilde{\delta}) h \}$ is convex for any $\tilde{\delta} \in [0, \delta/2]$ and $t \in [0, 1]$ as before, provided that $\delta > 0$ is small. \bigskip

Now, by the supporting hyperplane theorem, if we let $x$ be a fixed point in the interior of a convex set, then for any point $p$ on the boundary of this convex set, the inner product of the vector $x-p$ and the inner normal vector at $p$ will always be positive. We may view the inner product $- ( X_{2 \bar{2}} - \kappa - \lambda_2 ) g_2 - ( X_{3 \bar{3}} - \kappa - \lambda_3 ) g_3  - ( X_{4 \bar{4}} - \kappa - \lambda_4 ) g_4$ as a continuous function defined on the following compact subset of $\mathbb{R}^3_{y, z, w} \times [0, 1]$:
\begin{align*}
\bigcup_{\substack{\tilde{\delta} \in [0, \delta/2] \\t \in [0, 1] }} \Bigl( \{ g_t = (1 - \tilde{\delta}) h \} \cap [s, S]^3 \Bigr) \times \{t\}.
\end{align*}



Here, for convenience, we denote the solution set $\{ g_t = (1 - \tilde{\delta}) h \}$ as the connected component $\{ g_t = (1 - \tilde{\delta}) h \} \cap \bigl( \Upsilon^3_{2; (1 - \tilde{\delta})h, 0, -c_2} \cap \Upsilon^3_{1; 1, 0} \bigr)$. So we have a positive lower bound for the inner product, since on this compact set the continuous function is always positive. In conclusion, we can find a uniform $N>0$ such that if $\lambda_1 > N$, we have
\begin{align*}
 \sum_i f_i   ( u_{i \bar{i}}  + \kappa     ) \geq  - ( X_{2 \bar{2}} - \kappa - \lambda_2 ) f_2 - ( X_{3 \bar{3}} - \kappa - \lambda_3 ) f_3  - ( X_{4 \bar{4}} - \kappa - \lambda_4 ) f_4 + \lambda_1^{-1} \cdot O_1(1) \geq 0.
\end{align*}\end{proof}

\hypertarget{L:4.16}{\begin{flemma}}
There exists a uniform $N > 0$ and a constant $\tilde{\mathcal{F}}_0 > 0$ independent of $t \in [0, 1]$ such that if $\lambda_1 > N$, then we have the following estimates
\begin{align*}
-f_2 - f_3 -f_4\geq \tilde{\mathcal{F}}_0 > 0.
\end{align*}
\end{flemma}


\begin{proof}
First, if $\lambda_1$ is sufficiently large, then we have
\begin{align*}
\label{eq:4.42}
-f_2 - f_3 -f_4 
&=  \frac{c_2(\lambda_3 + \lambda_4) -2c_1\ct}{\lambda_2^2 \lambda_3 \lambda_4}  + \frac{c_2(\lambda_2 + \lambda_4) -2c_1\ct }{\lambda_2 \lambda_3^2 \lambda_4} \tag{4.42}  \\
&\kern2em  + \frac{c_2 \sigma_2(\lambda_{;4}) -2c_1\ct  \sigma_1(\lambda_{;4}) + c_0 \cc }{ \lambda_1\lambda_2 \lambda_3 \lambda_4^2}    + \lambda_1^{-1} \cdot O_1(1) \\
&=  \frac{c_2(\lambda_3 + \lambda_4) -2c_1\ct}{\lambda_2^2 \lambda_3 \lambda_4}  + \frac{c_2(\lambda_2 + \lambda_4) -2c_1\ct}{\lambda_2 \lambda_3^2 \lambda_4} + \frac{ h  }{   \lambda_4 }  \\
&\kern2em    - \frac{c_2(\lambda_1 + \lambda_2 + \lambda_3) - 2c_1 \ct}{\lambda_1 \lambda_2 \lambda_3 \lambda_4}  + \lambda_1^{-1} \cdot O_1(1).
\end{align*}If $\lambda_2$ is also sufficiently large, then we obtain 
\begin{align*}
\label{eq:4.43}
 -f_2 - f_3 -f_4  \tag{4.43}
&= \frac{c_2}{\lambda_3^2 \lambda_4} + \frac{h }{  \lambda_4} + \lambda_2^{-1}\cdot O_2(1) + \lambda_1^{-1}\cdot O_1(1) \geq   \frac{h }{  \lambda_4} + \lambda_2^{-1}\cdot O_2(1) + \lambda_1^{-1}\cdot O_1(1).
\end{align*}We have
\begin{align*}
\label{eq:4.44}
\lambda_4 &= \frac{c_2 \sigma_2(\lambda_{; 4}) -2c_1 \ct \sigma_1(\lambda_{;4}) + c_0 \cc }{h \lambda_1 \lambda_2 \lambda_3 - c_2 (\lambda_1 + \lambda_2 + \lambda_3) + 2 c_1 \ct} \leq \frac{4 c_2 }{h \lambda_3} <  {4 c_2^{1/2} h^{-1/2} }   \tag{4.44}
\end{align*} provided that $\lambda_2$ is large. The last inequality is by the lower bound of $\lambda_3$, that is, $h \lambda_3^2 >   c_2$. Combining inequalities (\ref{eq:4.43}) and (\ref{eq:4.44}), we get
\begin{align*}
 -f_2 - f_3 -f_4   &\geq   \frac{h }{  \lambda_4} + \lambda_2^{-1}\cdot O_2(1) + \lambda_1^{-1}\cdot O_1(1) > \frac{h^{3/2}}{4 c_2^{1/2}} + \lambda_2^{-1}\cdot O_2(1) + \lambda_1^{-1}\cdot O_1(1) > \frac{h^{3/2}}{8 },
\end{align*}provided that $\lambda_2$ is sufficiently large.\bigskip


On the other hand, if $\lambda_2$ is bounded from above, say $S$, then $\lambda_3$ is bounded from above, which implies that $\lambda_4$ is bounded from below. Hence, say $S \geq \lambda_2 \geq \lambda_3 \geq \lambda_4 \geq s$, for $S$ and $s$ both greater than zero. If $\lambda_1$ is sufficiently large, then by inequality (\ref{eq:4.40}), we have
\begin{align*}
\delta h \lambda_2 \lambda_3 \lambda_4/2 > h \lambda_2 \lambda_3 \lambda_4  - c_2   \sigma_1(\lambda_{;1}) + 2c_1 \ct > 0. 
\end{align*}By Lemma~\hyperlink{L:4.3}{4.3}, if $\delta > 0$ small, then the solution set $\{ g_t = h(1 -\delta) \}$ is convex. Similar to before, we define the following quantity 
\begin{align*}
\mathcal{F}_0 \coloneqq \min_{(\lambda_2, \lambda_3, \lambda_4, t) \in Q}  -(f_2 + f_3 + f_4) > 0,
\end{align*}where $Q$ is a compact set defined by
\begin{align*}
Q \coloneqq \bigcup_{\substack{\tilde{\delta} \in [0, \delta/2] \\t \in [0, 1] }} \Bigl( \{ g_t = (1 - \tilde{\delta}) h \} \cap [s, S]^3 \Bigr) \times \{t\} .
\end{align*}

By denoting $\tilde{\mathcal{F}}_0 \coloneqq \min \{ \mathcal{F}_0, h^{3/2}/8\} > 0$, if $\lambda_1$ is sufficiently large, then 
\begin{align*}
- f_2 - f_3 - f_4 \geq  \tilde{\mathcal{F}}_0 > 0.
\end{align*}
\end{proof}


Now we let $C$ be a constant depending only on the stated data, but which may change from line to line. We can finish the proof of the following $C^2$ estimate

\hypertarget{T:4.2}{\begin{fthm}}
Suppose $X$ is a $C$-subsolution to equation (\ref{eq:4.19}) and $u \colon M   \rightarrow \mathbb{R}$ is a smooth function solving equation (\ref{eq:4.19}). Then there exists a constant $C$ independent of $t$ such that 
\begin{align*}
|\partial \bar{\partial} u | \leq C \bigl ( 1 + \sup_M \bigl|\nabla u\bigr|^2  \bigr), 
\end{align*}where $C = C(\hat{\theta}, c_0, c_1, c_2 \osc_M u, M, X, \omega)$, $h = \sii$ is a constant, and $\nabla$ is the Levi-Civita connection with respect to $\omega$.\bigskip
\end{fthm}


\begin{proof}

First, by applying the operator \hyperlink{dim 4 operator}{$\mathcal{L}_t$} to $G(\tilde{\Lambda})$, at the maximum point, we obtain
\begin{align*}
\label{eq:4.45}
\mathcal{L}_t  \bigl (  G( \tilde{\Lambda} )   \bigr )   \tag{4.45}
&= - \sum_{i, j, k}f_k   g_{ij}  \frac{\partial \tilde{\Lambda}_i^i}{\partial z_k} \frac{\partial \tilde{\Lambda}_j^j}{\partial \bar{z}_k} - \sum_k f_k   \sum_{i \neq j} \frac{g_i - g_j}{ \tilde{\lambda}_i - \tilde{\lambda}_j}   \frac{\partial \tilde{\Lambda}_j^i}{\partial z_k} \frac{\partial \tilde{\Lambda}_i^j}{\partial \bar{z}_k} - \sum_{i,k} f_k   g_i   \frac{\partial^2 \tilde{\Lambda}_i^i}{\partial z_k \partial \bar{z}_k}    \\
&=  \sum_k f_k \frac{1}{(1 + \tilde{\lambda}_1)^2}   \bigl|    (  X_u     )_{1 \bar{1},k }    \bigr |^2 + \sum_k f_k  \frac{ {\lambda}_1}{1 + \tilde{\lambda}_1} \omega_{1\bar{1},k \bar{k}}  -  \sum_k f_k \frac{1}{1+ \tilde{\lambda}_1}   ( X_u   )_{1 \bar{1},k \bar{k}} \\
&\kern2em  - \sum_{k}f_k \sum_{j \neq 1}  \frac{1}{(1 + \tilde{\lambda}_1)( \tilde{\lambda}_1 - \tilde{\lambda}_j)}   \Bigl(    \bigl |    (  X_u    )_{j \bar{1},k}     \bigr |^2   +      \bigl |    (  X_u    )_{1 \bar{j},k}    \bigr  |^2      \Bigr ) \\
&\geq   \sum_i f_i \frac{1}{(1 + \tilde{\lambda}_1)^2}  \bigl |    (  X_u     )_{1 \bar{1}, i}    \bigr |^2 +  C \sum_i f_i  - \sum_i f_i \frac{1}{1+ \tilde{\lambda}_1}   (  X_u   )_{1 \bar{1}, i \bar{i}}.
\end{align*}

Here we change the index from $k$ to $i$ for convenience. Then by equation (\ref{eq:4.35}), we have
\begin{align*}
\label{eq:4.46}
0 = h_{k \bar{k}} &= \sum_{i, j}  f_{ij}   (  X_u     )_{i \bar{i}, \bar{k}}    (  X_u     )_{j \bar{j}, {k}} + \sum_{i \neq j}  \frac{f_i - f_j}{\lambda_i - \lambda_j}    | (X_u)_{j \bar{i},k}    |^2     +  \sum_i \Bigl( f_i  ( X_u   )_{i \bar{i},k \bar{k}} -f_i \lambda_i \omega_{i \bar{i},k \bar{k}} \Bigr) \tag{4.46} \\
&\geq \sum_{i \neq j} \Bigl (  \frac{h}{\lambda_i \lambda_j} - \frac{ c_2(t)   }{\lambda_1 \lambda_2 \lambda_3 \lambda_4  } \Bigr)   | (X_u)_{j \bar{i},k}    |^2     +  \sum_i \Bigl( f_i  ( X_u   )_{i \bar{i},k \bar{k}} -f_i \lambda_i \omega_{i \bar{i},k \bar{k}} \Bigr) \\
&\geq \sum_{i \neq j} \Bigl (  \frac{h}{\lambda_i \lambda_j} - \frac{ c_2(t)   }{\lambda_1 \lambda_2 \lambda_3 \lambda_4  } \Bigr)   | (X_u)_{j \bar{i},k}    |^2     +  \sum_i  f_i  ( X_u   )_{i \bar{i},k \bar{k}} + C \sum_i f_i \lambda_i,
\end{align*}where the inequality on the second line is due to Lemma~\hyperlink{L:4.13}{4.13}. Since the solution set $\{f_t = h\}$ is convex and $0 = h_k = \sum_i f_i (X_u)_{i \bar{i}, k}$ implies $(X_u)_{i \bar{i}, k}$ is a tangent vector, $\sum_{i, j} f_{ij}   (  X_u     )_{i \bar{i}, \bar{k}}    (  X_u     )_{j \bar{j}, {k}}  \geq 0$. Hence by setting $k = 1$, inequality (\ref{eq:4.46}) gives
\begin{align*}
\label{eq:4.47}
-\sum_i f_i   (  X_u     )_{1 \bar{1},i \bar{i}} &= - \sum_i f_i (X_u  )_{i  \bar{i},1 \bar{1}}  + \sum_i f_i \bigl (    (X_u  )_{i  \bar{i},1 \bar{1}} -   ( X_u   )_{1 \bar{1}, i \bar{i}}  \bigr ) \tag{4.47} \\
&\geq   C \sum_i f_i (1 + \lambda_i)   +   \sum_{j \neq 1}  \Bigl (  \frac{h}{\lambda_1 \lambda_j} - \frac{ c_2(t)   }{\lambda_1 \lambda_2 \lambda_3 \lambda_4  } \Bigr)    | (X_u)_{j \bar{1},1}    |^2.
\end{align*}

Combining inequalities (\ref{eq:4.45}) and (\ref{eq:4.47}), at the maximum point $p \in M$, if $\lambda_1$ is sufficiently large, then we have
\begin{align*}
\label{eq:4.48}
0 &\geq \mathcal{L}_t   ( U   ) \geq   A   \sum_i f_i u_{i \bar{i}} + \sum_i f_i \frac{\bigl |    (  X_u     )_{1 \bar{1},i}    \bigr |^2}{(1 + \tilde{\lambda}_1)^2}  +  C \sum_i f_i  - \sum_i f_i  \frac{(  X_u   )_{1 \bar{1}, i\bar{i}}}{1+ \tilde{\lambda}_1}    \tag{4.48} \\
&\geq    \sum_i f_i (Au_{i \bar{i}} + C ) + \sum_i  f_i \frac{  \bigl |    (  X_u     )_{1 \bar{1},i}    \bigr |^2}{(1 + \tilde{\lambda}_1)^2}        +   \sum_{j \neq 1}  \frac{ | (X_u)_{j \bar{1},1}    |^2 }{1+ \tilde{\lambda}_1   }    \Bigl (  \frac{h}{\lambda_1 \lambda_j} - \frac{ c_2(t)   }{\lambda_1 \lambda_2 \lambda_3 \lambda_4  } \Bigr)      \\
&\geq \bigl( C-A \kappa \bigr)  \sum_i f_i    + \sum_i f_i \frac{ \bigl |    (  X_u     )_{1 \bar{1},i}    \bigr |^2}{(1 + \tilde{\lambda}_1)^2}          +  \sum_{j \neq 1}  \frac{ | (X_u)_{j \bar{1},1}    |^2 }{1+ \tilde{\lambda}_1   }    \Bigl (  \frac{h}{\lambda_1 \lambda_j} - \frac{ c_2(t)   }{\lambda_1 \lambda_2 \lambda_3 \lambda_4  } \Bigr)      \\
&\geq \frac{A \kappa   \tilde{\mathcal{F}}_0 }{2}     +  f_1 \frac{\bigl |    (  X_u     )_{1 \bar{1},1}    \bigr |^2}{(1 + \tilde{\lambda}_1)^2}  +  \sum_{j \neq 1} f_j \frac{  \bigl |    (  X_u     )_{1 \bar{1},j}    \bigr |^2}{(1 + \tilde{\lambda}_1)^2}          +  \frac{ | (X_u)_{j \bar{1},1}    |^2 }{1+ \tilde{\lambda}_1   }    \Bigl (  \frac{h}{\lambda_1 \lambda_j} - \frac{ c_2(t)   }{\lambda_1 \lambda_2 \lambda_3 \lambda_4  } \Bigr). 
\end{align*}Here, provided that $A$ is sufficiently large such that $A \kappa - C > A \kappa/2$ and $\lambda_1$ is also sufficiently large. Here the inequality on the third line is by Lemma~\hyperlink{L:4.15}{4.15} and the last inequality is by Lemma~\hyperlink{L:4.16}{4.16}. \bigskip

We can also simplify the last two terms in inequality (\ref{eq:4.48}) 
\begin{align*}
\label{eq:4.49}
\kern1em
&\kern-1em  \sum_{j \neq 1} f_j \frac{  \bigl |    (  X_u     )_{1 \bar{1},j}    \bigr |^2}{(1 + \tilde{\lambda}_1)^2}          +  \frac{ | (X_u)_{j \bar{1},1}    |^2 }{1+ \tilde{\lambda}_1   }    \Bigl (  \frac{h}{\lambda_1 \lambda_j} - \frac{ c_2(t)   }{\lambda_1 \lambda_2 \lambda_3 \lambda_4  } \Bigr) \tag{4.49}   \\ 
&=  \sum_{j \neq 1} f_j \frac{ \bigl |    (  X_u     )_{1 \bar{1},j}    \bigr |^2}{(1 + \tilde{\lambda}_1)^2}   +    \frac{  | (X_u)_{1 \bar{1},j} + S_j   |^2 }{(1+ \tilde{\lambda}_1)   }      \Bigl (  \frac{h}{\lambda_1 \lambda_j} - \frac{ c_2(t)   }{\lambda_1 \lambda_2 \lambda_3 \lambda_4  } \Bigr)  \\
&\geq   \sum_{j \neq 1}  f_j \frac{ \bigl |    (  X_u     )_{1 \bar{1},j}    \bigr |^2}{(1 + \tilde{\lambda}_1)^2}  \\ 
&\kern2em +      \frac{ 1 }{(1+ \tilde{\lambda}_1)    }   \Bigl (  \frac{h}{\lambda_1 \lambda_j} - \frac{ c_2(t)   }{\lambda_1 \lambda_2 \lambda_3 \lambda_4  } \Bigr) \Bigl(  \frac{\lambda_1}{1+ \tilde{\lambda}_1} | (X_u)_{1 \bar{1},j} |^2 - \frac{\lambda_1 - B_{11}}{1+ B_{11}} | S_j   |^2 \Bigr) \\
&\geq    \sum_{j \neq 1}  \frac{ | (X_u)_{1 \bar{1},j} |^2 }{(1+ \tilde{\lambda}_1)^2    \lambda_j} \Bigl(   f_j \lambda_j +  h - \frac{c_2(t) \lambda_j}{\lambda_2 \lambda_3 \lambda_4}    \Bigr)         -  \sum_{j \neq 1}  \frac{ h | S_j   |^2 }{(1+ \tilde{\lambda}_1)    \lambda_j}         \\
&=  \sum_{j \neq 1}  \frac{\bigl |    (  X_u     )_{1 \bar{1},j}    \bigr |^2}{(1 + \tilde{\lambda}_1)^2} \frac{c_2 \sigma_1(\lambda_{;1,j}) -2c_1 \ct}{ \lambda_1 \lambda_2 \lambda_3 \lambda_4}   -   \sum_{j \neq 1}  \frac{ h | S_j   |^2 }{(1+ \tilde{\lambda}_1)    \lambda_j}         
\geq    -   \sum_{j \neq 1}  \frac{  h  | S_j   |^2 }{(1+ \tilde{\lambda}_1)     \lambda_j}         \geq -C,
\end{align*}where we denote $S_j \coloneqq  (X_u)_{j \bar{1}, 1}- (X_u)_{1 \bar{1}, j} = X_{j \bar{1}, 1}- X_{1 \bar{1}, j}$. Thus, by inequalities (\ref{eq:4.21}), (\ref{eq:4.48}), and (\ref{eq:4.49}), at the point $p$ we obtain

\begin{align*}
0 &\geq \mathcal{L}_t    ( U     ) \geq \frac{A \kappa \tilde{\mathcal{F}}_0  }{2}     +  f_1 \frac{\bigl |    (  X_u     )_{1 \bar{1},1}    \bigr |^2}{(1 + \tilde{\lambda}_1)^2}  -C \\
&\geq  \frac{A \kappa  \tilde{\mathcal{F}}_0  }{4} - A^2 |u_1|^2 \frac{c_2 \sigma_2(\lambda_{;1}) - 2c_1 \ct \sigma_1(\lambda_{;1}) + c_0 \cc   }{ \lambda_1^2 \lambda_2 \lambda_3 \lambda_4 },
\end{align*}provided that $A$ is sufficiently large. This implies,

\begin{align*}
\frac{A \kappa  \tilde{\mathcal{F}}_0  }{4}  &\leq  A^2 |u_1|^2  \frac{c_2 \sigma_2(\lambda_{;1}) - 2c_1 \ct \sigma_1(\lambda_{;1}) + c_0 \cc   }{ \lambda_1^2 \lambda_2 \lambda_3 \lambda_4 } \\
&\leq A^2 h |u_1|^2 \frac{c_2 \sigma_2(\lambda_{;1}) - 2c_1 \ct \sigma_1(\lambda_{;1}) + c_0 \cc }{ \lambda_1 \bigl( c_2 \sigma_2(\lambda) - 2c_1 \ct  \sigma_1(\lambda)+ c_0 \cc  \bigr) } \\
&\leq A^2 h |u_1|^2 \frac{1}{ \lambda_1} \leq A^2 h \cdot    \sup_M |\nabla u|^2  \cdot \frac{1}{ \lambda_1} .
\end{align*}Hence at the maximum point $p$ of $\tilde{U}$, we have
\begin{align*}
\lambda_1 \leq \frac{4 A h  }{\kappa \tilde{\mathcal{F}}_0  }   \sup_M |\nabla u|^2.
\end{align*}By plugging back to the original test function $U = -Au +G(\Lambda)$, we will obtain the $C^2$ estimate independent of $t$.
\end{proof}

\subsubsection{The $C^1$ Estimates}
\label{sec:4.2.2}
Here, same as Section~\ref{sec:4.1.2}, we use a \hyperlink{P:4.1}{blow-up argument} proved by Collins--Jacob--Yau \cite{collins20151} to obtain the $C^1$ estimate. Since everything follows verbatim, so we do not state it here.

\subsubsection{Higher Order Estimates}
Here, the proofs are similar to the proofs in Section~\ref{sec:4.1.3}, so we just state the results here without writing down the proofs. The equation is elliptic and the solution set is convex when $h \in \bigl (  0, c_2^3(t) \taa/ c_1^2(t)     \bigr )$ and 
\begin{align*}
\label{eq:4.50}
c_0(t) \cc h  + 24 c_2^2(t)  \cos^2(\theta_{h, c_1(t), c_2(t)})\cos(2\theta_{h, c_1(t), c_2(t)}) > 0. \tag{4.50}
\end{align*}
Here $\theta_{h, c_1(t), c_2(t)} \coloneqq \arccos \bigl(   {- c_1(t) \ct h^{1/2} }\big/{c_2^{3/2}(t) }  \bigr)\big/3  - {2\pi}/{3}$ and we specify the branch so that $ \arccos  ( \bullet ) \in  ( \pi, 3\pi/2  ]$. As long as $h$ satisfies these inequalities for all $t \in [0, 1]$, we can exploit the convexity of the solution sets to obtain $C^{2, \alpha}$ estimates by a blow-up argument. \bigskip




By shrinking the coordinate charts if necessary, we may assume that the manifold $M$ can be covered by finitely many coordinate charts $\bar{U}_a \subset V_a$ such that $X_u = \sqrt{-1} \partial \overline{\partial} u_a$ on $V_a$ for a smooth function $u_{\alpha}$ satisfying $\| u_a \|_{C^2(\bar{U}_a)}   \leq K$, where we use the standard Euclidean metric on $\mathbb{C}^4$ and $K$ is a uniform constant independent of $a$ and $t \in [0, 1]$. For convenience, we focus on a fixed coordinate chart $V_{a}$ and drop the subscript $a$, then the function $u$ on $V$ satisfies 
\begin{align*}
\label{eq:4.51}
 {F}_t(x, \partial \bar{\partial} u  ) = F_t \bigl (\Lambda(x) \bigr)  = h, \text{ for } x \in V, \tag{4.51}
\end{align*}where $\Lambda^j_i(x) = \omega^{j \bar{k}}(x) u_{i \bar{k}}(x)$ with eigenvalues $\lambda \bigl ( \Lambda^j_i (x) \bigr) \in \Upsilon^4_{3; h, 0, -c_2(t), 2c_1(t) \ct } \cap \Upsilon^4_{2; h, 0, -c_2(t) } \cap \Upsilon^4_{1; 1, 0}$ and $h = \sii$ is a constant satisfying the above inequalities for all $t \in [0, 1]$ by our \hyperlink{dim 4 cons}{4-dimensional four constraints}. Moreover, fix $\tilde{x} \in U$, we define the following operator which does not depend on $x \in V$, 
\begin{align*}
\tilde{F}_{t, \tilde{x}}( \partial \bar{\partial} u ) \coloneqq  {F}_t(    \omega^{j \bar{k}}(\tilde{x})  u_{i \bar{k}}   ).
\end{align*}

\bigskip


First, we prove a Hölder estimate for the second derivatives, we have the following. 

\hypertarget{L:4.17}{\begin{flemma}}
Fix $t \in [0, 1]$, let $U \subset \mathbb{C}^4$ be a connected open set, and fix $\tilde{x} \in U$. Suppose $u \colon U \subset  \mathbb{C}^4   \rightarrow \mathbb{R}$ is a $C^3$ function such that $\| \partial \bar{\partial}u \|_{L^\infty( U )}   < \infty$ and $\lambda \bigl (\omega^{j \bar{k}}(\tilde{x}) u_{i \bar{k}}(\tilde{x})  \bigr )   \in \Upsilon^4_{3; h, 0, -c_2(t), 2c_1(t) \ct } \cap \Upsilon^4_{2; h, 0, -c_2(t) } \cap \Upsilon^4_{1; 1, 0}$. If for all $x \in U$,
\begin{align*}
\tilde{F}_{t, \tilde{x}}   ( \partial \bar{\partial} u  )(x) = h,
\end{align*}then there exists a constant $\alpha \in (0, 1)$ such that for any $R > 0$ with $\overline{B_{2R}} \subset U$, the function $u$ satisfies 
\begin{align*}
\| \partial \bar{\partial} u \|_{C^\alpha(B_{R})} \leq C    \cdot R^{-\alpha}.
\end{align*}Here $C = C\bigl( \hat{\theta}, c_0, c_1, c_2, h, \| \partial \bar{\partial}u \|_{L^\infty( U )}  \bigr)$, $\lambda \bigl (\omega^{j \bar{k}}(\tilde{x}) u_{i \bar{k}}(\tilde{x})  \bigr )$ are the eigenvalues of $\omega^{j \bar{k}}(\tilde{x}) u_{i \bar{k}}(\tilde{x})$, and $h \in \bigl (  0, c_2^3(t) \taa/ c_1^2(t)     \bigr )$ is a constant satisfying inequality (\ref{eq:4.50}).
\end{flemma}

Then, with the above Lemma~\hyperlink{L:4.17}{4.17}, we have a Louiville-type result.

\hypertarget{P:4.3}{\begin{fprop}}
Fix $t \in [0, 1]$ and $\tilde{x} \in \mathbb{C}^4$. Suppose $u \colon \mathbb{C}^4   \rightarrow \mathbb{R}$ is a $C^3$ function such that $\| \partial \bar{\partial}u \|_{L^\infty( \mathbb{C}^4 )} < \infty$ and $\lambda \bigl (\omega^{j \bar{k}}(\tilde{x}) u_{i \bar{k}}(\tilde{x})  \bigr )   \in    \Upsilon^4_{3; h, 0, -c_2(t), 2c_1(t) \ct } \cap \Upsilon^4_{2; h, 0, -c_2(t) } \cap \Upsilon^4_{1; 1, 0}$. If for all $x \in \mathbb{C}^4$,
\begin{align*}
\tilde{F}_{t, \tilde{x}}   ( \partial \bar{\partial} u  )(x) = h,
\end{align*}then $u$ is a quadratic polynomial. Here $\lambda \bigl (\omega^{j \bar{k}}(\tilde{x}) u_{i \bar{k}}(\tilde{x})  \bigr )$ are the eigenvalues of $\omega^{j \bar{k}}(\tilde{x}) u_{i \bar{k}}(\tilde{x})$ and $h \in \bigl (  0, c_2^3(t) \taa/ c_1^2(t)     \bigr )$ is a constant satisfying inequality (\ref{eq:4.50}).
\end{fprop}

\hypertarget{L:4.18}{\begin{flemma}}
For $r > 0$, suppose $u \colon B_{2r} \subset \mathbb{C}^4   \rightarrow \mathbb{R}$ is a smooth function satisfying 
\begin{align*}
{F}_t (x,  \partial \bar{\partial} u ) = h,
\end{align*}where $h \in \bigl[  \sii - \epsilon,  \sii + \epsilon \bigr]$ is a constant and $\epsilon > 0$ small. Then, for every $\alpha \in (0, 1)$, we have the estimate
\begin{align*}
\|\partial \bar{\partial} u\|_{C^{\alpha} (B_{r/2})} \leq C( \alpha, \hat{\theta}, c_0, c_1, c_2, \epsilon, \|\partial \bar{\partial} u\|_{L^{\infty} (B_{2r})}).
\end{align*}
\end{flemma}

By arguing locally,  with Lemma~\hyperlink{L:4.18}{4.18} we have the following.

\hypertarget{C:4.2}{\begin{fcor}}
Suppose $X$ is a $C$-subsolution to equation (\ref{eq:4.19}) and $u \colon M  \rightarrow \mathbb{R}$ is a solution to equation (\ref{eq:4.19}), that is,
\begin{align*}
F_t \bigl(  \omega^{-1} X_u    \bigr) = h,
\end{align*}where $X_u \coloneqq X + \sqrt{-1} \partial \bar{\partial} u$ and $h  = \sii$ is a constant. Then for every $\alpha \in (0, 1)$, we have
\begin{align*}
\| \partial \bar{\partial} u \|_{C^\alpha(M)} \leq C(M, X, \omega, \alpha, \hat{\theta}, c_0, c_1, c_2, \| \partial \bar{\partial} u\|_{L^{\infty}(M)}).
\end{align*}

\end{fcor}

\section{Existence Results}
\label{sec:5}
\subsection{When \texorpdfstring{$n =3$}{}}
\label{sec:5.1}
In this subsection, we always assume that $\hat{\theta} \in \bigl( \pi/2, 5 \pi/6    \bigr)$ and there exists a $C$-subsolution. By changing representative, we say $X$ is this $C$-subsolution. We can find a pair $\bigl(c_1(t), c_0(t)\bigr)$ such that the \hyperlink{dim 3 cons}{3-dimensional four constraints} will all be satisfied. Moreover, we prove that when the complex dimension equals three, if there exists a $C$-subsolution, then the dHYM equation (\ref{eq:3.1}) is solvable. We consider the following continuity path,
\begin{align*}
\label{eq:5.1}
\css X^3 - 3 c_1(t)   \omega^2 \wedge X -2 c_0(t) \ta   \omega^3 = 0, \tag{5.1}
\end{align*}where $t \in [0, 1]$ and $c_0(t)$ and $c_1(t)$ are smooth functions in $t$ which satisfy all the following \hypertarget{dim 3 cons 5.1}{3-dimensional version of the four constraints}:
\begin{enumerate}[leftmargin=4.5cm]
	\setlength\itemsep{+0.4em}
\item[Topological constraint:] $\cos^2(\hat{\theta}) \Omega_0 - 3c_1(t)    \Omega_2 - 2 c_0(t) \ta   \Omega_3  =0$.
\item[Boundary constraints:] $c_1(1) = c_0(1) = 1;\quad c_1(0)>0; \quad c_0(0) = 0$.
\item[Positivstellensatz constraint:] $c_1(t)^{ {3}/{2}} > c_0(t) \sin(\hat{\theta})$.
\item[$\Upsilon$-cone constraint:] $c_1'(t) > 0$.
\end{enumerate}\smallskip

Here we denote $\Omega_i \coloneqq \int_M \omega^i \wedge X^{3-i}$. \bigskip

\hypertarget{L:5.1}{\begin{flemma}} If $ ( \Omega_2, \Omega_3   ) \in \Omega^{3, \hat{\theta}}$, then the following pair will satisfy all the \hyperlink{dim 3 cons 5.1}{3-dimensional four constraints}:
\begin{align*}
c_1(t) \coloneqq \frac{ \css \Omega_0 - 2t \ta   \Omega_3}{3   \Omega_2}; \quad c_0(t) \coloneqq t.
\end{align*}Here $\Omega_i = \int_M \omega^i \wedge X^{3-i}$ and 
\begin{align*}
\Omega^{3, \hat{\theta}} \coloneqq \Bigl \{     \Omega_3 < \inf_{t \in [0, 1) } - \frac{3(1- t^{2/3} \sin^{2/3}(\hat{\theta}) )\ct}{2(1-t)} \Omega_2  \Bigr \}. 
\end{align*}
\end{flemma}

\begin{proof}

%
%
%
%

First, the topological constraint is automatically satisfied. Then, we can check that they satisfy the boundary constraints
\begin{align*}
1 =  c_1(1) = c_0(1); \quad c_1(0) = 0.
\end{align*}Third, for the $\Upsilon$-cone constraint, we have
\begin{align*}
c_1'(t) = \frac{-2 \ta   \Omega_3}{3   \Omega_2} > 0 
\end{align*}So $c_0(t)$ and $c_1(t)$ are both increasing with $1 \geq c_1(t) \geq \css \Omega_0/ ( 3   \Omega_2) > 0$. \bigskip


Last, for the Positivstellensatz constraint, when $t =0$ or $t =1$, the Positivstellensatz constraint holds. We rewrite $c_1(t)$ as
\begin{align*}
c_1(t) &= \frac{\css \Omega_0 - 2t \ta   \Omega_3}{3   \Omega_2} = 1+ \frac{2}{3}(1-t) \ta   \frac{\Omega_3}{\Omega_2}.
\end{align*}

For $t \in (0, 1)$, if $ ( \Omega_2, \Omega_3   ) \in \Omega^{3, \hat{\theta}}$, then
\begin{align*}
c_1(t) &= 1+ \frac{2}{3}(1-t) \ta   \frac{\Omega_3}{\Omega_2} > 1- \frac{2}{3}(1-t) \ta \frac{3(1- t^{2/3} \sin^{2/3}(\hat{\theta}) )\ct}{2(1-t)} \\
&=   t^{2/3} \sin^{2/3}(\hat{\theta}).
\end{align*}Thus, $c_1(t)^{ {3}/{2}} > c_0(t) \sin(\hat{\theta})$, which finishes the proof.
\end{proof}

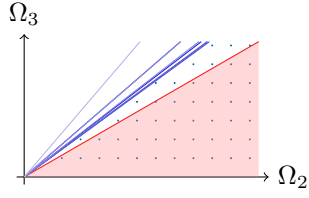
\begin{figure}
\centering
\begin{tikzpicture}[scale=0.5]
	\def\theta{2*pi/3}
	\xdef\acc{0}
  	\draw[->] (-0.2,0) -- (6.5,0) node[right] {$\Omega_2$};
  	\draw[->] (0,-0.2) -- (0,3.8) node[above] {$\Omega_3$};
  	\draw[color=red,domain=0:3.6]    plot ({-tan(\theta r)*\x},\x)             ;
	\foreach \t in {0, ..., 7}
	{
		\pgfmathsetmacro\r{\t/10 }
		\pgfmathsetmacro\g{\t/10 }	
		\pgfmathsetmacro\b{0.8+\t/64}
  		\draw[color={rgb, 1:red, \r; green, \g; blue, \b},domain=0:3.6,samples=100]    plot ({-2*tan(\theta r)*(1 - \t/8)*\x/(  3*(1 - (sin(\theta r)*\t/8)^(2/3))    ) },\x)             ;
	}
	\foreach \x in {0.5, 1, 1.5, 2, 2.5, 3, 3.5, 4, 4.5, 5, 5.5, 6 }
	{
			\foreach \y in {0.5, 1, 1.5, 2, 2.5, 3, 3.5}
			{
				\xdef\acc{0}
				\foreach \t in {0, ..., 7}
				{
					\xdef\C{0}
					\pgfmathsetmacro\C{  -3*\x*(1 - (sin(\theta r)*\t/8)^(2/3))/(2*tan(\theta r)*(1 - \t/8))  - \y }
					\ifdim \C pt > 0 pt 
						\pgfmathsetmacro\acc{\acc+1} 
						\xdef\acc{\acc}
					\fi
					\ifdim \acc pt = 8 pt 
						\filldraw[color=UCIB] (\x, \y) circle (0.3pt);
					\fi
				}
			}
	}
	
\begin{scope}
    \fill[color=red!50,opacity=0.3,thick,domain=0:3.6,samples=101]
    plot ({-tan(\theta r)*\x},\x) -- ({-tan(\theta r)*3.6}, 0)  ;
  \end{scope}
\end{tikzpicture}
\caption{$\Omega^{3, \hat{\theta}}$ and $C$-subsolution constraints}
\label{fig:5.1}
\end{figure}


Figure~\ref{fig:5.1} shows that $\Omega^{3, \hat{\theta}}$, which is the blue dotted region, contains the $C$-subsolution constraints, which is the pink shaded region. So as long as a $C$-subsolution exists, the numerical values $(\Omega_2, \Omega_3)$ will be in the pink shaded region, which will also be in $\Omega^{3, \hat{\theta}}$, thus the complex three dimensional dHYM equation is solvable. The following Corollary proves this observation.

\hypertarget{C:5.1}{\begin{fcor}}
When $\hat{\theta} \in  (   {\pi}/{2},   {5\pi}/{6}   )$. If there exists a $C$-subsolution to equation (\ref{eq:3.1}), then the three dimensional dHYM equation (\ref{eq:3.1}) is solvable.
\end{fcor}

\begin{proof}

If a $C$-subsolution exists, by Lemma~\hyperlink{L:3.1}{3.1}, at every point, the eigenvalues of this $C$-subsolution will be in $\Upsilon^3_{2; 1, 0, - \s }  \cap \Upsilon^3_{1; 1, 0}$. This implies that
\begin{align*}
X^2 > \s \omega^2   \Longrightarrow X^3 > \s \omega^2 \wedge X   \Longrightarrow \Omega_0 > \s \Omega_2.
\end{align*}

%

By rewriting the topological constraint, we get
\begin{align*}
\label{eq:5.2}
-\ta \frac{\Omega_3}{\Omega_2} =   \frac{3}{2} -  \frac{  \Omega_0}{2 \s   \Omega_2} < 1. \tag{5.2}
\end{align*}

On the other hand, consider the following quantity 
\begin{align*}
\label{eq:5.3}
\frac{3(1- t^{2/3} \sin^{2/3}(\hat{\theta}) ) }{2(1-t)}. \tag{5.3}
\end{align*} For $t \in (0, 1)$, quantity (\ref{eq:5.3}) has a lower bound 
\begin{align*}
\label{eq:5.4}
\frac{3(1- t^{2/3} \sin^{2/3}(\hat{\theta}) ) }{2(1-t)} > \frac{3(1- t^{2/3}   ) }{2(1-t)} \geq 1.  \tag{5.4}
\end{align*}The last inequality is due to the fact that the function $(1- t^{2/3})/(1-t)$ is decreasing when $t \in (0, 1)$ and by L'Hôpital's rule. Combining inequalities (\ref{eq:5.2}) and (\ref{eq:5.4}), we see that if there exists a $C$-subsolution, then we always have $(\Omega_2, \Omega_3) \in \Omega^{3, \hat{\theta}}$.
\end{proof}

\subsection{When \texorpdfstring{$n=4$}{}}
\label{sec:5.2}
In this subsection, we always assume that $\hat{\theta} \in \bigl( \pi, 5 \pi/4    \bigr)$ and there exists a $C$-subsolution. By changing representative, we say $X$ is this $C$-subsolution. We can find a triple $\bigl(c_2(t), c_1(t), c_0(t)\bigr)$ such that the \hyperlink{dim 4 cons}{4-dimensional four constraints} will all be satisfied. Moreover, we prove that when the complex dimension equals four, if there exists a $C$-subsolution, then the dHYM equation (\ref{eq:3.7}) is solvable. We consider the following continuity path,
\begin{align*}
\label{eq:5.5}
\sii X^4 - 6 c_2(t)   \omega^2 \wedge X^2 + 8 c_1(t) \ct   \omega^3 \wedge X - c_0(t) \cc   \omega^4 = 0. \tag{5.5}
\end{align*}
where $t \in [0, 1]$ and this triple $\bigl(c_2(t), c_1(t), c_0(t)\bigr)$ are smooth functions in $t$ satisfying the following \hypertarget{dim 4 cons 5.2}{4-dimensional four constraints}:
\begin{enumerate}[leftmargin=4.5cm]
	\setlength\itemsep{+0.4em}
\item[Topological constraint:] $\sii \Omega_0 - 6c_2(t) \Omega_2 + 8c_1(t) \ct \Omega_3 - c_0(t) \cc   \Omega_4 =0$.
\item[Boundary constraints:] $c_2(1) = c_1(1) = c_0(1) = 1;\quad c_2(0)>0; \quad c_1(0) = 0$.
\item[Positivstellensatz constraint:] $\cc c_0(t) >  -24   {c_2^2(t) \cs \cos^2(\theta_{c_1, c_2})\cos(2\theta_{c_1, c_2})} $.
\item[$\Upsilon$-cone constraints:] $\frac{d}{dt} \bigl( c_2^{3/2}(t) \bigr) \geq -\cos(\hat{\theta}) c_1'(t); \quad c_2'(t) > 0; \quad c_1'(t) > 0$.
\end{enumerate}Here we denote $\Omega_i \coloneqq \int_M \omega^i \wedge X^{4-i}$ and we define $\theta_{c_1, c_2} =  \arccos \bigl(   { c_1(t) \coo}/{c_2^{3/2}(t)}  \bigr)\big/ 3 -  {2\pi}/{3}$, where we specify the branch so that $\arccos \bigl(   {c_1(t) \cos(\hat{\theta})}/{c_2^{3/2}(t)}  \bigr) \in   \bigl( \pi, 3\pi/2 \bigr]$.\bigskip

\hypertarget{T:5.1}{\begin{fthm}} If $ ( \Omega_2, \Omega_3, \Omega_4   ) \in \Omega^{4,   \hat{\theta}}_{\ell}$, then the following triple will satisfy all the \hyperlink{dim 4 cons 5.2}{4-dimensional four constraints}:
\begin{align*}
c_{2, \ell} (t) \coloneqq \bigl[ 1 + (1-t) \ell   \cos (\hat{\theta} )   \bigr ]^{2/3}; \  c_1(t) \coloneqq t; \  c_{0, \ell}(t) \coloneqq \frac{ \sii  \Omega_0 -  6c_{2, \ell}(t)      \Omega_2 + 8c_1(t) \ct     \Omega_3}{\cc    \Omega_4}.
\end{align*}Here $\ell \in \bigl [ 1, -\sco \bigr)$ is a constant, $\Omega_i \coloneqq \int_M \omega^i \wedge X^{4-i}$, and 
\begin{align*}
\Omega^{4, \hat{\theta}}_{\ell} \coloneqq \Bigl \{     \Omega_3 <   \inf_{t \in [0, 1) } \frac{3(1- c_{2, \ell} (t))\ta}{4(1-t)} \Omega_2 + \frac{  \tilde{c}_{0, \ell} (t) \ta}{8(1-t)} \Omega_4  \Bigr \}, 
\end{align*}where we define $\tilde{c}_{0, \ell} (t) \coloneqq 3 \cs -4 +  24   {c_{2, \ell}^2 \cs \cos^2(\theta_{c_1, c_{2, \ell}})\cos(2\theta_{c_1, c_{2, \ell}})}$. Moreover, $\tilde{c}_{0, \ell}(t)$ is a positive function. 
\end{fthm}




\begin{proof}

%
%
%
%

First, the topological constraint is automatically satisfied. Then, we can check that they satisfy the boundary constraints
\begin{align*}
1 =  c_{2, \ell} (1) = c_1(1) = c_{0, \ell}(1); \quad c_1(0) = 0.
\end{align*}Third, for the $\Upsilon$-cone constraints, we check
\begin{align*}
c_{2, \ell}'(t) = -\frac{2}{3} \ell  \coo  \bigl[ 1 + (1-t) \ell \coo \bigr ]^{-1/3} > 0; \quad c_1'(t) = 1 > 0. 
\end{align*}So $c_1(t)$ and $c_{2, \ell}(t)$ are both increasing with $1 \geq c_{2, \ell}(t) \geq  \bigl( 1 + \ell \cos (\hat{\theta} )   \bigr )^{2/3} > 0$ and $1 \geq c_1(t) \geq 0$. We also check that
\begin{align*}
\frac{d}{dt} c_{2, \ell}^{3/2} (t) = \frac{d}{dt}  \bigl[ 1 + (1-t)  \ell  \coo  \bigr ]  = - \ell \coo \geq -\coo c_1'(t).
\end{align*}



Last, for the Positivstellensatz constraint, we are checking whether the following quantity
\begin{align*}
\label{eq:5.7}
&\kern-1em \sii  \Omega_0 -  6c_{2, \ell}(t)      \Omega_2 + 8c_1(t) \ct     \Omega_3 + 24   {c_{2, \ell}^2(t) \cs \cos^2(\theta_{c_1, c_{2, \ell}})\cos(2\theta_{c_1, c_{2, \ell}})} \Omega_4    \tag{5.7} \\
&=  6 (1 -  c_{2, \ell}(t))      \Omega_2 + 8 (c_1(t)-1) \ct     \Omega_3 + \tilde{c}_{0, \ell}(t)  \Omega_4
\end{align*}is positive. If $t = 1$, then quantity (\ref{eq:5.7}) will be
\begin{align*}
\label{eq:5.8}
&\kern-1em 6 (1 -  c_{2, \ell}(1))      \Omega_2 + 8 (c_1(1)-1) \ct     \Omega_3 + \tilde{c}_{0, \ell}(1)  \Omega_4  \tag{5.8} \\ 
&=  \cs \bigl( 3   -4 \sii +  24   {    \cos^2(\theta_{1, 1})\cos(2\theta_{1, 1})}  \bigr) \Omega_4,
\end{align*}where $\theta_{1, 1} =   \hat{\theta} / 3 -  {2\pi}/{3} \in (- \pi/3, -\pi/4)$. By the observation that $\hat{\theta} = 3 \theta_{1, 1} + 2\pi$, we have
\begin{align*}
3   -4 \sii+  24   {    \cos^2(\theta_{1, 1})\cos(2\theta_{1, 1})} &= 3  -4 \sin^2(3 \theta_{1, 1}) + 24   \cos^2(\theta_{1, 1})\cos(2\theta_{1, 1}).
\end{align*}The above quantity will be positive if $\theta_{1, 1} \in (- \pi/3, -\pi/4)$ by some standard calculus techniques. In fact, it will be strictly increasing in this range and hence have an infimum if one set $\theta_{1,1} = -\pi/3$. Thus quantity (\ref{eq:5.8}) will be positive. \bigskip

Now, for $t \in [0, 1)$ and $ ( \Omega_2, \Omega_3, \Omega_4   ) \in \Omega^{4, \hat{\theta}}_{\ell}$, we have
\begin{align*}
8(1-c_1(t)) \ct \Omega_3 &= 8(1-t) \ct \Omega_3 <    \inf_{t \in [0, 1) }  {6(1- c_{2, \ell}(t))\ta}  \Omega_2 +  { \tilde{c}_{0, \ell}(t) \ta}  \Omega_4 \\
&\leq {6(1- c_{2, \ell}(t))\ta}  \Omega_2 +  { \tilde{c}_{0, \ell}(t) \ta}  \Omega_4,
\end{align*}which implies quantity (\ref{eq:5.7}) will be positive. Thus we prove the Positivstellensatz constraint holds.\bigskip

In addition, to prove $\tilde{c}_{0, \ell}(t)$ is a positive function on $[0, 1]$, first, we calculate the first derivative of $\theta_{c_1, c_{2, \ell}}$ with respect to $t$.
\begin{align*}
\label{eq:5.9}
\frac{d}{dt} \theta_{c_1, c_{2, \ell}} &=   \frac{1}{3} \coo \bigl(1 + \ell \coo \bigr) \bigl(   1+ (1-t) \ell \coo    \bigr)^{-2}  \bigl(   1 -   {c_1^2 \css}/{c_{2, \ell}^3}       \bigr)^{-1/2} < 0. \tag{5.9}
\end{align*}Note that due to our choice of the branch of $\arccos$, thus the derivative of $\arccos$ will differ by a sign compared to the usual one, the derivative of $\arccos$ will be positive instead of negative. When $t = 0$, $\theta_{c_1, c_{2, \ell}} ( 0) = - \pi/6$, and when $t = 1$, $\theta_{c_1, c_{2, \ell}} ( 1) = - \theta_{1,1} \in (- \pi/3, -\pi/4)$.

Then, the derivative of $\tilde{c}_{0, \ell}(t)$ with respect to $t$ will be
\begin{align*}
\label{eq:5.10}
\tilde{c}_{0, \ell}'(t)  \tag{5.10}
&= 48 c_{2, \ell} \cs \cos(\theta_{c_1, c_{2, \ell}}) \bigl(   c_{2, \ell}'  \cos(\theta_{c_1, c_{2, \ell}})\cos(2\theta_{c_1, c_{2, \ell}})   -c_{2, \ell} \theta'_{c_1, c_{2, \ell}} \sin( 3\theta_{c_1, c_{2, \ell}})     \bigr) \\
&= 16 c_{2, \ell}^{1/2} \cs \cos(\theta_{c_1, c_{2, \ell}})  \Bigl( \bigl(   c_{2, \ell}^{3/2} \cos(3 \theta_{c_1, c_{2, \ell}})    \bigr)'   + (c_{2, \ell}^{3/2})' \cos(\theta_{c_1, c_{2, \ell}})    \Bigr) \\
&= 16 c_{2, \ell}^{1/2} \cos(\hat{\theta}) \cs \cos(\theta_{c_1, c_{2, \ell}}) \bigl  (  1 - \ell \cos(\theta_{c_1, c_{2, \ell}})   \bigr  ). 
\end{align*} 

By inequality (\ref{eq:5.9}), the value of $\cos(\theta_{c_1, c_{2, \ell}})$ will be strictly decreasing. By this fact and $\ell < -\sco \leq \sqrt{2}$, the last term in equation (\ref{eq:5.10}), $1 - \ell \cos(\theta_{c_1, c_{2, \ell}})$, will be positive when $t$ approaches $1$. Hence $\tilde{c}'_{0, \ell}(t)$ will be negative when $t$ approaches $1$. This implies that $\tilde{c}_{0, \ell}(t)$ will attain its minimum at one of its boundary points.
\begin{align*}
\tilde{c}_{0, \ell}(t) &\geq \min \{ \tilde{c}_{0, \ell}(0), \tilde{c}_{0, \ell}(1) \} \\
&= 3 \cs -4 + \cs \min \bigl \{          9 c_{2, \ell}^2(0),    24    {    \cos^2(\theta_{1, 1})\cos(2\theta_{1, 1})}    \bigr    \} > 0.
\end{align*}

Here, we use the fact that quantity (\ref{eq:5.8}) will be positive, which implies $\tilde{c}_0(t)$ is a positive function. This finishes the proof.
\end{proof}

%


\hypertarget{T:5.2}{\begin{fthm}}
There exists an $\ell_N$ sufficiently close to $-\sco$, such that if there exists a $C$-subsolution $X$ to equation (\ref{eq:3.7}), then
\begin{align*}
(\Omega_2, \Omega_3, \Omega_4 ) \in \Omega^{4, \hat{\theta}}_{{\ell}_N}
\end{align*}
and the four dimensional dHYM equation (\ref{eq:3.7}) is solvable. Here $\Omega_i \coloneqq \int_M \omega^i \wedge X^{4-i}$.  
\end{fthm}

\begin{proof}

If a $C$-subsolution exists, by Lemma~\hyperlink{L:3.4}{3.4}, we have
\begin{align*}
\sii X^3 - 3 \omega^2 \wedge X + 2 \ct \omega^3 > 0 \Longrightarrow \sii \Omega_0 - 3 \Omega_2 + 2 \ct \Omega_3 > 0.
\end{align*}
By the topological constraint that $\sii \Omega_0 = 6 \Omega_2 - 8 \ct \Omega_3 + \cc \Omega_4$, we obtain the following inequality 
\begin{align*}
\label{eq:5.11}
3 \Omega_2  -6 \ct \Omega_3 + \cc \Omega_4 > 0. \tag{5.11}
\end{align*}So if a $C$-subsolution exists, then we automatically have the following inequality
\begin{align*}
4 \Omega_2 +\frac{4}{3} \cc \Omega_4 > 8 \ct \Omega_3.
\end{align*}

By comparing the above inequality with $\Omega^{4, \hat{\theta}}_{\ell}$, if for any $t \in [0, 1)$, the following quantity 
\begin{align*}
\label{eq:5.12}
&\kern-1em \Bigl[   6(1-c_{2, \ell}(t)) \Omega_2 + \tilde{c}_{0, \ell}(t) \Omega_4    \Bigr] -  \Bigl[   4 (1-t) \Omega_2 +   \frac{4}{3}(1-t) \cc  \Omega_4    \Bigr] \tag{5.12} \\
&= \bigl(2 +4t - 6c_{2, \ell}(t) \bigr)  \Omega_2 + \bigl(  \tilde{c}_{0, \ell}(t) -4(1-t) \cc/3 \bigr) \Omega_4
\end{align*}is positive for some $\ell$, then $( \Omega_2, \Omega_3, \Omega_4   ) \in \Omega^{4, \hat{\theta}}_{\ell}$. This will finish the proof if $\ell$ is sufficiently close to $-\sco$. \bigskip

By taking the derivative of the coefficient of $\Omega_2$ of quantity (\ref{eq:5.12}) with respect to $t$, we get
\begin{align*}
\label{eq:5.13}
\frac{d}{dt} \bigl(2 +4t - 6c_{2, \ell}(t) \bigr) &=  4 - 6 c'_{2, \ell}(t) = 4 + 4 \ell \coo   \bigl[ 1 + (1-t) \ell \coo \bigr ]^{-1/3}. \tag{5.13}
\end{align*}Above quantity (\ref{eq:5.13}) will be negative when $t < 1 + \ell^2 \css +  \ell^{-1} \sec(\hat{\theta})$. If $\ell$ is sufficiently close to $-\sco$, then we have 
$1/100 < 1 + \ell^2 \css +  \ell^{-1} \sec(\hat{\theta})$. Thus the coefficient of $\Omega_2$ will be decreasing on $[0, 1/100]$. We get
\begin{align*}
\label{eq:5.14}
2 +4t - 6c_{2, \ell}(t) \geq 2.04 - 6c_{2, \ell}( {1}/{100}) = 2.04 - 6  \bigl[ 1 +  {99} \ell \coo/{100} \bigr ]^{2/3} \geq 1.5, \tag{5.14}
\end{align*}for $t \in [0, 1/100]$, provided that $\ell$ is sufficiently close to $-\sco$.\bigskip

On the other hand, by taking the derivative of the coefficient of $\Omega_4$ of quantity (\ref{eq:5.12}) with respect to $t$, when $t \in [0, 1/100]$, we obtain
\begin{align*}
\label{eq:5.15}
&\kern-1em \frac{d}{dt} \bigl(  \tilde{c}_{0, \ell}(t) -4(1-t) \cc/3 \bigr) \tag{5.15} \\
&= 4 \cc/3 + 16 c_{2, \ell}^{1/2} \cos(\hat{\theta}) \cs \cos(\theta_{c_1, c_{2, \ell}}) \bigl  (  1 - \ell \cos(\theta_{c_1, c_{2, \ell}})   \bigr  ) \\
&= 4 \cs -  {16}/{3} +  16 c_{2, \ell}^{1/2} \cs \cos(\theta_{c_1, c_{2, \ell}}) \bigl(  \cos(\hat{\theta})   -  \ell \cos(\hat{\theta})   \cos(\theta_{c_1, c_{2, \ell}}) \bigr) \\
&\geq 4 \cs -  {16}/{3}  \\
&\kern2em+  16 c_{2, \ell}^{1/2}( {1}/{100}) \cs \cos(\theta_{c_1, c_{2, \ell}})  \min \bigl \{ 0,   \cos(\hat{\theta})   -  \ell \cos(\hat{\theta})   \cos(\theta_{c_1, c_{2, \ell}}) \bigr \} \\
&\geq 4 \cs -  {16}/{3} +  16 \cdot 0.25 \cdot  \cs \cos(\theta_{c_1, c_{2, \ell}})  \min \bigl \{ 0,   \cos(\hat{\theta})   +0.9   \cos(\theta_{c_1, c_{2, \ell}}) \bigr \} \\
&\geq 4 \cs -  {16}/{3} +  4  \cs \cos(\theta_{c_1, c_{2, \ell}})  \min \bigl \{ 0,   \cos(\hat{\theta})   +0.9   \cos(\theta_{c_1, c_{2, \ell}}) \bigr \}.
\end{align*}For the inequality on the second to last line, for $\ell$ sufficiently close to $-\sco$, we have the bound that $c_{2, \ell}^{1/2}( {1}/{100}) = \bigl[ 1 + 99 \ell   \cos (\hat{\theta} ) /100  \bigr ]^{1/3} \leq 0.25$. Then, we try to estimate the following quantity
\begin{align*}
\label{eq:5.16}
&\kern-1em \cos(\theta_{c_1, c_{2, \ell}})  \bigl(   \cos(\hat{\theta})   +0.9   \cos(\theta_{c_1, c_{2, \ell}}) \bigr ) \tag{5.16} \\
&= 0.9 \bigl(  \cos(\theta_{c_1, c_{2, \ell}})  +  {\coo}/{1.8}    \bigr)^2 - \css/3.6 \geq - \css/3.6 \geq -5/18.
\end{align*}
By combining inequalities (\ref{eq:5.15}) and (\ref{eq:5.16}), we obtain 
\begin{align*}
\frac{d}{dt} \bigl(  \tilde{c}_{0, \ell}(t) -4(1-t) \cc/3 \bigr) &\geq 26 \cs /9 -  {16}/{3}  \geq 4/9 > 0.
\end{align*}Here $\hat{\theta} \in (\pi, 5\pi/4)$, thus we have the bound $\cs \geq 2$. In conclusion, the coefficient of $\Omega_4$ of quantity (\ref{eq:5.12}) will be increasing on $[0, 1/100]$. By combining (\ref{eq:5.12}), (\ref{eq:5.14}), and the inequality $\Omega_2 > \cs \Omega_4$, for $t \in [0, 1/100]$, we have
\begin{align*}
\label{eq:5.17}
&\kern-1em \bigl(2 +4t - 6c_{2, \ell}(t) \bigr)  \Omega_2 + \bigl(  \tilde{c}_{0, \ell}(t) -4(1-t) \cc/3 \bigr) \Omega_4 \tag{5.17} \\
&> 1.5 \cs \Omega_4 + \bigl(  \tilde{c}_{0, \ell}(0) -4(1-0) \cc/3 \bigr) \Omega_4  \\
&=   \bigl( 0.5 \cs + 9 c_{2, \ell}^2(0) \cs  +4/3     \bigr) \Omega_4 > 0.
\end{align*}

Last, we just need to check that whether the quantity (\ref{eq:5.12}) is positive on $[1/100, 1]$. By letting $\ell_n = -\sco - 1/n$ for $n$ large, we define the following sequence of functions on $[1/100, 1]$: 
\begin{align*}
d_n (t) \coloneqq 6(1-c_{2, \ell_n}(t)) \Omega_2 + \tilde{c}_{0, \ell_n}(t) \Omega_4
\end{align*}and the following function on $[1/100, 1]$: 
\begin{align*}
d_{\infty} (t) \coloneqq 6(1- t^{2/3}) \Omega_2 +  \bigl (  3 \cs -4 +  24   {t^{4/3} \cs \cos^2(\theta_{1, 1})\cos(2\theta_{1, 1})}  \bigr) \Omega_4. 
\end{align*}

One can check that the sequence $\{d_n\}$ converges pointwise to $d_{\infty}$. Moreover, by checking the $C^1$ norm is uniformly bounded on $[1/100, 1]$, the sequence is an equicontinuous sequence that converges uniformly to $d_{\infty}$ on $[1/100, 1]$. Thus instead of quantity (\ref{eq:5.12}), we consider the following quantity 
\begin{align*}
&\kern-1em d_{\infty} (t) -  \Bigl[   4 (1-t) \Omega_2 +    \frac{4}{3}(1-t) \cc  \Omega_4    \Bigr]   \\
&= (2 + 4t - 6 t^{2/3}) \Omega_2 + \Bigl[       {(4t-1)} \cc/{3} + 24 {t^{4/3} \cs \cos^2(\theta_{1, 1})\cos(2\theta_{1, 1})}     \Bigr] \Omega_4.
\end{align*}

By the $\Upsilon$-cone condition, we have $\sii X^2 > \omega^2$, which implies $\Omega_2 > \cs \Omega_4$. If we prove the following quantity 
\begin{align*}
\label{eq:5.18} 
\kern0.55em
&\kern-1em \Bigl[    (2 + 4t - 6 t^{2/3}) \cs +    {(4t-1)} \cc/{3} + 24 {t^{4/3} \cs \cos^2(\theta_{1, 1})\cos(2\theta_{1, 1})}     \Bigr] \Omega_4 \tag{5.18} \\
&= \cs  \Bigl[    2 + 4t - 6 t^{2/3} +    {(4t-1)} (3- 4 \sii)/{3} + 24 {t^{4/3}   \cos^2(\theta_{1, 1})\cos(2\theta_{1, 1})}     \Bigr] \Omega_4
\end{align*}is positive on $[1/100, 1]$, then $d_{\infty} (t) -   [   4 (1-t) \Omega_2 +    \frac{4}{3}(1-t) \cc  \Omega_4     ]$ will always be positive on $[1/100, 1]$.
Since the sequence $\{d_n\}$ converges uniformly to $d_{\infty}$, there exists $\ell_N$ sufficiently close to $-\sco$ such that 
\begin{align*}
d_{N} (t) -  \bigl[   4 (1-t) \Omega_2 +    \frac{4}{3}(1-t) \cc  \Omega_4    \bigr] > 0
\end{align*}for $t \in [1/100, 1]$. In addition, by inequality (\ref{eq:5.17}), the above quantity will also be positive on $[0, 1/100]$, provided that $N$ is sufficiently large. By choosing $N$ sufficiently large if necessary, then $d_{\ell_N} (t) -   [   4 (1-t) \Omega_2 +    \frac{4}{3}(1-t) \cc  \Omega_4     ] $ will be positive on $[0 ,1]$, which will finish the proof.\bigskip

Now, we consider the coefficient of quantity (\ref{eq:5.18}), we have
\begin{align*}
\label{eq:5.19}
&\kern-1em  2 + 4t - 6 t^{2/3} +    {(4t-1)} (3- 4 \sii) /{3}+ 24 {t^{4/3}   \cos^2(\theta_{1, 1})\cos(2\theta_{1, 1})}   \tag{5.19} \\
&= 2 + 4t - 6 t^{2/3} +    {(4t-1)} (3- 4 \sin^2(3 \theta_{1, 1})) /{3}+ 24 {t^{4/3}   \cos^2(\theta_{1, 1})\cos(2\theta_{1, 1})},
\end{align*}where $\theta_{1, 1} \in (-\pi/3, \pi/4)$. By taking the partial derivative with respect to $\theta_{1, 1}$, the quantity (\ref{eq:5.19}) becomes
\begin{align*}
&\kern-1em -8(4t-1)\sin(3 \theta_{1, 1}) \cos(3 \theta_{1, 1}) -48 t^{4/3} \cos( \theta_{1, 1}) \sin(3 \theta_{1, 1}) \\
&= 8 \sin(3 \theta_{1, 1}) \bigl( \cos(3 \theta_{1, 1})  -4t \cos(3 \theta_{1, 1}) - 6t^{4/3} \cos(\theta_{1, 1})        \bigr) \\
&\geq  8 \sin(3 \theta_{1, 1}) \bigl( \cos(3 \theta_{1, 1})  -4 \cos(3 \theta_{1, 1}) - 6  \cos(\theta_{1, 1})        \bigr)  \\
&= -24 \cos(\theta_{1, 1}) \sin(3 \theta_{1, 1}) (4 \cos^2(\theta_{1, 1}) - 1) > 0.
\end{align*}Note that the inequality on the second to last line is due to standard calculus methods. The function $\cos(3 \theta_{1, 1})  -4t \cos(3 \theta_{1, 1}) - 6t^{4/3} \cos(\theta_{1, 1})   $ will be increasing on $[0, 1]$. The last inequality is by the range of $\theta_{1, 1}$: since $\theta_{1, 1} \in (-\pi/3, \pi/4)$, we have $\cos(\theta_{1, 1}) > 0$, $\sin(3 \theta_{1, 1}) < 0$, and $4 \cos^2(\theta_{1, 1}) > 1$. With this observation, quantity (\ref{eq:5.19}) has a pointwise lower bound when setting $\theta_{1, 1} = -\pi/3$. That is, for $t \in [0, 1]$, we have
\begin{align*}
&\kern-1em 2 + 4t - 6 t^{2/3} +    {(4t-1)} (3- 4 \sin^2(3 \theta_{1, 1})) /{3}+ 24 {t^{4/3}   \cos^2(\theta_{1, 1})\cos(2\theta_{1, 1})} \\
&> 2 + 4t - 6 t^{2/3} +    {(4t-1)} (3- 4 \sin^2(-\pi ) )/{3}+ 24 {t^{4/3}   \cos^2(-\pi/3)\cos(-2\pi/3)} \\
&= 1 + 8t - 6 t^{2/3}  -3 t^{4/3}   \geq 1 + 8\cdot 1 - 6 \cdot1^{2/3}  - 3 \cdot 1^{4/3} = 0,
\end{align*}where the last inequality is by the fact that $1 + 8t - 6 t^{2/3}  -3 t^{4/3}$ is a decreasing function on $[0, 1]$. In conclusion, we showed that $d_{\infty} (t) -  \bigl[   4 (1-t) \Omega_2 +    \frac{4}{3}(1-t) \cc  \Omega_4    \bigr]$ will be positive on $[1/100, 1]$. So for $N$ sufficiently large, we have that $d_{N} (t) -  \bigl[   4 (1-t) \Omega_2 +    \frac{4}{3}(1-t) \cc  \Omega_4    \bigr]$ is positive on $[1/100, 1]$ as well. As we discussed, for $N$ sufficiently large, $d_{N} (t) -  \bigl[   4 (1-t) \Omega_2 +    \frac{4}{3}(1-t) \cc  \Omega_4    \bigr]$ is positive on $[0, 1]$. Hence, $( \Omega_2, \Omega_3, \Omega_4   ) \in \Omega^{4, \hat{\theta}}_{\ell_N}$. \bigskip

In particular, by Theorem~\hyperlink{T:5.1}{5.1}, the following triple will satisfy all the 
\hyperlink{dim 4 cons 5.2}{4-dimensional four constraints}:
\begin{align*}
c_{2, \ell_N} (t) = \bigl[ 1 + (1-t) \ell_N   \cos (\hat{\theta} )   \bigr ]^{2/3};   c_1(t) = t;   c_{0, \ell_N}(t) = \frac{ \sii  \Omega_0 -  6c_{2, \ell_N}(t)      \Omega_2 + 8c_1(t) \ct     \Omega_3}{\cc    \Omega_4}.
\end{align*}When $t = 0$, we have
\begin{align*}
c_{2, \ell_N} (0) = \bigl[ 1 +   \ell_N   \cos (\hat{\theta} )   \bigr ]^{2/3}; \quad  c_1(0) = 0; \quad   c_{0, \ell_N}(0) = \frac{ \sii  \Omega_0 -  6c_{2, \ell_N}(0)      \Omega_2  }{\cc    \Omega_4}.
\end{align*}If $N$ is sufficiently large, then $c_{2, \ell_N} (0) < 1/6$. By the $\Upsilon$-cone condition, we have $\sii \Omega_0 > \Omega_2$. This implies that $\sii \Omega_0 - 6c_{2, \ell_N}(0) \Omega_2$ is positive. Finally, we have $ c_{0, \ell_N}(0) > 0$, so when $t = 0$, equation (\ref{eq:3.9}) becomes a general inverse $\sigma_k$ equation, which is solvable due to the result of Collins--Székelyhidi \cite{collins2017convergence}. This finishes the proof.
\end{proof}

\vspace{3cm}
\Address


\begin{thebibliography}{99} 
\bibitem{blocki2005uniform}
Z.~Błocki, {\em On uniform estimate in Calabi--Yau theorem}, Science in China Series A: Mathematics,
  vol.~48, no.~1, pp.~244--247, 2005.

\bibitem{blocki2011uniform}
Z.~Błocki, {\em On the uniform estimate in the Calabi--Yau theorem, II}, Science China Mathematics,
  vol.~54, no.~7, pp.~1375--1377, 2011.

\bibitem{caffarelli1985dirichlet}
L.~Caffarelli, L.~Nirenberg, and J.~Spruck, {\em The Dirichlet problem for nonlinear second order elliptic equations, III: Functions of the eigenvalues of the Hessian}, Acta Mathematica, vol.~155, pp.~261--301, 1985.

\bibitem{calabi1954kahler}
E.~Calabi, {\em The space of Kähler metrics}, Proceedings of the International Congress of Mathematicians Amsterdam, pp.~206--207, 1954

\bibitem{calabi1957kahler}
E.~Calabi, {\em On Kähler manifolds with vanishing canonical class, Algebraic geometry and topology. A symposium in honor of S.~Lefschetz}, Princeton Mathematical Series, vol.~12, pp.~78--89, 1957.

\bibitem{chen2021j}
G.~Chen, {\em The $J$-equation and the supercritical deformed Hermitian--Yang--Mills equation}, Inventiones Mathematicae, pp.~1--74, 2021.

\bibitem{chu2021nakai}
J.~Chu, M.-C.~Lee, and R.~Takahashi, {\em A Nakai--Moishezon type criterion for supercritical deformed Hermitian--Yang--Mills equation}, arXiv:2105.10725, 2021.

\bibitem{collins20151}
T.~C.~Collins, A.~Jacob, and S.-T.~Yau, {\em \texorpdfstring{$(1, 1)$}{(1,1)} forms with specified
  Lagrangian phase: a priori estimates and algebraic obstructions}, arXiv:1508.01934, 2015.

\bibitem{collins2020stability}
T.~C.~Collins and Y.~Shi, {\em Stability and the deformed Hermitian--Yang--Mills
  equation}, arXiv:2004.04831, 2020.

\bibitem{collins2017convergence}
T.~C.~Collins and G.~Székelyhidi, {\em Convergence of the $J$-flow on toric manifolds}, Journal of Differential Geometry,
  vol.~107, no.~1, pp.~47--81, 2017.

\bibitem{collins2018moment}
T.~C.~Collins and S.-T.~Yau, {\em Moment maps, nonlinear PDE, and stability in mirror symmetry}, arXiv:1811.04\\824, 2018.

\bibitem{collins2018deformed}
T.~C.~Collins, D.~Xie, and S.-T.~Yau, {\em The deformed Hermitian--Yang--Mills equation 
in geometry and Pphysics}, Geometry and Physics: Volume 1: A Festschrift in 
Honour of Nigel Hitchin, vol.~1, pp.~69, 2018.

\bibitem{datar2021numerical}
V.~Datar and V.~P.~Pingali, {\em A numerical criterion for generalised Monge--Ampère equations on projective manifolds}, Geometric and Functional Analysis,
  vol.~31, no.~4, pp.~767--814, 2021.

\bibitem{dinew2017liouville}
S.~Dinew and S.~Kołodziej, {\em Liouville and Calabi--Yau type theorems for
  complex Hessian equations}, American Journal of Mathematics, vol.~139,
  no.~2, pp.~403--415, 2017.

\bibitem{fang2013convergence}
H.~Fang and M.~Lai, {\em Convergence of general inverse $\sigma_k$-flow on Kähler manifolds}, Transactions of the American Mathematical Society, vol.~365,
  no.~12, pp.~6543--6567, 2013.

\bibitem{fang2011class}
H.~Fang, M.~Lai, and X.~Ma, {\em On a class of fully nonlinear flows in Kähler geometry}, J.~Reine Angew.~Math.~653, 2011.

\bibitem{gilbarg2015elliptic}
D.~Gilbarg and N.~S.~Trudinger, {\em Elliptic partial differential equations of second order}, vol.~224, 2015.

\bibitem{guan2014second}
B.~Guan, {\em Second order estimates and regularity for fully
  nonlinear elliptic equations on Riemannian manifolds}, Duke
  Mathematical Journal, vol.~163, no.~8, pp.~1491--1524, 2014.

\bibitem{hou2010second}
Z.~Hou, X.~Ma, and D.~Wu, {\em A second order estimate for complex Hessian equations on a compact Kähler manifold}, Mathematical Research Letters, vol.~17, no.~3, pp.~547--561, 2010.

\bibitem{jacob2019weak}
A.~Jacob, {\em Weak geodesics for the deformed Hermitian--Yang--Mills equation}, 
arXiv:1906.07128, 2019.

\bibitem{jacob2020deformed}
A.~Jacob and N.~Sheu, {\em The deformed Hermitian--Yang--Mills equation on the
  blowup of \texorpdfstring{$\mathbb{P}^n$}{Pn}}, arXiv:2009.00651, 2020.

\bibitem{jacob2017special}
A.~Jacob and S.-T.~Yau, {\em A special Lagrangian type equation for holomorphic
  line bundles}, Mathematische Annalen, vol.~369, no.~1-2, pp.~869--898,
  2017.

\bibitem{leung2000special}
N.~C.~Leung, S.-T.~Yau, and E.~Zaslow, {\em From special Lagrangian to
  Hermitian--Yang--Mills via Fourier--Mukai transform}, arXiv:math/0005118, 2000.

\bibitem{lejmi2015j}
M.~Lejmi and G.~Székelyhidi, {\em The $J$-flow and stability}, Advances in mathematics, vol.~274, pp.~404--431, 2015.

\bibitem{lin2020}
C.-M.~Lin, {\em Deformed Hermitian--Yang--Mills Equation on compact Hermitian manifolds}, arXiv:math
/2012.00\\487, 2000.

\bibitem{marino2000nonlinear}
M.~Marino, R.~Minasian, G.~Moore, and A.~Strominger, {\em Nonlinear instantons
  from supersymmetric $p$-branes}, Journal of High Energy Physics,
  vol.~2000, no.~01, pp.~005, 2000.

\bibitem{pingali2019deformed}
V.~P.~Pingali, {\em The deformed Hermitian--Yang--Mills equation on three-folds}, 
arXiv:1910.01870, 2019.

\bibitem{siu2012lectures}
Y.-T.~Siu, {\em Lectures on Hermitian--Einstein metrics for stable bundles and Kähler--Einstein metrics: delivered at the German Mathematical Society Seminar in Düsseldorf in June, 1986}, vol.~8, 2012.

\bibitem{schlitzer2021deformed}
E.~Schlitzer and J.~Stoppa, {\em Deformed Hermitian--Yang--Mills connections,
  extended gauge group and scalar curvature}, Journal of the London Mathematical Society, 2021.

\bibitem{song2020nakai}
J.~Song, {\em Nakai--Moishezon criterions for complex Hessian equations}, arXiv:2012.07956, 2020.

\bibitem{szekelyhidi2018fully}
G.~Székelyhidi, {\em Fully non-linear elliptic equations on
  compact Hermitian manifolds}, Journal of Differential Geometry,
  vol.~109, no.~2, pp.~337--378, 2018.

\bibitem{szekelyhidi2017gauduchon}
G.~Székelyhidi, V.~Tosatti, and B.~Weinkove, {\em Gauduchon metrics
  with prescribed volume form}, Acta Mathematica, vol.~219, no.~1,
  pp.~181--211, 2017.

\bibitem{trudinger1995dirichlet}
N.~Trudinger, {\em On the Dirichlet problem for Hessian equations}, Acta Mathematica,
  vol.~175, no.~2, pp.~151--164, 1995.

\bibitem{yau1978ricci}
S.-T.~Yau, {\em On the Ricci curvature of a compact Kähler manifold and the
  complex Monge--Ampère equation, I}, Communications on Pure and
  Applied Mathematics, vol.~31, no.~3, pp.~339--411, 1978.

\end{thebibliography}
\end{document}